\documentclass[11pt, a4paper]{article}
 
\usepackage[english]{babel}
\usepackage{url}
\usepackage{paralist}

\usepackage{amssymb,amsmath,verbatim,graphicx, amsthm}
\swapnumbers
\usepackage{caption}
\usepackage{subcaption}
\usepackage{titlesec, blindtext, color}

\usepackage{afterpage}
\usepackage{appendix}

\parskip\medskipamount
\newtheorem{theorem}{Theorem}[section]

\newtheorem{prop}[theorem]{Proposition}
\newtheorem{definition}[theorem]{Definition}
\newtheorem{lemma}[theorem]{Lemma}
\newtheorem{fact}[theorem]{Fact}
\newtheorem{constr}[theorem]{Construction}
\newtheorem{subcon}[subsubsection]{}
\newtheorem{cor}[theorem]{Corollary}

\newtheorem{main}[theorem]{Main Theorem}
\newtheorem{rem}[theorem]{Remark}

\newtheorem{defi}[theorem]{Definition}
\newtheorem{ex}[theorem]{Example}
\newenvironment{proof*}{\par\noindent\textbf{Proof}\hspace{1em}}{}
\def\<{\langle}
\def\>{\rangle}

\newcommand{\proj}{\mathrm{proj}}
\newcommand{\Aut}{\mathrm{Aut}}

\newcommand{\PG}{\mathsf{PG}}

\newcommand{\cS}{\mathcal{S}}

\newcommand{\cA}{\mathcal{A}}

\newcommand{\Res}{\mathrm{Res}}

\newcommand{\cod}{\mathrm{codim}}

\newcommand{\N}{\mathbb{N}}

\usepackage[bitstream-charter]{mathdesign}
\usepackage[T1]{fontenc}
\usepackage{tikz}
\usepackage{titlesec}
\usepackage{changepage}
\usepackage{lipsum}

\begin{document}
\author{Anneleen De Schepper\thanks{Supported by the Fund for Scientific Research - Flanders (FWO - Vlaanderen)} \and Hendrik Van Maldeghem\thanks{Partly supported by the Fund for Scientific Research - Flanders (FWO - Vlaanderen)}}
\title{Maps related to polar spaces preserving a Weyl distance or an incidence condition}
\date{\footnotesize  Department of Mathematics,\\
Ghent University,\\
Krijgslaan 281-S25,\\
B-9000 Ghent,
BELGIUM\\
 \texttt{Anneleen.DeSchepper@UGent.be}\\  \texttt{Hendrik.VanMaldeghem@UGent.be}\\ $^*$Corresponding author}
\maketitle

\begin{abstract}
Let $\Omega_i$ and $\Omega_j$ be the sets of elements of respective types $i$ and $j$ of a polar space~$\Delta$ of rank at least $3$, viewed as a Tits-building. For any Weyl distance $\delta$ between $\Omega_i$ and $\Omega_j$, we show that $\delta$ is characterised by $i$ and $j$ and two additional numerical parameters $k$ and $\ell$. We consider permutations $\rho$ of $\Omega_i \cup \Omega_j$ that preserve a single Weyl distance $\delta$. Up to a minor technical condition on $\ell$, we prove that, up to trivial cases and two classes of true exceptions, $\rho$ is induced by an automorphism of the Tits-building associated to $\Delta$, which is always a type-preserving automorphism of $\Delta$ (and hence preserving \emph{all} Weyl-distances), unless $\Delta$ is hyperbolic, in which case there are outer automorphisms. For each class of exceptions, we determine a Tits-building $\Delta'$ in which $\Delta$ naturally embeds and is such that $\rho$ is induced by an automorphism of $\Delta'$. At the same time, we prove similar results for permutations preserving a natural incidence condition. These yield combinatorial characterisations of all groups of algebraic origin which are the full automorphism group of some polar space as the automorphism group of many bipartite graphs.  
\end{abstract}

{\footnotesize
}

\section{Introduction}
Let $\Delta$ be a polar space of rank $n$ with $n \geq 3$, with $\mathsf{T}$ its set of types and $\Omega_s$ its set of singular subspaces of type $s$ (the type of a singular subspace is its dimension, except for the maximal singular subspaces of a hyperbolic quadric). The following situation is the central theme of some recent papers: Define some natural (adjacency) relation $\sim$ on $\Omega_s$ and determine the full automorphism group of the corresponding graph $(\Omega_s,\sim)$, hoping for the full automorphism group of $\Delta$. For instance, Liu, Ma \& Wang \cite{Liu-Ma-Wan:12} essentially prove that when $\Delta$ is a finite unitary polar space, $s$ is arbitrary but not maximal, and adjacency is ``being incident with common singular subspaces  of types $s-1$ and $s+1$'', then the automorphism group of $(\Omega,\sim)$ coincides with the full automorphism group of the polar space. Zeng, Chai, Feng \& Ma \cite{Zen-Cha-Fen-Ma:13} prove the same thing for finite symplectic polar spaces. Pankov \cite{gras} shows this for general polar spaces, and points out the only exception, namely the polar space related to the triality quadric, where also trialities and dualities preserve this adjacency relation on the set of lines of the polar space (however, implicitly, this result was known long before, see Section~\ref{grassmann}). M.~Pankov, K.~Prazmovski \& M.~Zynel \cite{Pan-Pra-Zyn:06} show for an arbitrary polar space $\Delta$ and arbitrary $s$ that, when adjacency is ``being incident with a common singular subpace  of type $s-1$'', then the automorphism group of $(\Omega,\sim)$ coincides with the full automorphism group of the polar space (without exception). Huang \& Havlicek \cite{Hua-Hav:08} develop a technique  that can be applied to this problem when the adjacency relation is given by ``opposition'' (see below for the precise definition of this notion). However, their result can not be applied to all polar spaces. Kasikova \& Van Maldeghem \cite{Kas-Mal:13} solve the case of opposition for all polar spaces and all possible types (pointing out several exceptions to the expectation of getting the full automorphism group of the polar space). Huang \cite{Hua:10,Hua:11} shows that for many polar spaces, when $s$ is maximal and adjacency is given by ``intersecting in a singular subspace of type at most some fixed number'', the automorphism groups of the graph and the polar space coincide.  Liu, Pankov \& Wang \cite{pan} treat the case where adjacency is given by ``being incident with a common singular subpace  of type $s-1$ and not with one of type $s+1$'', and also the case where adjacency is defined as ``being contained in a unique maximal singular subpace''.  In the present paper we consider adjacency relations that contain and generalise all previously mentioned relations. Moreover, we consider these relations between singular subspaces of possibly different types, which gives rise to bipartite graphs and yields slightly more general results and more counter examples. We note that the adjacency relations in  \cite{Pan-Pra-Zyn:06,Hua:10,Hua:11} express an intersection property of the singular subspaces in question, while the adjacency relations in \cite{Liu-Ma-Wan:12, Zen-Cha-Fen-Ma:13, Hua-Hav:08, Kas-Mal:13, pan} express a certain \emph{Weyl distance} in the associated Tits-building.

Hence we study permutations of $\Omega_i \cup \Omega_j$, for $i,j \in \mathsf{T}$ (note that in most cases $i,j$ represent dimensions, only when $\Delta$ is hyperbolic there are two types of $(n-1)$-dimensional subspaces; hence, in general, $|i|$ and $|j|$ denote the corresponding dimensions), preserving either \begin{compactenum} \item[$(i)$] a single Weyl distance between elements of $\Omega_i$ and $\Omega_j$ in the Tits-building associated to $\Delta$, \end{compactenum} or \begin{compactenum} \item[$(ii)$] the members of $\Omega_i\times\Omega_j$ which intersect in a subspace of given dimension, \end{compactenum} or \begin{compactenum} \item[$(iii)$] the members of $\Omega_i\times\Omega_j$ whose intersection has dimension at least some given value. \end{compactenum}  In graph-theoretical terms, this amounts to automorphisms of bipartite graphs having $\Omega_i$ and $\Omega_j$ as bipartition classes, where $I \in \Omega_i$ and $J \in \Omega_j$ are adjacent if,  for some $k,\ell \in \mathsf{T}\cup\{-1,n-2\}$ (defining the type of the empty set as $-1$ and including $n-2$ in case $\Delta$ is a hyperbolic polar space, as then $n-2 \notin \mathsf{T}$), the type of $I \cap J$ is $k$ and the type of $I^J$ is $\ell$ in Case $(i)$  (this will be explained below), the type of $I \cap J$ is $k$ in Case $(ii)$, and the type of $I \cap J$ is at least $k$ in Case $(iii)$. With only a single restriction on the parameters, being ``$|\ell| = n-1$ implies $|i|=|j|=n-1$'', we prove that, up to two classes of exceptions and trivial graphs (meaning that the adjacency relation is empty, the graph is complete bipartite, a matching or the bipartite complement of a matching), every automorphism of these graphs is induced by an automorphism of the Tits-building corresponding to $\Delta$, which is just an automorphism of the polar space $\Delta$ if $\Delta$ is not hyperbolic. The mentioned restriction is not expected to give rise to counterexamples, yet it does require a different approach which does not fit in the current paper (and as such we leave this case for future work).   If $i\neq j$, then by considering all possible Weyl distances between elements of type $i$ and $j$, Case $(i)$ provides a partition of the complete bipartite graph $\Omega_i\times \Omega_j$ such that the automorphism group of each class of the partition coincides with the full automorphism group of the polar space. This is a nice and unexpected, though theoretical combinatorial property of these groups. 

In \cite{proj}, we studied a similar problem for a projective space $\mathbb{P}$, which gave rise to two types of graphs. Both have $\mathbb{P}_i$ and $\mathbb{P}_j$ as bipartition classes (with similar notation as above and $0 \leq i,j \leq \dim(\mathbb{P})$). In the first case (resp.\ the second case),  $I\in \mathbb{P}_i$ and $J \in \mathbb{P}_j$ are adjacent if $\dim(I \cap J)=k$ (resp.\ $\dim(I \cap J) \geq k$)  for a fixed $k$ with $k\geq -1$. The main result of \cite{proj} states that if these graphs are nontrivial, then all their automorphisms are induced by automorphisms of $\mathbb{P}$  (possibly including a duality). Both cases fit in an incidence geometric setting, since in the first case (resp.\ the second case), $I$ and $J$ are adjacent whenever there is exactly one $k$-space (resp.\ at least one $k$-space) incident with both of them. However, the first case in fact also fits in a metric setting, since it corresponds with the preservation of a single Weyl distance in the Tits-building corresponding to $\mathbb{P}$.  As a projective space is a particular type of Tits-building and the Weyl distance is defined for Tits-buildings in general, it is natural to ask whether this also holds for other types of (spherical) Tits-buildings. This paper answers this question for Tits-buildings associated to polar spaces. Yet, we are also able to treat analogs of the incidence-geometric case at the same time. Precise definitions and statements will be given in Section~\ref{statements}. The case of the preservation of a Weyl distance yields a rather general Beckman-Quarles \cite{Bec-Qua:53} type result for the vertices of spherical Tits-buildings of classical type. 

Since the analogous problem in the rank 2 spherical case is completely solved by Govaert \& Van Maldeghem \cite{Gov}, only the exceptional Tits-buildings of types $\mathsf{F_4,E_6,E_7,E_8}$ remain. These yield a finite number of possible Weyl distances. Note that a similar question for \emph{chambers} of any Tits-building has been answered by Abramenko \& Van Maldeghem \cite{Ab2}.

\section{Preliminaries}
To avoid ambiguity, we give definitions of the concepts that we will frequently use.
\subsection{Polar spaces and related notions} A \emph{ polar space} $\Delta = (X,\Omega)$ of rank $n$, $n \geq 2$, consists of a set of points $X$ and a family $\Omega$ of subsets of $X$, satisfying the following axioms.

\begin{itemize}
\item[(PS$1$)] Each element $U$ of $\Omega$ together with all elements of $\Omega$ contained in $U$ is a projective space of dimension at most $n-1$ (this dimension will be called the \emph{dimension} of $U$ and is denoted by $\dim(U)$). A projective space of dimension $-1$ is just the empty set, a projective space of dimension 0 is a point and a projective space of dimension 1 is a set of at least three points with no further structure. 
\item[(PS$2$)] The intersection of any number of elements of $\Omega$ is again contained in $\Omega$.
\item[(PS$3$)] For $U \in \Omega$ with $\dim(U) = n-1$ and $p \in X\setminus U$, the union of all elements of $\Omega$ of dimension 1 containing $p$ and intersecting $U$ nontrivially is an element of $\Omega$ of dimension $n-1$ which intersects $U$ in a hyperplane.
\item[(PS$4$)] There are two disjoint elements of $\Omega$ of dimension $n-1$.
\end{itemize} \vspace{-1em}
A set $X$ of cardinality at least two, together with $\Omega=X\cup\{\emptyset\}$ is considered to be a polar space of rank 1. Henceforth, $\Delta$ denotes a polar space of rank $n$ with $n\geq 2$.

\textbf{Collinearity and opposition} $-$  An element of $\Omega$ of dimension $n-1$ is called a \emph{maximal singular subspace} (MSS for short) and an element of $\Omega$ of dimension 1 is called a \emph{line}. Let $x$ and $y$ be two distinct points. If they are on a common line, they are called \emph{collinear} and we write $x \perp y$, if not, they are called \emph{opposite}. The set of points equal or collinear with $x$ is denoted by $x^{\perp}$. A \emph{subspace} $S$ of $\Delta$ is a subset of $X$ such that the lines joining any two collinear points of $S$ are contained in $S$. Moreover, if $S$ contains no pair of opposite points, the subspace is called \emph{singular}. The elements of $\Omega$ are precisely the singular subspaces of $\Delta$. If $U$ and $V$ are singular subspaces with $U \subseteq V$, then the \emph{codimension} $\cod_{V} U$ of $U$ in $V$ is defined as $\dim(V) - \dim(U) -1$.  

For a singular subspace $U$, we define $U^{\perp}$ as $\bigcap_{x \in U} x^{\perp}$. For any singular subspace $V$, we say that $U$ and $V$ are \emph{collinear} if $V \subseteq U^\perp$. If they are collinear but disjoint, we write $U \perp V$. Let $T$ be a set of pairwise collinear singular subspaces. We denote by $\<T\>$  the smallest singular subspace containing all members of $T$, and we also say that the members of $T$ generate $\<T\>$ or that $\<T\>$ is spanned by the members of $T$. If $T$ consists of two distinct collinear points $x,y$, we denote the unique line joining these points by $x\!y$. The \emph{projection} $\proj_V(U)$ of a singular subspace $U$ on a singular subspace $V$ is $V \cap U^\perp$ and the subspace spanned by $U$ and $\proj_V(U)$ is denoted by $U^V$ (note that  $\dim(U^V)=\dim(V^U)$). If $\proj_V(U)$ or $\proj_U(V)$ is empty, we say that $U$ and $V$ are \emph{semi-opposite}.  Now let $U$ and $V$ be semi-opposite singular subspaces. If $\dim(U) = \dim(V)$, then both $\proj_V(U)$ and $\proj_U(V)$ are empty and $U$ and $V$ are just called \emph{opposite}; in case $\dim(U)  < \dim(V)$, the projection $\proj_U(V)$ is empty whereas $\proj_V(U)$ is not, more precisely, it has dimension $\dim(V)- \dim(U) -1$.

\par\bigskip
\textbf{Embeddable and non-embeddable polar spaces} $-$ A polar space $\Delta=(X,\Omega)$ is called \emph{embeddable} when $X$ is a (spanning) subset of the point set of a projective space and the elements of $\Omega$ are subspaces of that projective space. 

According to the classification of polar spaces of rank at least $3$ by Jacques Tits, there are only two classes which are not embeddable. Both occur when the rank equals $3$ and are denoted by $\Delta(\mathbb{L})$ and $\Delta(\mathbb{O})$, respectively. The first one has diagram of type $\mathsf{D}_3$, more precisely, it is a line Grassmannian of a projective space of dimension $3$ over a non-commutative skew field $\mathbb{L}$ and hence it has projective planes over both $\mathbb{L}$ and its opposite field, $\mathbb{L}^{\leftrightarrow}$; the second has diagram of type $\mathsf{C}_3$ and has planes over an octonion Cayley-Dickson division algebra $\mathbb{O}$  (hence these planes are non-Desarguesian). We now turn to the embeddable polar spaces.

An embeddable polar space does not necessarily admit a unique representation in projective space. However, it will suffice for us to have one specific representation, namely, the one arising from a pseudo-quadratic form. The following is based on Chapter 10 of \cite{Bru-Tit:72}, slightly modified by Tits in \cite{Tit:95}. Let $\Delta$ be an embeddable polar space of rank $n$ at least $3$. Then there are a skew field $\mathbb{L}$, a right vector space $V$ over $\mathbb{L}$ (of possibly infinite dimension), an isomorphism $\sigma$ of order at most 2 between $\mathbb{L}$ and its dual $\mathbb{L}^{\leftrightarrow}$, and a $(\sigma, \mathrm{id})$-linear form $g:V\times V\rightarrow \mathbb{L}$ (i.e., $g$ is $\sigma$-linear in the first argument and linear in the second argument) such that $\Delta$ can be described as follows. Put $\mathbb{L}_\sigma^\epsilon=  \{x-\epsilon x^\sigma \mid x \in \mathbb{L}\}$ for $\epsilon\in\{+1,-1\}$, and consider it as an additive group. Let $f: V \times V \rightarrow \mathbb{L}$ be the $(\sigma, \mathrm{id})$-linear mapping defined by $f(u,v)=g(u,v)+\epsilon g(v,u)^\sigma$, and define the \emph{pseudo-quadratic form} $\mathfrak{q}$ as 
\[\mathfrak{q}: V \rightarrow \mathbb{L}/\mathbb{L}^\epsilon_\sigma: v \mapsto g(v,v) + \mathbb{L}_\sigma^\epsilon,\] where $\mathbb{L}/\mathbb{L}^\epsilon_\sigma$ is considered as a quotient of additive groups.
We must assume that $\mathfrak{q}$ is anisotropic over the radical $\mbox{Rad}(f)=\{v \in V: f(v,w)=0,\forall w \in V\}$ of $f$, i.e., for $v\in \mbox{Rad}(f)$, $\mathfrak{q}(v)=0$ (this is the zero of the additive group $\mathbb{L}/\mathbb{L}^\epsilon_\sigma$) if and only if $v=\vec{o}$. Then the point set $X$ of $\Delta$ consists precisely of the points of the projective space $\PG(V)$ represented by vectors $v$ which vanish under $\mathfrak{q}$, i.e., $\mathfrak{q}(v)=0$. Two points of $\Delta$, say corresponding with the 1-spaces generated by respective vectors $u,v\in V$, are collinear precisely if $f(u,v)=0$.

In the above, we can always assume that $\epsilon=+1$ if $\sigma$ is nontrivial. If $\sigma$ is trivial, then $\mathbb{L}$ is commutative and $f$ is either symmetric ($\epsilon=+1$) or alternating ($\epsilon=-1$). 

Depending on $g$, $\sigma$ and $\epsilon$, we get different kinds of polar spaces, on which we will now comment. 
First note though that $g$ is not uniquely determined by $\mathfrak{q}$. In spite of this, the pseudo-quadratic form $\mathfrak{q}$, if nontrivial, does determine the form $f(u,v) = g(u, v) + g(v, u)^\sigma \epsilon$ uniquely. 

We start assuming that  $\mathsf{char}\, \mathbb{L} \neq 2$, in which case the polar spaces described below correspond to non-degenerate alternating forms, bilinear forms and Hermitian forms, respectively.
\begin{itemize} 
\item \textit{Every point of $\PG(V)$ is a point of $\Delta$.} In this case $\sigma=1$, $\epsilon=-1$ and $f$ is alternating. Then $V=2n$ for some $n$ and we can choose a basis $\{e_{-n},...,e_{-1},e_1,...,e_n\}$ for $V$ such that, for $x^{(\prime)}=\sum_{i=-n, i\neq 0}^n e_ix^{(\prime)}_i$ with $x_{-n},...,x_n \in \mathbb{L}$, 
\[f(x,x')=x_{-n}x'_n-x_nx'_{-n} +x_{-n+1}x'_{n-1}-x_{n-1}x'_{-n+1} + \cdots x_{-1}x'_1 - x_1x'_{-1}.\]
These polar spaces are called \textbf{symplectic}. They have the property that every line $L$ of $\PG(V)$ is either a line of $\Delta$ or a full \emph{hyperbolic line} (see later on). 

\item \textit{Not all points of $\PG(V)$ are points of $\Delta$.} Here, as alluded to above, we may always assume $\epsilon=1$. In this case, there is a subspace $V_0$ of $V$ of (vectorial) codimension $2n$ and an \emph{anisotropic} pseudo-quadratic form $\mathfrak{q}_0: V_0 \rightarrow \mathbb{L}_\sigma^\epsilon$  (meaning that $\mathfrak{q}_0(v)=0$ if and only if $v=\vec{o}$, for all $v\in V_0$) and a basis $\{e_{-n},...,e_{-1},e_1,...,e_n\}$ of a subspace complementary to $V_0$ in $V$ such that for any vector $v=\sum_{i=-n, i\neq 0}^n (e_ix_i) +v_0$ with $x_{-n},...,x_n \in \mathbb{L}$ and $v_0 \in V_0$ we have
\[\mathfrak{q}(v)=x_{-n}^\sigma x_{-n} + x_{-n+1}^\sigma x_{n-1} + \cdots x_{-1}^\sigma x_1 + \mathfrak{q}_0(v_0)\]

We now distinguish between $\sigma$ being the identity, and $\sigma$ not being the identity. Let $x,y$ be any pair of non-collinear points of $\Delta$. Let $L$ be the line in $\PG(V)$ joining $x$ and $y$.
\begin{itemize} 
\item[$\circ$] If $\sigma$ is the identity, then \textit{$L$ always intersects $\Delta$ precisely in  $\{x,y\}$.} These polar spaces are called \textbf{orthogonal} (and sometimes also \emph{strictly orthogonal} for consistency with the case of characteristic 2). In particular, if $V_0=\{0\}$, then $\Delta$ is \textit{hyperbolic}; if $\dim(V_0)=1$ then $\Delta$ is \textit{parabolic}. Note that $\dim(V_0)$ can be arbitrary, even every infinite cardinal.
\item[$\circ$] If $\sigma$ is nontrivial, then \textit{$L$ always intersects $\Delta$ in at least 3 points.} Then $\Delta$ is called \textbf{unitary} or \textbf{Hermitian}. Note that $\mathbb{L}$ is not necessarily commutative here, as opposed to the previous cases.\end{itemize}
\end{itemize}
In both cases one sees that $n-1$ is the maximum dimension of a subspace of $\PG(V)$ entirely contained in the point set $X$ of $\Delta$.

If $\mathsf{char}\, \mathbb{L} =2$, the situation is richer. 
\begin{itemize}

\item If $\mathbb{L}$ is a perfect field, then a parabolic polar space (similarly defined as above for characteristic different from 2, in particular we assume $\sigma$ trivial) is isomorphic to a symplectic polar space (of the same rank and over $\mathbb{L}$). Consequently, the parabolic polar space can now be  embedded in $\PG(2n-1,\mathbb{L})$ as the \textbf{symplectic} one, and we will consider this as its standard embedding.
\item If $\mathbb{L}$ is an imperfect field, $\sigma$ is trivial, then we consider as standard embedding the embedding of the polar space induced in $\PG(V/\mbox{Rad}(f))$. If, and only if, Rad$(f)$ is nontrivial, then a line of $\PG(V/\mbox{Rad}(f))$ intersecting the polar space in at least two points, intersects it in at least three points. If Rad$(f)$ is trivial,  we say that the polar space is \textbf{strictly orthogonal}; otherwise \textbf{mixed}. 
\item If $\sigma$ is not trivial, it could happen that the corresponding polar space can also be obtained as the zeros of the diagonal of a non-degenerate Hermitian form (and this always happens if $\mathbb{L}$ is commutative), but if $\mathbb{L}$ is not commutative, then this is not necessarily true. In any case, we will refer to a polar space from a pseudo-quadratic form with $\sigma$ nontrivial as a \textbf{Hermitian} polar space. Again, we consider as standard embedding the embedding of the polar space induced in $\PG(V/\mbox{Rad}(f))$. Independently of the dimension of $\mbox{Rad}(f)$, every line intersecting the polar space in at least two points, intersects it in at least three points. \end{itemize}
\par\bigskip

\textbf{Residues of $\Delta$} $-$ Let $K$ be a singular subspace of dimension $k$ with $k \leq n-2$ and put $X_K=\{ U \in \Omega \; | \; K \subset U \text{ and } \dim(U) =k+1\}.$ If $M$ is an element of $\Omega$ containing $K$,  we let $M/K$ represent the elements of $X_K$ contained in $M$. We then define $\Omega_K$ as $\{M/K \mid K\subseteq M \in \Omega\}$. The resulting structure $\Res_{\Delta}(K)=(X_K,\Omega_K)$,  i.e., the \emph{residue}, is a polar space of rank $n-k-1$ of the same ``kind'' as $\Delta$, e.g.\ the residue of a parabolic polar space is parabolic too, and likewise for hyperbolic, unitary, mixed and so on. As such, we extend this terminology to rank 2 and rank 1 residues. An element $M/K \in \Omega_K$ has dimension $\dim(M) - k -1$ and will often be identified with $M$. If $\dim(K)=n-2$, then $\Res_{\Delta}(K)$ has rank $1$. This residue contains at least $2$ points and it contains precisely $2$ if and only if $\Delta$ is hyperbolic. 
\par\bigskip
\textbf{The Tits-building associated to $\Delta$} $-$ Denote by $\Delta^b$ the Tits-building associated to $\Delta$. Note that, if $\Delta$ is hyperbolic, $\Delta^b$ is in fact the Tits-building associated to the oriflamme complex of $\Delta$. This is the geometry having as elements of type $i$, with $i \leq n-3$, the elements of dimension $i$ of $\Delta$, and as elements of types $(n-1)'$ and $(n-1)''$ the elements of $\Delta$ of dimension $n-1$, hereby distinguishing between the two natural families of MSS. Incidence between elements of the latter two types is given by intersecting in an $(n-2)$-space of $\Delta$, incidence between all other pairs of elements is given by incidence in $\Delta$.  We define the type set $\mathsf{T}$ of $\Delta$ in this case as $\{0,...,n-3,(n-1)',(n-1)''\}$; in case $\Delta$ is not hyperbolic, $\mathsf{T}$ is just $\{0,...,n-1\}$. For $t\in \mathsf{T}$, we denote by $|t|$ the corresponding dimension if confusion is possible.
The type of a flag of elements is then the set of types of these elements.  If $\Delta$ is hyperbolic however, the type of a flag of type $\{(n-1)',(n-1)''\}$ will conveniently be denoted by $n-2$ sometimes, as this is the dimension of the corresponding subspace. Furthermore, to the empty subspace we \emph{assign} the type $-1$, as this is its projective dimension.

\par\bigskip

\textbf{Automorphisms of $\Delta$ and $\Delta^b$} $-$  We denote by $\Aut(\Delta)$ the group of all automorphisms of the polar space $\Delta$, i.e., all permutations of the point set of $\Delta$ preserving collinearity and opposition of points. Further, we denote by $\Aut(\Delta^b)$ the group of automorphisms of the building $\Delta^b$, i.e., all permutations of the elements of the building preserving incidence and non-incidence. Finally, we denote by $\Aut^o(\Delta^b)$ the group of type preserving automorphisms of $\Delta^b$.  However, an automorphism $\rho$ of the Tits-building $\Delta^b$ associated to $\Delta$ is always type-preserving (recall that we assume that the rank $n$ of $\Delta$ is at least 3), unless possibly if $\Delta$ is hyperbolic as then $\Delta^b$ allows \emph{dualities}, or even \emph{trialities} if $n=4$. So assume that $\Delta$ is hyperbolic. A duality is an automorphism of $\Delta^b$ preserving all types but the maximal ones, which are interchanged. If $n=4$, a triality is an automorphism of $\Delta^b$ only preserving type $1$ and cyclically permuting the types $0,3',3''$. The composition of a duality and a triality of $\Delta^b$ yields an automorphism of $\Delta^b$ preserving types $1$ and $t$ for some $t\in\{3',3''\}$ while interchanging types $0$ and~$t'$. We call this automorphism a \emph{$t$-duality}. Analogously, we sometimes also speak of a $0$-duality.
\par\bigskip

\textbf{Hyperbolic subspaces} $-$ Let $U$ and $V$ be opposite $t$-spaces with $t \in \mathsf{T}$ non-maximal. We define the \emph{double perp} $\{U,V\}^{\perp\!\!\!\perp}$ of $U$ and $V$ as the set of points collinear with  $U^\perp \cap V^\perp$. If $U \cup V \subsetneq \{U,V\}^{\perp\!\!\!\perp}$, this double perp induces a polar space $\Delta' \subseteq \Delta$ of rank $t+1$, which is called a \emph{hyperbolic $(2t+1)$-space}. A hyperbolic $1$-space is just called a \emph{hyperbolic line} and hence has at least three points. In the standard embedding of $\Delta$ in a projective space $\PG(V)$, we obtain $\Delta'$ by intersecting $\Delta$ with the $(2t+1)$-space of $\PG(V)$ generated by $U$ and $V$. This way it is easily seen that each point in $\Delta$ which is collinear with two opposite $t$-spaces of $\{U,V\}^{\perp\!\!\!\perp}$ is collinear with all elements of $\{U,V\}^{\perp\!\!\!\perp}$, though this property also holds when $\Delta$ is not embeddable.

If each point $p$ collinear to $U$ and $V$ should also be collinear with some point $q$, then it follows immediately that $q$ belongs to $\{U,V\}^{\perp\!\!\!\perp}$, since $q \in p^\perp$ for all $p \in \{U,V\}^\perp$. This property will often be used.

If $t=0$, two opposite points determine a hyperbolic line unless $\Delta$ is  a strictly orthogonal polar space. In case $\Delta$ is Moufang (which it certainly is if $n \geq 3$), the existence of one hyperbolic line is equivalent with all pairs of opposite points contained in a hyperbolic line.  If $t =1$, a hyperbolic $3$-space is a hyperbolic quadrangle (that is, a hyperbolic polar space of rank 2) precisely if $\Delta$ is orthogonal.  This is the only kind of polar spaces in which a maximal set $R$ of pairwise opposite lines of a hyperbolic $3$-space has the property that each line intersecting two of them intersects all of them ($R$ is a regulus of a hyperbolic quadrangle then). For $\Delta \in \{\Delta(\mathbb{O}),\Delta(\mathbb{L})\}$, the hyperbolic $3$-space $\{U,V\}^{\perp\!\!\!\perp}$ is given by $\{x,y\}^\perp$, for any two points $x,y \in \{U,V\}^\perp$.




\subsection{Weyl distance between two subspaces of a polar space}
We recall the definition of the Weyl distance but assume the reader to be familiar with its basic properties. For more details, see for example Sections~3.5 and~4.8 of \cite{Abr-Bro:08}, or Section 11 of \cite{diag}. The Weyl distance is defined in any Tits-building, in particular in $\Delta$. First assume that $\Delta$ is not of hyperbolic type. 

Let $[-n,n]_0$ denote the set of nonzero integers not smaller than $-n$ and not larger than $n$, for $n$ any natural number. Let $\Xi$ be the graph of a cross-polytope with $2n$ vertices (where $n$ is now indeed the rank of $\Delta$), i.e., $\Xi$ consists of the vertices $\xi_{-n},\xi_{-n+1},\ldots,\xi_{-1},\xi_1,\xi_2,\ldots,\xi_n$ and $\xi_i$ is adjacent with $\xi_j$, $i,j\in [-n,n]_0$, if and only if $i\neq-j$. The automorphism group of $\Xi$ is a Coxeter group $W$ of type $\mathsf{B}_n$, and we choose the following canonical set $S$ of generators. The automorphism $s_i$, $i\in\{1,2,\ldots,n-1\}$ is given by the involution interchanging $\xi_i$ with $\xi_{i+1}$ and $\xi_{-i}$ with $\xi_{-i-1}$. The automorphism $s_n$ is given by interchanging $\xi_{-n}$ with $\xi_n$. The group $W$ is generated by $s_1,\ldots,s_n$ and by no proper subset of it, and we have the relations $(s_is_j)^{m_{ij}}=1$, where $m_{ij}$ is really the order of the product $s_is_j$, given by $$m_{ij}=\left\{\begin{array}{lr} 1 \hspace{.5cm}& \mbox{if }i=j,\\2 & \mbox{if } |i-j|>1, \\ 3 & \mbox{if }\min\{i,j\}+1=\max\{i,j\}<n, \\ 4 & \mbox{if }\{i,j\}=\{n,n-1\}.\end{array}\right.$$

A \emph{chamber} of $\Xi$ is a maximal set of nested cliques. The standard chamber is the nested chain $C_0=\{\{\xi_1\},\{\xi_1,\xi_2\},\ldots,\{\xi_1,\xi_2,\ldots,\xi_n\}\}$. One easily verifies that $W$ acts sharply transitively on the set of all chambers of $\Xi$ (there are $|W|=2^nn!$ chambers in $\Xi$). Hence, given any chamber $C$, there exists a unique $w\in W$ such that $C=C_0^w$. We say that $w$ is the Weyl distance from $C_0$ to $C$, in symbols $\delta(C_0,C)=w$. In general, for two chambers $C,C'$, we define $\delta(C,C')=\delta(C_0,C)^{-1}\delta(C_0,C')$. The numerical distance $d(C,C')\in\mathbb{N}\cup\{0\}$ is the minimal length of any expression of $\delta(C,C')$ in terms of the generators in $S$ (that number is also called the length of the corresponding element of $W$).

It is well known that $W$, just like each finite Coxeter group, contains a unique element $w_0$ of maximal length. In our case, the maximal length is $n^2$ and $w_0$ is given by $$w_0=(s_ns_{n-1}\cdots s_1)\cdot(s_ns_{n-1}\cdots s_2)\cdots(s_ ns_{n-1})\cdot(s_n)\cdot(s_{n-1}s_{n-2}\cdots s_1)\cdot(s_{n-1}s_{n-2}\cdots s_2)\cdots(s_{n-1}s_{n-2})\cdot (s_{n-1}).$$

An \emph{apartment} $\cA$ of $\Delta$ is a set of all singular subspaces spanned by a subset of the set $\mathcal{S}=\{x_1,\ldots,x_n$, $y_1,\ldots,y_n\}$ of $2n$ points for which $y_i$ is the unique point of $\mathcal{S}$ opposite $x_i$ and $x_i$ the unique point of $\mathcal{S}$ opposite $y_i$, for $1 \leq i \leq n$. The set $\cS$ is called a \emph{frame}. A \emph{chamber} $C$ of $\Delta$ is a maximal chain of nested nonempty singular subspaces. A chamber $C$ is contained in the apartment $\cA$ if each of the singular subspaces of $C$ is contained in $\cA$, i.e., each member of $C$ is generated by a subset of the point set $\mathcal{S}$. By Theorem~7.4 of \cite{Tits}, for every pair of chambers $C,C'$ there exists an apartment containing both $C$ and $C'$. The frame $\mathcal{S}=\{x_1,\ldots,x_n,y_1,\ldots,y_n\}$ of that apartment can be numbered so that $C$ contains the singular subspace spanned by $x_1,\ldots,x_i$, for every $i\in\{1,2,\ldots,n\}$. We can then attach to $C'$ a nested sequence of $n$ subsets of $\mathcal{S}$ such that each subset generates a singular subspace of $C'$. The bijection $x_i\mapsto\xi_i$, $y_i\mapsto \xi_{-i}$, $i\in\{1,2\ldots,n\}$, identifies this sequence with a chamber $C_1$ of $\Xi$. A similar identification maps $C$ to the standard chamber $C_0$ of $\Xi$. The \emph{Weyl distance $\delta(C,C')$} is now by definition equal to $\delta(C_0,C_1)$. It is independent of the choice of the apartment containing $C$ and $C'$. If $\delta(C,C')=w_0$, then we say that $C$ and $C'$ are \emph{opposite}. 


This Weyl distance can also be defined in a natural way on pairs of singular subspaces of $\Delta$ as follows. 
Let $U$ and $W$ be two singular subspaces of $\Delta$. Let $D$ be the set of Weyl distances from a chamber of $\Delta$ containing $U$ to a chamber of $\Delta$ containing $W$. Then one shows (Proposition~4.88 in \cite{Abr-Bro:08}) that $D$ contains a unique element $w$ of minimal length. We set $w=\delta(U,W)$. 


Now suppose that $\Delta$ is hyperbolic and of rank $n$. Then the building associated with $\Delta$ is identified with the oriflamme complex rather than with $\Delta$ itself. We still consider the cross polytope graph $\Xi$ as above, but we define the chambers in a different way. A chamber is now a nested set of $n-2$ cliques of size at most $n-2$, together with two cliques of size $n$ intersecting in a clique of size $n-1$ which contains each clique of the nested set of $n-2$ cliques. Note that the set of maximal cliques of $\Xi$ falls naturally into two classes in such a way that the size of the intersection of two elements (not) belonging to the same class has (does not have) the same parity as $n$.  Every chamber contains a maximal clique of each class. The Coxeter group $W'$ is now defined as the group of automorphisms of $\Xi$ preserving the two classes of maximal cliques. It is a Coxeter group of type $\mathsf{D}_n$. It is generated by the same elements $s_1,s_2,\ldots,s_{n-1}$ and a new element $s_n'$ interchanging $\xi_n$ with $\xi_{-n+1}$ and $\xi_{n-1}$ with $\xi_{-n}$ (and then $s_n'=s_ns_{n-1}s_n$).  Writing for a moment $s_i$ as $s_i'$ for convenience ($i\in\{1,2,\ldots,n-1\}$),  we again have relations $(s'_is'_j)^{m'_{ij}}=1$, where $m'_{ij}$ is really the order of the product $s'_is'_j$, given by $$m'_{ij}=\left\{\begin{array}{ll} 1 \hspace{.5cm}& \mbox{if }i=j,\\2 & \mbox{if } |i-j|>1\mbox{ with }\max\{i,j\}<n, \mbox{ and if }\{i,j\}=\{n-1,n\}, \\ 3 & \mbox{if }\min\{i,j\}+1=\max\{i,j\}<n, \mbox{ and if }\{i,j\}=\{n-2,n\}.\end{array}\right.$$

Again, there is a unique longest element $w_0'$ in $W'$, and it has length $n^2-n$. It reads $$w_0'=s_1s_2\ldots s_{n-1}s_n's_{n-2}\ldots s_1.s_2s_3\ldots s_{n-1}s_n's_{n-2}\ldots s_2.\cdots.s_{n-3}s_{n-2}s_{n-1}s_n's_{n-2}s_{n-3}.s_{n-2}s_{n-1}s_n's_{n-2}.s_{n-1}s_n'.$$ 

The standard chamber is now $$C_0=\{\{\xi_1\},\{\xi_1,\xi_2\},\cdots,\{\xi_1,\xi_2,\ldots,\xi_{n-2}\},\{\xi_1,\xi_2,\ldots,\xi_n\},\{\xi_1,\xi_2,\dots,\xi_{n-1},\xi_{-n}\}\}.$$ The Weyl distance from $C_0$ to any other chamber of $\Xi$ is defined as above, now using the Coxeter group $W'$ and the set $\mathcal{S}'=\{s_1,s_2,\ldots,s_{n-1},s_n'\}$ of generators. Similarly as above, also the Weyl distance between two arbitrary chambers is defined. Also, if we define a chamber $C$ in $\Delta$ as the union $C_{\leq n-2}\cup C_{n-1,n}$ of a set $C_{\leq n-2}$ of $n-2$ nested singular subspaces of dimension $0$ up to $n-3$ with a pair $C_{n-1,n}$ of maximal singular subspaces intersecting in a singular subspace $U$ of dimension $n-2$ which contains each element of $C_{\leq n-2}$, then we can define the Weyl distance from one chamber to another in the same way as for type $\mathsf{B}_n$ above. Similarly, one also defines the Weyl distance between singular subspaces in $\Delta$. 

Before we prove the next lemma, we note that, if $U$ is an element of type $i$ of $\Delta^b$, where $\Delta$ is not of hyperbolic type, and $W$ is an element of type $j$ not incident with $U$, $i,j\in\{0,1,\ldots,n-1\}$, then any shortest expression of $\delta(U,W)$ in terms of the generators in $S$ starts with $s_{i+1}$ and ends with $s_{j+1}$. Indeed, it follows from the proof of Proposition 3.87 in \cite{Abr-Bro:08} that $\delta(U,W)$ is the shortest element of the double coset $W_{i+1}\delta(U,W)W_{j+1}$, where $W_t$ denotes the (Weyl) subgroup generated by $S\setminus\{s_t\}$, $t\in\{1,2,\ldots, n\}$; hence if the shortest expression of $\delta(U,W)$ would start with $s_t$, $t\neq i+1$, then we can absorb it in $W_{i+1}$ and get a shorter representative of the double coset, a contradiction (similarly if the shortest expression of $\delta(U,W)$ would not end with $s_{j+1}$). Hence the Weyl distance between two distinct elements reveals the type of the elements. Similarly for the case that $\Delta$ is hyperbolic (but then, if $U$ has type $(n-1)'$ or $(n-1)''$, then $\delta(U,W)$ starts with $s_{n-1}$ or $s_n'$, respectively, and similar for $W$). 

\begin{lemma}\label{weyl} Let $I$, $I'$, $J$, $J'$ be four singular subspaces of $\Delta$ conforming to a type of the building and such that neither $\{I,J\} $ nor $\{I',J'\}$ are flags. Then $\delta(I,J)=\delta(I',J')$ if and only if $t(I)=t(I')$, $t(J)=t(J')$, $t(I\cap J)=t(I'\cap J')$ and $t(I^J)=t(I'^{J'})$. 
\end{lemma}
\begin{proof}
First suppose that $\delta(I,J)=\delta(I',J')$. By the definition of Weyl distance, we find chambers $c,c',d$ and $d'$ containing $I, I',J$ and $J'$, respectively, such that $\delta(c,d)=\delta(c',d')$. As $\Aut^o(\Delta^b)$ acts strongly transitively on $\Delta^b$, it acts transitively  on the family of pairs of chambers at the same Weyl distance (see e.g.~Proposition~7.11 in \cite{Abr-Bro:08}). Hence there is a type-preserving automorphism $g$ of $\Delta^b$ mapping $(c,d)$ on $(c',d')$. Since the Weyl distance $\delta(I,J)=\delta(I',J')$ determines the types of $I,I',J,J'$, we deduce that the types of $I$ and $I'$ are the same, and also the types of $J$ and $J'$ coincide. This means that $(I,J)$ is mapped by $g$ onto $(I',J')$ (because each chamber contains a unique element of each type), and moreover, $I \cap J$ is mapped on $I' \cap J'$ and $\proj_J(I)$ on $\proj_{J'}(I')$. As $g$ is type preserving, the assertion follows.  

To show the converse, it suffices to find an element of $\Aut^o(\Delta^b)$ that sends $(I,J)$ to $(I',J')$, since such a map preserves the Weyl distance. Without loss, $J=J'$, for there is a type preserving automorphism mapping $J$ onto $J'$ and this of course preserves the respective types of intersection and projection. We may also assume that $I,I'$ and $J$ are in a common apartment $\cA$ determined by the frame $\{x_1,\ldots, y_n\}$ (with previous notation). Indeed, suppose that $\Sigma$ is an apartment containing $I$ and $J$ and $\Sigma'$ an apartment containing $I'$ and $J$.  Then by the strong transitivity of $\Aut^o(\Delta^b)$, there is a type preserving automorphism mapping $\Sigma$ on $\Sigma'$ while fixing $J$. We now look for a type preserving automorphism in $\cA$ that fixes $J$ and maps $I$ on $I'$. Let $Q$ (resp.\ $Q'$) be a subspace of $I$ (resp.\ $I'$) complementary to $\proj_I(J)$.  The subspaces $I,I',J$ and their subspaces correspond to subsets of $\{x_1,...,y_n\}$. Applying the bijection $x_i\mapsto\xi_i$, $y_i\mapsto \xi_{-i}$, $i\in\{1,2\ldots,n\}$,  the assertion is now easily checked in the graph $\Xi$ (for both cases of type $\mathsf{B}_n$ and $\mathsf{D}_n$). 
\end{proof}

\begin{rem}
\em 
\begin{itemize}
\item The condition that both $\{I,J\}$ and $\{I',J'\}$ are not flags is necessary but harmless. Indeed, it is necessary because if $\{I,J\}$ and $\{I',J'\}$ are flags then $\delta(I,J)=\delta(I',J')=\mathrm{id}$, regardless of the types of $I,I',J,J'$. It is harmless because, in our case we always have $\mathsf{t}(I)=\mathsf{t}(I')$ and $\mathsf{t}(J)=\mathsf{t}(J')$ and, given this, we also have $\delta(I,J)=\delta(I',J')$ if and only if $t(I\cap J)=t(I'\cap J')$ and $t(I^J)=t(I'^{J'})$. 

\item The previous lemma also holds if $I, I', J, J'$ are flags (with an obvious definition of Weyl distance). However, we would only need this when $\Delta$ is hyperbolic of rank $n$, when dealing with singular subspaces of dimension $n-2$, i.e., flags with type set $\{(n-1)', (n-1)''\}$. Yet, in that situation we will consider $\Delta$ as a non-thick building of type $\mathsf{B}_n$ and then we can apply the previous lemma anyway. So we do not need the flag version of the lemma after all.

\end{itemize}

\end{rem}

\section{Statements of the results}\label{statements}

Let $\Delta$ be a polar space of rank $n$, with $n \geq 3$, having type set $\mathsf{T}$. Again, denote by $\Omega_s$ the set of singular subspaces of $\Delta$ having type $s$. 
We define, for each pair $i,j \in \mathsf{T}$, three classes of bipartite graphs with bipartition classes $C_1 = \Omega_i$ and $C_2 = \Omega_j$ (this entails two disjoint copies of $\Omega_i$ if $i=j$). The first one's adjacency corresponds to a Weyl distance $w$ between some $i$-space $I_0$ and some  $j$-space $J_0$. By Lemma~\ref{weyl}, $(I,J) \in C_1 \times C_2$ are adjacent if $\mathsf{t}(I\cap J) =\mathsf{t}(I_0 \cap J_0)$ and $\mathsf{t}(I^J)=\mathsf{t}({I_0}^{J_0})$. The latter type sets can also be $-1$ and, in case $\Delta$ is hyperbolic, also $\{(n-1)',(n-1)''\}$. Therefore, we let $k,\ell$ be elements of $\mathsf{T}\cup\{-1\}$ and, if $\Delta$ is hyperbolic, we also allow $\{(n-1)',(n-1)''\}$ (which we abbreviate to $n-2$). 

\begin{definition} \rm 
\begin{itemize}
\item In the \emph{$(k,\ell)$-Weyl graph} $\Gamma_{i,j;k,\ell}^n(\Delta)$, a pair of vertices $(I,J) \in C_1 \times C_2$ is adjacent precisely if $\mathsf{t}(I \cap J)=k$ and $\mathsf{t}(I^J)=\ell$,
\item  In the \emph{$k$-incidence graph} $\Gamma_{i,j;k}^n(\Delta)$, a pair of vertices $(I,J) \in C_1 \times C_2$ is adjacent precisely if $\mathrm{t}(I \cap J)=k$,
\item In the \emph{$k_\geq$-incidence graph} $\Gamma_{i,j;\geq k}^n(\Delta)$, a pair of vertices $(I,J) \in C_1 \times C_2$ is adjacent precisely if $\dim(I\cap J) \geq |k|$.
\end{itemize}
 \end{definition}
 
\textbf{Convention} $-$ In short, we will determine the automorphism groups of the above graphs. However, there is just  one case that we will not consider in this paper, being the $(k,\ell)$-Weyl graph where $|\ell|=n-1$ when $|i| < n-1$ or $|j| <n-1$, i.e., we only allow $|\ell|=n-1$ in case $|i|=|j|=n-1$. This is a very specific case that does not fit in the technique used in this paper. 
 
Clearly, the definitions of the $k$-incidence graphs and the $k_{\geq}$-incidence graphs are independent of the order of $i$ and $j$. We now discuss what happens for the $(k,\ell)$-Weyl graph if we switch the roles of $i$ and $j$. Let $I$ and $J$ be adjacent vertices in $\Gamma_{i,j;k,\ell}^n(\Delta)$ and put $\overline{\ell} = \mathsf{t}(J^I)$. If $\Delta$ is not hyperbolic or $|\ell| < n-1$, then $\overline{\ell} = \ell$ and hence switching the roles of $i$ and $j$ yields the same graph. If $|\ell|=n-1$ and $\Delta$ is hyperbolic, possibly $\overline{\ell} \neq \ell$ (i.e., then $\overline{\ell}=\ell'$) and in that case $\Gamma_{i,j;k,\ell}^n(\Delta) \neq \Gamma_{j,i; k, \ell}^n(\Delta)$, however, $\Gamma_{i,j;k,\ell}^n(\Delta) = \Gamma_{j,i;k,\overline{\ell}}^n(\Delta)$.

As we are only concerned with the automorphism group of the graphs, isomorphic graphs are considered equivalent.

\par\medskip 
\textbf{Trivial and equivalent cases} $-$ The above graphs are considered trivial if they or their bipartite complements (which are obtained by interchanging edges and non-edges between the biparts while keeping no edges within the biparts) are empty or matchings.  We list the cases for which it is obvious that they are trivial or equivalent to other cases.
\begin{itemize}
\item \textit{Suppose first that $\Delta$ is not hyperbolic.} In order for the graphs to be nonempty, we need $k \leq \min\{i,j\}$ and for $\Gamma_{i,j;k,\ell}^n(\Delta)$ we also need $\max\{i,j\} \leq \ell \leq i+j-k$.  A matching occurs if $k=i=j$.  If $k+1=i=j=0$, then $\Gamma_{i,j; k}^n(\Delta)$ is the bipartite complement of a matching; if $k=-1$, then $\Gamma_{i,j;\geq k}^n(\Delta)$ is a complete bipartite graph. Also note that, if $i=j=n-1$, then $\Gamma_{k;\ell} = \Gamma_k$ (as $\ell=n-1$ anyhow). 

\item \textit{Next suppose that $\Delta$ is hyperbolic.}  The previous paragraph still applies if we replace $i,j,k, \ell$ by $|i|, |j|, |k|, |\ell|$. However, if $|k|=|i|=|j|=n-1$, we need to be more precise: if $i=j$ then the graphs are matchings if $k=i$ and empty if $k \neq i$; if $i \neq j$ then they are empty.  Moreover, there are additional trivial/equivalent cases when $n-1 \in \{|i|,|j|,|k|,|\ell|\}$. To study those cases, assume the previously mentioned measures have already been taken into account. This implies that we may assume that $k = |k| < n-1$. 

Assume $|i|=|j|=n-1$. Note that this is always the case for the $(k,\ell)$-Weyl graph as soon as $|\ell|=n-1$, by our convention. In order for this graph to be non-empty, $i =\ell$, $j=\overline{\ell}$ and, moreover, if $i=j$ then $n-k$ should be odd, if $i \neq j$ then $n-k$ should be even. The latter also holds when $\Gamma=\Gamma_{i,j;k}^n(\Delta)$ when $|i|=|j|=n-1$, note that in fact $\Gamma_{i,j;k,\ell}^n(\Delta)=\Gamma_{i,j;k}^n(\Delta)$ when $|i|=|j|=n-1$. If $i=j$ (resp., $i \neq j$) and $n-k$ is even (resp., odd), then $\Gamma_{i,j;\geq k}^n(\Delta)=\Gamma_{i,j; \geq(k+1)}^n(\Delta)$. As the latter two graphs are equivalent, we will choose not to work with $\Gamma_{i,j;\geq k}^n(\Delta)$, since intersecting in exactly a $k$-space does not occur. Lastly, if $k\leq 0$, we also have that $\Gamma_{i,j;k,\ell}^n(\Delta)$ (and hence also $\Gamma_{i,j;k}^n(\Delta)$) is isomorphic to the bipartite complement of $\Gamma_{i,j;\geq k+2}^n(\Delta)$. The latter graph is easier to work with, so that is what we will do.

If $n=4$ then $\Gamma_{1,1;0,1}^4(\Delta) \cong \Gamma_{1,1;-1,3'}^4(\Delta) \cong \Gamma_{1,1;-1,3''}^4(\Delta)$ as we can apply a triality. Hence in this specific situation, we can treat a case where $|\ell|=3$ and $|i|,|j|<3$.

\end{itemize}

The automorphism groups of the trivial graphs are readily deduced. For the nontrivial graphs, it is clear that each automorphism of the associated building induces an automorphism of the graph. We aim for the converse, which roughly says that each automorphism of the graph is induced by an automorphism of the associated building. This statement is made precise in Main Theorems~\ref{main1} and~\ref{main2}, including the description of the two cases in which there are more automorphisms. In each of the latter two cases, the graph $\Gamma$, related to some building $\Delta$, turns out to be isomorphic to a graph $\Gamma'$ related to \emph{another} building in which the original building can be embedded naturally, and as such, each automorphism of this other building,  also those  not preserving $\Delta^b$, will induce an automorphism of $\Gamma$. We first discuss those two cases in detail.

\begin{ex} [\textbf{Special equivalent case 1}] \label{special1} \em  Let $\Delta$ be a parabolic polar space of rank $n$ and $\Delta'$ a hyperbolic polar space of rank $n+1$ containing $\Delta$ as a subspace. Put $\Gamma = \Gamma_{n-1,n-1; -1, n-1}^n(\Delta)$ (hence adjacent vertices in $\Gamma$ correspond to opposite MSS of $\Delta$) and 
\[ \Gamma' = \begin{cases} \Gamma_{n',n'; -1, n'}^{n+1}(\Delta') \quad \text{(} n \text{ odd)}\\  \Gamma_{n',n''; -1, n'}^{n+1}(\Delta') \quad \text{(} n \text{ even)} \end{cases}\]
(in $\Gamma'$, adjacent vertices correspond to opposite MSS of $\Delta'$) and denote the bipartition classes of $\Gamma$ by $C_1$ and $C_2$ again, and those of $\Gamma'$ by $C'_1$ and $C'_2$.

We claim that $\Gamma \cong \Gamma'$. Indeed, let $\mathsf{M}_1$ be one of the two families of MSS of $\Delta'$ and let $\mathsf{M}_2$ be the family of MSS of $\Delta'$ of the opposite type (i.e., $\mathsf{M}_1 = \mathsf{M}_2$ if $n$ is odd  and $\mathsf{M}_1$ and $\mathsf{M}_2$ are distinct if $n$ is even). We may assume that $C'_1 = \mathsf{M}_1$ and then our choice of $\mathsf{M}_2$ implies that $C'_2 = \mathsf{M}_2$. 
For $r=1,2$, consider the mappings $\beta_r: C_r \rightarrow C'_r$ which takes an element $X \in C_r$ to the  the unique element of $\mathsf{M}_r$ containing it. Then the mapping \[\beta_1 \times \beta_2: C_1 \times C_2 \rightarrow C'_1 \times C'_2: (I,J) \rightarrow (\beta_1(I),\beta_2(J))\]
defines a graph isomorphism between $\Gamma$ and $\Gamma'$: if $(I,J) \in C_1 \times C_2$ is an adjacent pair of $\Gamma$, i.e., if they are disjoint, then $\beta_1(I)$ and $\beta_2(J)$ are also disjoint and hence adjacent (precisely by our choice of $\mathsf{M}_2$); if $(I',J') \in C'_1 \times C'_2$ are disjoint, then clearly $\beta_1^{-1}(I')=I' \cap \Delta$ and $\beta_2^{-1}(J')=J' \cap \Delta$ are disjoint.

We now describe the action of an automorphism $\sigma$ of $\Delta'$ on $\Gamma$ (note that $\sigma$ does not necessarily stabilise $\Delta$, i.e., possibly $\sigma(\Delta) \neq \Delta$). Each vertex $X \in C_r$, $r=1,2$, is mapped to the vertex $(\beta_r^{-1}\circ \sigma \circ \beta_r)(X)$. As $\sigma$ preserves the adjacency of $\Gamma'$ and $\beta_1 \times \beta_2$ defines an isomorphism between $\Gamma$ and $\Gamma'$, this map preserves the adjacency of $\Gamma$ and as such, $\sigma$ induces an automorphism of $\Gamma$. Note that, in the non-bipartite case, i.e., for $\Gamma_{n-1;-1,n-1}^n(\Delta)$ (as treated in \cite{Kas-Mal:13}), we can only work with one class of MSS of $\Delta'$ at a time, so there is only such an isomorphism for $n$ odd. Note that its bipartite double is isomorphic to $\Gamma$, so when $n$ is even, taking the bipartite double yields additional automorphisms.
\end{ex}

\begin{ex} [\textbf{Special equivalent case 2}] \label{special2} \em Let $\Delta$ be a symplectic polar space of rank $n$. Then $\Delta$ arises from a symplectic polarity $\rho$ in a projective space $\mathbb{P}=\PG(2n-1, \mathbb{L})$, for some field $\mathbb{L}$. Let $\Gamma = \Gamma_{0,0;-1,0}^n(\Delta)$ (hence adjacent vertices in $\Gamma$ correspond to opposite points of $\Delta$) and $\Gamma'$ be the bipartite graph with bipartition classes $C'_1$ and $C'_2$ containing the points and hyperplanes of $\mathbb{P}$, respectively, and a point $p$ and a hyperplane $H$ are adjacent if $p \notin H$ (hence $p$ and $H$ are opposite in $\mathbb{P}$). 

Again, we claim that $\Gamma \cong \Gamma'$. 
The points of $\Delta$ are precisely those of $\mathbb{P}$, so $C_1 = C'_1$. Let $x$ be a vertex in $C_2$. We define $B(x)$ as the set of vertices of $C_1$ not adjacent with $x$. Then $B(x)$ equals the set of points of $\Delta$ equal to or collinear with $x$, i.e., this is exactly $\rho(x)$ as a set of points. Hence the morphism $\beta_2: C_2 \rightarrow C'_2: x \mapsto B(x)$ is well defined. As $B(x)= \rho(x)$, it follows that $\beta_2$ is an isomorphism. Putting $\beta_1 = \mathsf{id}_{C_1}$, we have that $\beta_1 \times \beta_2$ defines an isomorphism between $\Gamma$ and $\Gamma'$: $(p,q)$ is an adjacent pair of $\Gamma$, i.e., $p \notin q^\perp$, if and only if $\beta_1(p)=p \notin \beta_2(q)=q^\perp$, i.e., if $(\beta_1(p),\beta_2(q))$ is an adjacent pair of $\Gamma'$.

Like above, an automorphism $\sigma$ of $\mathbb{P}$ (not necessarily preserving $\Delta$) induces an automorphism of $\Gamma$ by mapping each vertex $x \in C_r$ on $(\beta_r^{-1} \circ \sigma \circ \beta_r)(x)$. The smallest example of this case has already been explained in \cite{proj} (Theorem 4.2$(vi)$). 

In the non-bipartite case, i.e., for $\Gamma_{0,-1;0}^n(\Delta)$, there is no meaningful isomorphism like above to consider, since we worked with \emph{two} types of subspaces. Also here, the bipartite double of $\Gamma_{0,-1;0}^n(\Delta)$ is isomorphic to $\Gamma$, which has additional automorphisms.   \end{ex}

As one can see, there is a similarity between those two special cases, even more when we observe that also in the first case, the vertex sets $C_1$ and $C'_1$ are point sets of certain geometries: the dual parabolic polar space and the half spin geometry, respectively. 

\begin{rem}\em
The pairs of point-line geometries corresponding to the two counter examples above, namely, \begin{compactenum} \item the pair of a projective space of odd dimension $2d-1$ and a symplectic polar space of rank $d$, $d\geq 2$, over the same field, and \item the pair of a half spin geometry of type $\mathsf{D}_k$ and a dual parabolic polar space of type $\mathsf{B}_{k-1}$, $k\geq 3$, defined over the same field (for $k=3$, this pair coincides with the first pair for $d=2$, using the same field), \end{compactenum}
are precisely the pairs of geometries related to split spherical buildings with the property that their point sets have a common projective representation as a projective variety, and the line set of the second is strictly contained in the line set of the first (the line set of the first one consists of all lines on the projective variety). Such pairs are classified by Cohen \& Cooperstein \cite{Coh-Coo:98}. The explanation why exactly these pairs turn up in our result is that the relation of being not opposite induces geometric hyperplanes in these geometries, which are induced by ordinary projective hyperplanes; these hyperplanes coincide for both geometries in the pair, and so the opposition relation in both geometries are indistinguishable. This points to the conjecture that there are no more examples of this phenomenon to be found in the non-split case. It is conceivable that our result, together with the analogue for the exceptional buildings, can be used to prove this. Note that Cardinali, Giuzzi and Pasini \cite{Car-Giu-Pas:18} verify the conjecture for (Grassmannians of) polar spaces arising from reflexive bilinear and sesquilinear forms in finite dimensional vector spaces over commutative fields. 
\end{rem}

\par\bigskip

We now state our main results. Denote by $\Aut_c(\Gamma)$ the group of automorphisms of $\Gamma$ preserving the bipartition classes of $\Gamma$. We use the terminology regarding automorphisms of $\Delta^b$ defined in the previous section.

\begin{main}\label{main1}
Let $\Gamma=\Gamma_{i,j;k,\ell}^n(\Delta)$ be nontrivial and assume moreover that if $|\ell|=n-1$, then $|i|=|j|=n-1$. Let $\rho$ be an arbitrary element of  $\Aut_c(\Gamma)$.
\begin{enumerate}[$(i)$]
\item If $\Delta$ is a parabolic polar space, $i=j=\ell=n-1$ and $k=-1$,  then $\rho$ is induced by an automorphism of a hyperbolic polar space of rank $n+1$ containing $\Delta$ and every such automorphism induces an element of $\Aut_c(\Gamma)$ (see Example~\ref{special1}).

\item If $\Delta$ is a symplectic polar space, $i=j=\ell=0$ and $k=-1$, then $\rho$ is induced by an automorphism of  its ambient projective space $\PG(2n-1,\mathbb L)$ for some field $\mathbb L$ and every such automorphism induces an element of $\Aut_c(\Gamma)$ (see Example~\ref{special2}).

\item In all other cases, $\rho$ is induced by an automorphism $\rho$ of $\Delta^b$. Moreover, the automorphisms of $\Delta^b$ inducing an element of $\Aut_c(\Gamma)$ are precisely the type-preserving ones, except if $\Delta$ is hyperbolic and one of the following holds.
\begin{enumerate}
\item  The dualities of $\Delta$ also induce elements of $\Aut_c(\Gamma)$ if $|\ell|< n-1$.
\item If $n=4$, then for each $t \in \{3',3''\}$,  the $t$-dualities of $\Delta$ also induce elements of $\Aut_c(\Gamma)$ if either $0$ and $t'$ do not occur in $\{i,j,k,\ell\}$ (including $(i,j,k,\ell)=(1,1,-1,2)$), or if $(i,j,k,\ell)=(1,1,0,2)$.

\item[(ab)] If $n=4$ and all conditions mentioned in both (a) and (b) are satisfied, i.e., if $i=j=1$ and $(k,\ell)\in\{(-1,1),(-1,2),(0,2)\}$, then also the trialities of $\Delta$ induce elements of $\Aut_c(\Gamma)$. 
\end{enumerate}
 \end{enumerate} 

If $i=j$ or if $\Delta^b$ has an automorphism switching $i$ and $j$ then $\Aut(\Gamma) = \Aut_c(\Gamma) \times 2$; otherwise $\Aut(\Gamma) = \Aut_c(\Gamma)$.
\end{main}

\begin{ex}\em
As an example to the cases mentioned in Main Theorem~\ref{main1}$(iii)$, we explain the following situation. Suppose $\Delta$ is hyperbolic, $n=4$ and $(i,j,k,\ell)=(1,1,0,2)$. We show that the $t$-dualities, for each $t \in \{0,3',3''\}$, indeed preserve the adjacency of $\Gamma$ (hence also their compositions, trialities in particular, preserve the adjacency). Let $L$ and $L'$ be adjacent lines in $\Gamma$. This means that $L \cap L'$ is a point $p$ and there is a $3'$-space $U$ and a $3''$-space $V$ containing $\<L,L'\>$. Equivalently, there is a set $\{p, U, V'\}$ of pairwise incident elements which  are all incident with both $L$ and $L'$. If we apply a $t$-duality $\theta$, then $\{p^\theta, U^\theta, V^\theta\}$ is also a set containing a point, a $3'$-space and a $3''$-space which are pairwise incident, and all of them are incident with both lines $L^\theta$ and $L'^\theta$. Hence $L^\theta$ and $L'^\theta$ are indeed adjacent vertices of $\Gamma$. The types $0,3',3''$ play the same role in the adjacency relation.
\end{ex}

\begin{main}\label{main2}
Let $\Gamma$ be $\Gamma_{i,j;\geq k}^n(\Delta)$ or $\Gamma_{i,j; k}^n(\Delta)$ and suppose $\Gamma$ is nontrivial. If $|i|=|j|=n-1$, assume moreover that $\Gamma \neq \Gamma_{i,j;k}^n(\Delta)$, since $\Gamma_{i,j;k}^n(\Delta)=\Gamma_{i,j;k,\ell}^n(\Delta)$. Let $\rho$ be an arbitrary element of  $\Aut_c(\Gamma)$. Then $\rho$ is induced by an automorphism of $\Delta^b$. Moreover, the automorphisms of $\Delta^b$ inducing an element of $\Aut_c(\Gamma)$ are precisely the type-preserving ones, except if $\Delta$ is hyperbolic and one of the following holds.

\begin{enumerate}[(a)]
\item The dualities of $\Delta$ also induce elements of $\Aut_c(\Gamma)$ if  $|i|,|j| < n-1$. 

\item If $n=4$, then for $t \in \{3',3''\}$,  the $t$-dualities of $\Delta$ also induce elements of $\Aut_c(\Gamma)$ if $(i,j)\in\{(1,t), (t,1), (t,t)\}$.
\end{enumerate}
If $i=j$ or if $\Delta^b$ has an automorphism switching $i$ and $j$ then  $\Aut(\Gamma) = \Aut_c(\Gamma) \times 2$; otherwise $\Aut(\Gamma) = \Aut_c(\Gamma)$.\end{main}

\begin{rem}
\em If $\Delta$ is of type $\mathsf{D}_4$, then the \emph{nontrivial} graphs $\Gamma_{\geq k}$ are all equivalent with $\Gamma_{k'}$ for some $k'$ or the complement of such a graph. Indeed, $\Gamma_{\geq 0} \cong \overline{\Gamma_{-1}}$; if $1 \in \{i,j\}$ or $\{i,j\}=\{3',3''\}$ then $\Gamma_{\geq 1} \cong \Gamma_1, \Gamma_2$, respectively; if $i=j=3'$ or $i=j=3''$ then $\Gamma_{\geq 1} \cong \overline{\Gamma_{-1}}$; lastly, $\Gamma_{\geq 2} \cong \Gamma_2$ (in this case $\{i,j\}=\{3',3''\}$ in order for the graph to be nontrivial). This, together with the fact that the presence of $t$-dualities ($t \in \{0,3',3''\}$) only depends on $i$ and $j$, explains why we do not distinguish between those two types of graphs in Main Theorem~\ref{main2}$(a)$ and $(b)$. 
\end{rem}
\begin{ex} \em
Note that in Main Theorem~\ref{main2}, there are no trialities of $\Delta$ inducing elements of $\Aut_c(\Gamma)$. Indeed: for example, if $\Gamma = \Gamma_{1,1; -1}^4(\Delta)$,  two lines corresponding to adjacent vertices are mapped by a triality on two lines that possibly share a point, which happens if the original lines were contained in a $t$-space ($t\in\{3',3''\}$).
Like before, there is only a $t$-duality ($t\in\{3',3''\}$) if the relations of an adjacent pair of vertices w.r.t.\ subspaces of types $0$ and $t'$ is symmetrical.
\end{ex}
One could also consider the non-bipartite versions of the graphs defined above, denoted by $\Gamma_{j;k,\ell}^n(\Delta)$, $\Gamma_{j;k}^n(\Delta)$ and $\Gamma_{j;\geq k}^n(\Delta)$, respectively, with self-explaining notation. In general, the (extended) bipartite double $2\Gamma$ ($\overline{2}\Gamma$) of a given graph $\Gamma$  is obtained by taking two copies of the vertex set of $\Gamma$, without the edges, and defining a vertex of one copy to be adjacent to a vertex of the other copy if the corresponding vertices are (equal or) adjacent in $\Gamma$. It is clear that $\Aut\,\Gamma$ is isomorphic to a (possibly proper) subgroup of  $\Aut_c\,(2\Gamma)\leq \Aut\,(2\Gamma)$ and of $\Aut_c\,(\overline{2}\Gamma)\leq \Aut\,(\overline{2}\Gamma)$. This almost immediately yields the following corollaries. Note that there is no counterpart of Main Theorem~\ref{main1}$(i)$ for $n$ is even, nor for Main Theorem~\ref{main1}$(ii)$, as was explained in Examples~\ref{special1} and~\ref{special2}.

\begin{cor}\label{cor1}
Let $\Gamma=\Gamma_{j;k,\ell}^n(\Delta)$ be nontrivial and assume moreover that if $|\ell|=n-1$, then $|j|=n-1$. Let $\rho$ be an arbitrary element of  $\Aut(\Gamma)$.
 
\begin{itemize}
\item[$(i)$] If $\Delta$ is a parabolic polar space and $j=\ell=n-1$, $k=-1$ and $n$ is odd, then $\rho$ is induced by an automorphism of hyperbolic polar space of rank $n+1$ containing~$\Delta$ and every such automorphism induces an element of $\Aut(\Gamma)$ (see Example~\ref{special1}).
\item[$(ii)$] In all other cases, $\rho$ is induced by an automorphism of $\Delta^b$. Moreover, the automorphisms of $\Delta^b$ inducing an element of $\Aut_c(\Gamma)$ are precisely the type-preserving ones, except if $\Delta$ is hyperbolic and one of the following holds.
\begin{enumerate}[(a)]
\item The dualities of $\Delta$ also induce elements of $\Aut_c(\Gamma)$ if $|\ell|<n-1$.
\item If $n=4$, then for $t \in \{3',3''\}$,  the $t$-dualities of $\Delta$ also induce elements of $\Aut_c(\Gamma)$
  if either $0,t'$ do not occur in  $\{j,k,\ell\}$ (including $(j,k,\ell)=(1,-1,2)$), or if $(j,k,\ell)=(1,0,2)$.

\item[(ab)] If $n=4$ and the conditions mentioned in both (a) and (b) are satisfied, i.e.,  if $j=1$ and $(k,\ell)\in\{(-1,1), (-1,2), (0,2)\}$, the trialities of $\Delta$  also induce elements of $\Aut_c(\Gamma)$.

\end{enumerate}

\end{itemize} \end{cor}

\begin{cor}\label{cor2}
Let $\Gamma$ be $\Gamma_{j;\geq k}^n(\Delta)$ or $\Gamma_{j;k}^n(\Delta)$ and suppose $\Gamma$ is nontrivial. If $|j|=n-1$, assume moreover that $\Gamma \neq \Gamma_{j;k}^n(\Delta)$ since $\Gamma_{j;k}^n(\Delta)=\Gamma_{j;k,\ell}^n(\Delta)$. Let $\rho$ be an arbitrary element of $\Aut(\Gamma)$. Then $\rho$ is induced by an automorphism of $\Delta^b$. Moreover, the automorphisms of $\Delta^b$ inducing an element of $\Aut_c(\Gamma)$ are precisely the type-preserving ones, except if $\Delta$ is hyperbolic one of the following holds.
\begin{enumerate}[(a)]
\item The dualities of $\Delta$ also induce elements of $\Aut_c(\Gamma)$ if $|j| < n-1$.
\item  If $n=4$, then for $t \in \{3',3''\}$,  the $t$-dualities of $\Delta$ also induce elements of $\Aut_c(\Gamma)$
  if $j=t$.
\end{enumerate}

\end{cor}

 \par
\bigskip
For simplicity, we henceforth denote the graphs $\Gamma^n_{i,j;k,\ell}(\Delta)$, $\Gamma^n_{i,j; \geq k}(\Delta)$ and $\Gamma^n_{i,j; k}(\Delta)$ by $\Gamma_k^{\ell}$, $\Gamma_{\geq k}$ and $\Gamma_{k}$, respectively. We always assume these graphs to be nontrivial. According to the following remark, we may also assume that $\Delta$ is an infinite polar space which is not $\Delta(\mathbb{L})$.

\begin{rem} \label{infinite} \rm When $\Delta$ is a finite polar space, Main Theorems~\ref{main1} and~\ref{main2} can be proven using a group-theoretical result of Liebeck, Praeger and Saxl \cite{Lie-Pre-Sax:87} on the maximal subgroups of the alternating and symmetric groups. For more details, see \cite{proj}. Also, if $\Delta=\Delta(\mathbb{L})$, then $\Delta^b$ is isomorphic to a projective space of dimension $3$ over $\mathbb{L}$ and all occurring graphs in this case are also graphs that occurred in \cite{proj}. Hence the result follows from this paper. \end{rem}

\section{Sketch of the proof}\label{specialcase}
It is in fact possible to prove Main Theorem~\ref{main2} along the lines of \cite{proj}, though extra cases arise. However, Main Theorem~\ref{main1} requires another approach, as only the concept of the so-called \emph{ round-up triples} and \emph{round-up quadruples}  can be recycled from \cite{proj}. This new approach is general enough to cover Main Theorem~\ref{main1} and~\ref{main2} at the same time and provides a more elegant proof for Main Theorem~\ref{main2}.

We start by (re)defining the round-up triples and quadruples,
stated in terms of $j$, but equally valid for $i$. For any graph and any subset $V$ of its vertices, we denote by $\mathsf{N}(V)$ the set of all common neighbours of $V$, i.e., $\mathsf{N}(V)=\bigcap_{v\in V} \mathsf{N}(v)$, with $\mathsf{N}(v)$ the neighbourhood of~$v$. 

\begin{definition}\emph{ A set $\{J_1,J_2,J_3\}$ of three distinct element of $\Omega_j$ is called a \emph{round-up triple} if no vertex is adjacent to exactly two of them and $\mathsf{N}(J_1,J_2,J_3)$ is nonempty.}\end{definition}

\begin{definition}\emph{
A set $\{J_1,J_2,J_3,J_4\}$ of four distinct elements of $\Omega_j$ is called  a \emph{round-up quadruple} if every vertex that is adjacent to at least two of them is adjacent to at least three of them and the sets $\mathsf{N}(J_1,J_2,J_3,J_4)$ and $\mathsf{N}(J_1,J_2,J_3)\setminus \mathsf{N}(J_1,J_2,J_3,J_4)$ are nonempty for any permutation of the indices.}  \end{definition}

When $\Gamma=\Gamma_{\geq k}$, we aim to classify round-up triples; when $\Gamma \in \{\Gamma_k, \Gamma_k^{\ell}\}$, we aim to classify round-up quadruples. To this end, we give a construction of an $i$-space adjacent to two $j$-spaces at distance $2$ in $\Gamma$ (Section~\ref{constructie}). Since such an $i$-space then has to be adjacent to a third member of the round-up triple or quadruple, this limits the possible configurations of such triples and quadruples. We narrow down these possibilities until we obtain a Grassmann graph or a graph strongly related to it (Section~\ref{k>-1} for $k>-1$ and Section~\ref{k=-1} for $k=-1$).  The latter graphs determine $\Delta^b$ completely (see Section~\ref{grassmann}).

As such, an automorphism $\sigma$ of $\Gamma$ extends to an automorphism $\overline{\sigma}$ of $\Delta^b$. We even claim that $\sigma$ is the restriction of $\overline{\sigma}$ to the $i$- and $j$-spaces. Suppose that we constructed $\Delta^b$ out of its $j$-spaces (which is the case if we first construct $\mathsf{G}_j$ from $\Gamma$). By definition of $\overline{\sigma}$, its action on the $j$-spaces coincides with the action of $\sigma$ on the $j$-spaces. Now the action of $\sigma$ on one of the biparts of $\Gamma$ uniquely determines the action on the other bipart, since  $\mathsf{N}_\Gamma(I) = \mathsf{N}_\Gamma(I')$ if and only if $I=I'$ (of course still under the assumption that $\Gamma$ is nontrivial). Hence also the actions of $\overline{\sigma}$ and $\sigma$ on the $i$-spaces coincides. This shows that $\sigma$ is indeed the restriction of an automorphism of~$\Delta^b$.

\par\bigskip
\textbf{Convention} $-$ In order to consider round-up triples and round-up quadruples at the same time, a round-up triple $\{J_1,J_2,J_3\}$ will be written as $\{J_1,J_2,J_3,J_3\}$. Conversely, if $\{J_1,J_2,J_3,J_4\}$ in fact represents a round-up triple, we assume $J_3 = J_4$. For simplicity, we will refer to a round-up quadruple simply by a \emph{quadruple}, whenever we need ordinary quadruples, we will make this clear by calling these $4$-tuples; likewise for the (round-up) triples.
\par\bigskip

\section{Grassmann graphs}\label{grassmann}
The $t$-Grassmann graph is the collinearity graph of the so-called $t$-Grassmannian geometry associated to $\Delta$ and is defined as follows. 

\begin{definition} For $t \in \mathsf{T}$, the \emph{$t$-Grassmann graph} $\mathsf{G}_t(\Delta)$ has $\Omega_t$ as vertex set, and two vertices $U$ and $V$ are adjacent if $\dim(U \cap V) = \max\{T \cap T' \mid T,T' \in \Omega_t,\,T \neq T'\}$ and $\dim(U^V)=\dim(V^U)= \min\{|t|+1,n-1\}$. 
\end{definition}

If $\Delta$ is hyperbolic, we also consider $\mathsf{G}_{n-2}(\Delta)$ and $\mathsf{G}_{n-1}(\Delta)$, whose definitions are analogous up to the indices that now refer to dimensions only. If no confusion is possible, we omit $\Delta$.
Throughout the proofs of Main Results~\ref{main1} and~\ref{main2}, we will encounter a graph with the same vertex set as $\mathsf{G}_t$ where two $t$-spaces are adjacent precisely if their intersection has maximal dimension amongst all elements of $\{T \cap T' \, \mid \, T,T' \in \Omega_t,\,T \neq T'\}$. This graph will be denoted $\mathsf{G}'_t$ and $\mathsf{G}_t$ can be reconstructed from it, as the following lemma says. 

\begin{lemma}\label{gras3} For all $t \in (\mathsf{T} \cup \{n-2,n-1\})\setminus\{0\}$, we can construct $\mathsf{G}_t$ from $\mathsf{G'_t}$.\end{lemma}
\begin{proof} If $|t|=n-1$, clearly $\mathsf{G}_t=\mathsf{G'_t}$. So suppose $|t| < n-1$. A standard arguments yields two types of maximal cliques in $\mathsf{G'_t}$: One consisting of all $t$-spaces in a $(t+1)$-space, and one containing all $t$-spaces containing a common $(t-1)$-space. Either way, two $t$-spaces in such a maximal clique are contained in a singular subspace precisely if there exists a vertex outside the clique that is adjacent to both of them. Removing the edges in $\mathsf{G'_t}$ between vertices for which this is not the case, $\mathsf{G}_t$ is obtained. \end{proof}

\par\bigskip
The following proposition can be found in the literature (\cite{gras}), but we include a proof written in the same spirit as the rest of this paper for completeness' sake. Note also that this result was implicitly contained in the characterisations of polar Grassmannians obtained in the eighties mainly (\cite{Bro-Wil:83}, \cite{Cam:82}, \cite{Coh-Coo:83}, \cite{Coo:77}, \cite{Ell-Shu:88}, \cite{Han:86}, \cite{Sch:94}).

\begin{prop}\label{gras1} For all $t \in \mathsf{T}$, the $t$-Grassmann graph $\mathsf{G}_t$ uniquely determines $\Delta^b$. That is, it uniquely determines $\Delta$ if $\Delta$ is not hyperbolic and, if $\Delta$ is hyperbolic, up to triality or $t_1$-duality for $t_1\in \{0,3',3''\}$ if $(n,t)=(4,1)$, up to $t$-duality if $(n,t)\in\{(4,3'),(4,3'')\}$ and up to duality if $t\notin \{(n-1)',(n-1)''\}$.
\end{prop}

As a consequence of Lemma~\ref{gras3}, we immediately have the following corollary.

\begin{cor}\label{gras2}
For all $t \in (\mathsf{T} \cup \{n-2,n-1\})\setminus\{0\}$, the graph $\mathsf{G'_t}$ uniquely determines $\Delta^b$. That is, it uniquely determines $\Delta$ if $\Delta$ is not hyperbolic and, if $\Delta$ is hyperbolic, up to triality or $t_1$-duality for $t_1\in \{0,3',3''\}$ if $(n,t)=(4,1)$, up to $t$-duality if $(n,t)\in\{(4,3'),(4,3'')\}$ and up to duality if $t\notin \{(n-1)',(n-1)''\}$.
\end{cor}

\begin{lemma}\label{step1} For all $t \in \mathsf{T} \cup \{n-2\}$ with $|t|<n-1$, we can construct $\mathsf{G_{n-1}}$ from $\mathsf{G}_t$, unless $\Delta$ is of type $\mathsf{D}_4$ and $t=1$. In the latter case, $\mathsf{G}_1$ uniquely determines $\Delta^b$, i.e., it uniquely determines $\Delta$ up to  $t_1$-duality for $t_1 \in \{0,3',3''\}$.
\end{lemma}

\begin{proof} As $|t|<n-1$, $t=|t|$. A set consisting of all $t$-spaces in a $(t+1)$-space is clearly a clique of $\mathsf{G}_t$, as is a set consisting of all $t$-spaces that go through a fixed $(t-1)$-space and are contained in some MSS. Denote by $\mathcal{C}_1$ all cliques of the first type and by $\mathcal{C}_2$ all cliques of the second type. A clique maximal with the property of being contained in more than one maximal clique is the set of all lines through a $(t-1)$-space and contained in a $(t+1)$-space. Such a clique is denoted by $\mathcal{C}(M,N)$ if $M$ and $N$ are two of its members. 

Suppose first that $0 < t < n-3$. Again, standard arguments imply that $\mathcal{C}_1 \sqcup \mathcal{C}_2$ coincides with the set of maximal cliques of $\mathsf{G}_t$. Let $\mathcal P$ be the poset consisting of elements $\{ M_1 \cap \dots \cap M_m \mid m \in \mathbb{N}_{>1}, M_i \text{ a maximal clique}, M_i \neq M_j,1 \leq i < j \leq m\}$. A member of $\mathcal{P}\setminus\{\emptyset\}$ is the set of all $t$-spaces through a $(t-1)$-space and contained in a $(t+s-1)$-space for some $s$ with $1 \leq s \leq n-t-1$ (denote the subset of $\mathcal{P}$  consisting of those members by $\mathcal{P}_s$). 
A maximal clique is of the second type precisely if it contains an element of $\mathcal{P}_3$. By taking the cliques of the first type, we obtain the vertices of $\mathsf{G'}_{t+1}$. They correspond to adjacent vertices of this graph if they have a one element (a $t$-space) in common. By Lemma~\ref{gras3}, we obtain $\mathsf{G}_{t+1}$ and we can continue up to $\mathsf{G}_{n-3}$. Hence we still need to deal with $t \in \{0,n-3,n-2\}$. 
\begin{itemize}
\item[($t=0$)] In this case, the set of maximal cliques is given by $\mathcal{C}_2$ and hence we obtain all $(n-1)$-spaces. Considering the poset $\mathcal{P}$ again, it is easily seen that we can determine when two such MSS intersect in an $(n-2)$-space. 
\item[($t=n-2$)] Now, $\mathcal{C}_1$ is the set of all maximal cliques and again, we obtain all $(n-1)$-spaces. They intersect each other in an $(n-2)$-space if they share precisely one element.

\item[($t=n-3$)] Note that, when endowed with the elements of $\mathcal{P}_2$ as lines, a clique of type $\mathcal{C}_1$,  is isomorphic to a $(t+1)$-space and a clique of type $\mathcal{C}_2$ to an $(n-t-1)$-space. If $t+1 > n-t-1$, which happens if $n>4$, we can distinguish the cliques by their dimension. Indeed, in our graph this comes down to the following: in a maximal clique of type $\mathcal{C}_2$, each two elements of $\mathcal{P}_2$ have a $t$-space in common, whereas, if $n > 4$, a maximal clique of type $\mathcal{C}_1$ contains elements of $\mathcal{P}_2$ that do not share a $t$-space. This way we can again recognise the cliques of type $\mathcal{C}_1$, after which we can proceed by constructing $\mathsf{G}_{n-2}$ from this, like above.

If $n=4$, then $t=1$ and we consider two lines $M$ and $N$ at distance 2 in $\mathsf{G}_{1}$ having at least two common neighbours. Either $M$ and $N$ are disjoint collinear lines and hence $\<M,N\>$ is a $3$-space $V$, or $M$ and $N$ are intersecting non-collinear lines and hence $M \cap N$ is a point, which we also call $V$ (to deal with both cases at the same time). We now construct their convex closure (called a \emph{symplecton}), which in the first case consists of all lines incident with $V$.  We start with all members of $\mathsf{N}(M,N)$, which are clearly incident with $V$. For any $R \in \mathsf{N}(M,N)$, we also take all members of $\mathcal{C}(M,R)$ and of $\mathcal{C}(R,N)$. This way we have already obtained the set $C_M$ which are, in the first case, all lines incident with $V$ and intersecting $M$ and, in the second case, all lines incident with $V$ that are collinear with $M$; likewise we have a set $C_N$. Finally, for any pair $(V,V') \in C_M \times C_N$ having distance two in $\mathsf{G}_1$, we add all members of $\mathsf{N}(V,V')$. It is easily verified that we have all elements of $\mathsf{G}_1$ incident with $V$. 

Now, the set of all lines an $3$-space is, regardless of the type of $\Delta$, a Klein quadric (note that we also have its lines, which are the planar line pencils). However, the set of lines through a point is a hyperbolic polar space of rank 3 if and only if $\Delta$ is hyperbolic. So if $\Delta$ is of type $\mathsf{D}_4$, we cannot distinguish between the two types of symplecta; if $\Delta$ is not of type $\mathsf{D}_4$, we can. In order to do so, take two such symplecta that have more than one line in common. It is impossible that both symplecta are of the second type; if the symplecta are both of the first type, this means that the $3$-spaces intersect in a plane, moreover, as $\Delta$ is not of type $\mathsf{D}_4$, there are more symplecta through this intersection; if the symplecta are of different types, the point is contained in the $3$-space and, regardless of the type of $\Delta$, there are no other symplecta through this intersection. Hence we can indeed distinguish between the two types of symplecta. 

We conclude that, if $\Delta$ is of type $\mathsf{D}_4$, the set of all symplecta yields a tripartite graph by letting two of them be adjacent whenever they have more than one line in common. So up to a permutation of the types $\{0,3',3''\}$, we obtain $\Delta$. In all other cases, we construct the graph having the symplecta corresponding to the $3$-spaces as vertices and with adjacency ``having more than one line in common'' and obtain $\mathsf{G}_{n-1}$.
\end{itemize}
\end{proof}

Let $\mathsf{C}_{\mathsf{T}}(\Delta)$ denote the incidence graph of $\Delta$ and, for $\mathsf{T}' \subseteq \mathsf{T}$, let $\mathsf{C}_{\mathsf{T}'}(\Delta)$ denote its restriction to elements of types in $\mathsf{T'}$. We will use the notation $[s,t]$ for all types in between $s$ and $t$.

\begin{lemma} 
\begin{itemize}
\item For any polar space $\Delta$, the graph $\mathsf{G_{n-1}}(\Delta)$ completely determines $\Delta$ if $\Delta$ is not hyperbolic and up to duality if $\Delta$ is hyperbolic.
\item  For any hyperbolic polar space $\Delta$, the graph $\mathsf{G}_t(\Delta)$, for $t \in \{(n-1)',(n-1)''\}$, completely determines $\Delta$ if $n\neq 4$; and up to $t$-duality if $n=4$. 
\end{itemize}
\end{lemma}
\begin{proof}
Let $\mathsf{G}$ be one of $\mathsf{G}_{n-1}(\Delta), \mathsf{G}_t(\Delta)$ with $t \in \{(n-1)',(n-1)''\}$. In the latter case,  we assume $\Delta$ to be hyperbolic and $n>4$ (note that, if $n=3$, $\mathsf{G}_{t}(\Delta)$ is the collinearity graph of a projective $3$-space). 
We now construct $\mathsf{C}_{[n-3,n-2]}(\Delta)$ and $\mathsf{C}_{[n-4,n-3]}(\Delta)$, respectively. First observe that the maximal cliques in $\mathsf{G}_{n-1}$ correspond to the $(n-2)$-spaces, and are determined by any two of its members, say $U$ and $V$, and then the clique is denoted by $\mathcal{C}(U,V)$. The maximal cliques in $\mathsf{G}_t(\Delta)$ correspond to the  $(n-4)$-spaces and the $t'$-spaces. Those are not determined by any two of their members: suppose $(U_i)_i$ is a subset of a maximal clique which all contain a common $(n-3)$-space $W$, then $(U_i)_i$ is contained in all maximal cliques corresponding to $(n-4)$-spaces containing $W$ and in all maximal cliques corresponding to $t'$-spaces containing $W$. In general, ``submaximal'' cliques (cliques which are maximal with the property of being contained in more than two maximal cliques) correspond to the $(n-3)$-spaces and are determined by any two of its members $U$ and $V$ and we also denote this by $\mathcal{C}(U,V)$.  

Now take two elements $M$ and $N$ at distance two in $\mathsf{G}$. Their convex closure can be obtained as in the proof of the previous lemma and yields all elements of $\mathsf{G}$ containing $M \cap N$. In the first case, this readily gives us $\mathsf{C}_{[n-3,n-2]}(\Delta)$. In the second case, we can construct $\mathsf{C}_{[n-4,n-3]}(\Delta)$, as two distinct $(n-3)$-spaces intersect in an $(n-4)$-space if they have at least two $(n-5)$-spaces in common; and the intersection of the two sets of $t$-spaces through those two $(n-5)$-spaces is then exactly the set of $t$-spaces through this $(n-4)$-space.

Given $\mathsf{C}_{[i,i+1]}(\Delta)$, for $1\leq i\leq n-3$, we can build the graph $\mathsf{C}_{[i-1,i]}(\Delta)$. Let $R$ and $Q$ be two $i$-spaces contained in an $(i+1)$-space $S_{R,Q}$. Then $R\cap Q$ is $(i-1)$-dimensional and we aim for all $i$-spaces through $R\cap Q$. We start by taking all $i$-spaces $L \notin \mathsf{N}(S_{R,Q})$ such that $\mathsf{N}(L,R)$ and $\mathsf{N}(L,Q)$ are nonempty. Then $L$ contains $R \cap Q$. Also, the members of $\mathsf{N}(S_{R,Q})$ that have a common neighbour with $L$ are all $i$-spaces in $S_{R,Q}$ through $R \cap Q$. This way, we already obtained all $i$-spaces through $R\cap Q$ collinear with $S_{R,Q}$. Now let $M$ be such an $i$-space that is not contained in $S_{R,Q}$. Then any $i$-space through $R\cap Q$ is collinear with at least an $i$-space of each of $S_{Q,R}, S_{R,M},S_{M,Q}$, and at least two of these $i$-spaces are, or can be chosen, distinct. Hence, repeating the previous argument for all pairs $(R',Q')$ in $S_{Q,R} \cup S_{R,M} \cup S_{M,Q}$, we obtain all $i$-spaces through $R\cap Q$.

Finally, we deal with the graph $ \mathsf{G}_t(\Delta)$ in the case where $\Delta$ is hyperbolic and $n=4$. As before, we can construct the lines of $\Delta$, and given the lines and the $t$-spaces, we can construct $\Delta$ up to a $t$-duality. Indeed, using the fact that $\Delta$ allows a triality, we may assume that we are given its points and lines, and then the planes and the $3$-spaces (so without distinction between the $3'$- and $3''$-spaces) can be constructed from this.
\end{proof}
\par\medskip
The above lemmas now prove Proposition~\ref{gras1}. 

\section{Construction of an $i$-space adjacent to two $j$-spaces at distance $2$ in~$\Gamma$}\label{constructie}

Let $\Gamma$ be one of $\Gamma_{\geq k}, \Gamma_k, \Gamma_k^\ell$. Most of the time, it will be most convenient to assume $|i|\leq |j|$, up to one particular situation:

\textbf{Convention on $i$ and $j$} $-$  If $\max\{|i|,|j|\}<|\ell|$, we suppose $|j|=\max\{|i|,|j|\}$; if  $\max\{|i|,|j|\}=|\ell|$, we suppose that $|j|=\min\{|i|,|j|\}$.

Let $J_1$ and $J_2$ be elements of $\Omega_j$ at distance $2$. In general, an  element of $\mathsf{N}(J_1,J_2)$ is generated by three kinds of subspaces (those will be called the `building bricks') which we want to be able to place in ``good'' positions, as will be explained below. We first introduce notation regarding these building bricks, after which we start the construction. 

\subsection{The building bricks}

The mutual position of any pair of subspaces of $\Delta$ is determined by their intersection and projection on each other. 

\textbf{The mutual position of $J_1$ and $J_2$} $-$ Let $\{c,c'\}$ be $\{1,2\}$ throughout this section. The mutual position of $J_1$ and $J_2$ is  described as follows (see Figure~\ref{adjij}). \vspace{-0.5em}

\begin{itemize}[$\circ$]
\item The \emph{intersection} is the subspace $J_1 \cap J_2$ and is denoted by $S$, its dimension by $s$;
\item the \emph{collinear part} is the set  $\proj_{J_c}(J_{c'})\setminus S$ and is denoted by $\mathsf{P}_c$. We fix a subspace $P_c \subseteq \mathsf{P}_c$ such that $\<S, P_c\> = \proj_{J_c}(J_{c'})$, and denote by $p^*$ the dimension of $P_c$;
\item the \emph{semi-opposite part} is the set $J_c\setminus (S \cup \mathsf{P}_c)$ and is denoted by  $\mathsf{Q}_c$. Let $Q_c\subseteq\mathsf{Q}_c$ be a fixed subspace such that $\<S,P_c,Q_c\>=J_c$, and denote by $q^*$ the dimension of $Q_c$. 
\end{itemize} 
\vspace{-0.5em}
As the notation suggests, $p^*$ and $q^*$ do not depend on the value of $c$. The subspace spanned by the intersection and the collinear part, i.e., $\<S,P_1,P_2\>$, is sometimes denoted by $P$ and it is equal to $\<\proj_{J_1}(J_2),\proj_{J_2}(J_1)\>$.
\begin{figure}
    \centering
      \includegraphics[width=0.8\textwidth]{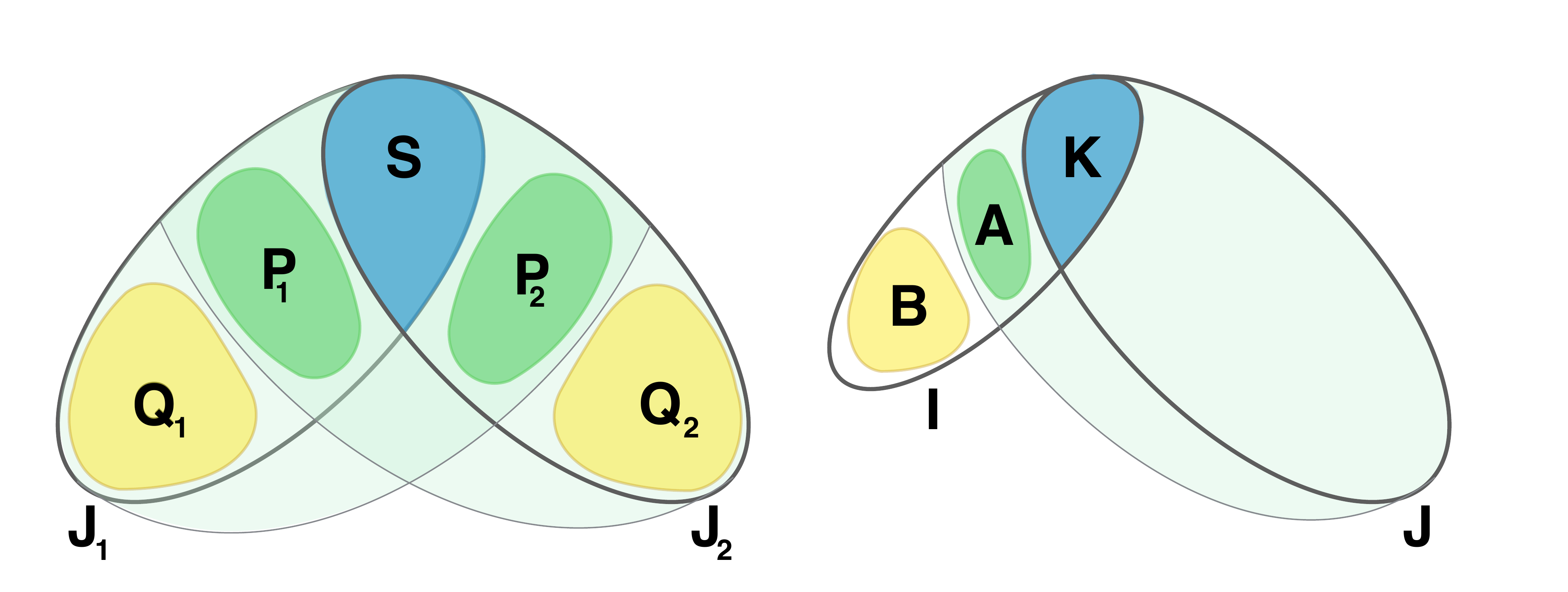}
        \caption{Mutual position of  $J_1$ and $J_2$ (left) and the position of $I$ w.r.t.\ $J$ (right)}
\label{adjij}
\end{figure}
\par\bigskip

\textbf{The position of $I$ w.r.t.\ $J_c$} $-$ For $I \in \mathsf{N}(J_1,J_2)$, its position w.r.t.\ $J_1$ and $J_2$ has an analogous description but we use $K$, $A$ and $B$ instead of $S,P$ and $Q$ to denote each of the previous subsets (see Figure~\ref{adjij}). Again we fix subspaces $A_c \subseteq \mathsf{A}_c$ and $B_c \subseteq \mathsf{B}_c$ such that $I=\<K_c,A_c,B_c\>$. The adjacency relation in $\Gamma$ puts restrictions on the dimensions $k_c$, $a_c$ and $b_c$ of $K_c$, $A_c$ and $B_c$. Clearly,  $k_c+a_c+b_c+2=i$. If $\Gamma=\Gamma_{\geq k}$, then $k_c \geq k$, if $\Gamma=\Gamma_k$, then $k_1=k_2=:k$ and if $\Gamma = \Gamma_k^{\ell}$, all values are determined and independent of the index $c$:
\begin{align*}
k&:=k_1=k_2;\\
a&:=a_1=a_2=|\ell| -|j|-1; \\
b&:=b_1=b_2=|i|+|j|-k-|\ell|-1.
\end{align*}

For $X \in \{K,A,B\}$, $X_1 \cap X_2$ is a subspace incident with/collinear with/semi-opposite \emph{both} $J_1$ and $J_2$ and will be referred to by $X^-$, or simply $X$ whenever $X_1=X_2$. One should picture this for choices of $A_1 \subseteq \mathsf{A}_1$ and $A_2 \in \mathsf{A}_2$ for which $\<K_1 \cap K_2, A_1 \cap A_2\>\setminus \<K_1,K_2\> = \mathsf{A}_1 \cap \mathsf{A_2}$; likewise with $B_1 \subseteq \mathsf{B}_1$ and $B_2 \subseteq \mathsf{B}_2$ such that $B_1 \cap B_2$ is complimentary to $\<K_1,A_1\>$ and to $\<K_2,A_2\>$ (i.e., the subspaces $A_1$ and $A_2$, respectively, $B_1$ and $B_2$, are chosen such that their intersection is maximal). 

\par\bigskip

We say that a subspace \emph{avoids} a set of subspaces if it is disjoint from each of its members. Fact~\ref{lem4a} below describes how we can choose collinear parts and (semi-)opposite parts while avoiding a finite set of subspaces. Parts of the proofs of these facts can be found in the literature; yet the ``avoiding''-part cannot and this will be essential for us. 

\begin{fact}\label{lem4a}\label{lem5a}\label{lem6a}
Let $\Delta$ be infinite and let $i,j \in \mathsf{T}$. Let $\mathcal{F}\subseteq \Omega_j$ and $\mathcal{F}'\subseteq \Omega$ be finite sets.
\begin{enumerate}[$(i)$]

\item[$(i)^p$] Suppose $\mathcal{F}=\{U,V\}$ and let $\mathcal{F}'$ be such that each of its members intersects $U$ and $V$ in subspaces of dimension at most $\dim(U\cap V)$. Then there is a subspace $W^p$ in $\<\proj_U(V),\proj_V(U)\>\setminus(U\cup V)$ if and only if $w^p:=\dim(W^p) \leq \cod_{\proj_U(V)}(U\cap V)$. Moreover, $W^p$ can be chosen such that it avoids $\mathcal{F}'$.

\item[$(i)^o$] Suppose $\mathcal{F}=\{U,V\}$ and let $\mathcal{F}'$ be such that each of its members intersects $U$ and $V$ in subspaces of dimension at most $\dim(U\cap V)$. Then there is a singular subspace $W^o$ avoiding $(U^V \cup V^U)$ collinear with $U$ and $V$ if and only if $w^o:=\dim(W^o) \leq n-\dim(U^V) -2$. Moreover, $W^o$ can be chosen such that it avoids $\mathcal{F}'$, unless if $\Delta$ is hyperbolic, $\dim(U^V)=n-2$ and $w^o = 0$, then we only have $\dim(W^o \cap F) \leq 0$ for all $F \in \mathcal{F'}$.

\item[$(i)^*$] Combining $(i)^p$ and $(i)^o$, there is a subspace $W=\<W^p,W^o\>$ such that $W \subseteq (U^\perp \cap V^\perp)\setminus(U \cup V)$ if and only if $w:=\dim(W) \leq n-|j|-2$. Moreover, $W$ can be chosen such that it avoids $\mathcal{F}'$, unless if $\Delta$ is hyperbolic, $\dim(U^V)=n-2$ and $w=n-|j|-2$, in this case there are exactly two subspaces $P^1$ and $P^2$ containing $\<\proj_V(U), \proj_U(V)\>$ as a hyperplane and such that $W$ is contained in one of them, say in $P^1$, then for all $F \in \mathcal{F}'$ with $\dim(F \cap P^1)=\dim(\proj_{U}(V))+1$ we only have $\dim(W \cap F) = 0$.

\item[$(ii)$] There is an element  $W \in \Omega_t$ such that it is opposite each member of $\mathcal{F}$ for some type $t \in \mathsf{T}$ (if $\Delta$ is hyperbolic, $|j|=n-1$ and $n$ is odd, then $t=j'$; in all other cases $t=j$). Moreover, $W$ can be chosen such that it avoids $\mathcal{F'}$, unless if $\Delta$ is hyperbolic and $|j|=n-1$, then $W$ can be chosen such that it avoids $\mathcal{F}' \cap (\Omega\setminus\Omega_{j'})$ and intersects each member of $\mathcal{F}' \cap \Omega_{j'}$ in exactly a point.

\item[$(iii)$] Let $\mathcal{F}=\{U,V\}$ with $U \neq V$ and put $t=\mathsf{t}(U^V)$. If $|i| \leq |t| - |j| -1$, there is an element $W\in \Omega_i$ such that $W \subseteq U^\perp\setminus U$ and with $W$ and $V$ semi-opposite. Moreover, $W$ can be chosen such that it avoids $\mathcal{F}'$, except if $\Delta$ is hyperbolic and $|j|<|t|=n-1$ and $|i|=|t|-|j|-1$, then $W$ avoids $\mathcal{F}' \cap (\Omega\setminus\Omega_{t'})$ and intersects each member $F$ with $U \subseteq F \in \mathcal{F}' \cap \Omega_{t'}$ in exactly a point.

\end{enumerate}
\end{fact}

\begin{proof}
\begin{enumerate}
\item[$(i)^p$]
Put $P = \<\proj_V(U), \proj_U(V)\>$, $\overline{U}=P \cap U$, $\overline{V}=P \cap V$ and write $\ell$ for $\dim(U^V)$ and $s$ for $\dim(U\cap V)$. We show that we can find a subspace $W^p$ of dimension $w^p:=\cod_{\proj_U(V)(U\cap V)}$ in $P$ that avoids $\overline{U}$, $\overline{V}$ and $\mathcal{F}'$, which then also shows that any subspace of smaller dimension with the same properties can be found as a subspace of this one.

By assumption, each member $F \in \mathcal{F}'$ intersects $U$ and $V$, hence also $\overline{U}$ and $\overline{V}$, in subspaces of dimension at most $s$. This implies $\dim(F \cap P) \leq \dim(\overline{U})$, for if not, $F\cap P$ would intersect $\overline{U}$ in a subspace with a dimension strictly bigger than $s$. 

Hence, $W^p$ has to be a subspace complementary to $\overline{U}$ and $\overline{V}$ in $P$, which implies $W^p$ would not be found if $w^p > \cod_{\proj_U(V)}(U\cap V)$. Moreover, $W^p$ has to avoid $\mathcal{F}'$, a finite set of subspaces of dimensions at most $\dim(\overline{U})$. As $P$ is an infinite projective space, this is possible.

\item[$(i)^o$]
We keep on using the notation introduced above. We first establish $W^o \subseteq (U^\perp \cap V^\perp)\setminus(U^V \cup V^U)$, afterwards we verify whether we can do this while avoiding $\mathcal{F}'$.  Again, it suffices to do this for $w^o=n-\ell-2$ and show that we cannot find such a $W^o$ with $w_o > n-\ell-2$.

We look in $\Res_{\Delta}(P)$, where $U$ and $V$ correspond to opposite subspaces $U'$ and $V'$. In $\Res_\Delta(P)$, consider  $U'^\perp \cap V'^\perp$, which is isomorphic to a polar space $\Delta'$ of rank $n-\ell-1$. Note that $n-\ell-1 \geq 1$ since we may assume that $\ell < n-1$. Indeed, if $\ell=n-1$,  necessarily each subspace $W \subseteq U^\perp \cap V^\perp$ belongs to $P$ since no point outside $P$ can be collinear with both $U$ and $V$, as it would be collinear with the $(n-1)$-dimensional subspaces $U^V$ and $V^U$. Observe moreover that each point collinear with both $U$ and $V$ and not contained in $U^V \cup V^U$ corresponds to a point of $\Delta'$. 

Now, in $\Delta'$, let $W^o$ be a maximal singular subspace, i.e., a subspace of dimension $n-\ell-2$ (this, and the above observation, shows that $w^0 > n-\ell-2$ will not work). If $\Delta'$ has infinitely many MSS then we can choose $W^0$ in $\Delta'$ such that it avoids the set corresponding to $\mathcal{F}'$, i.e., the set $\{ P^F/P  \cap \Delta'  \mid F \in \mathcal{F}'\}$. The only case in which $\Delta'$ does not have infinitely many MSS is when $\Delta$ is hyperbolic and $\ell=n-2$, so $w^o=0$.

\item[$(i)^*$] First note that $\cod_{\proj_U{(V)}}(U\cap V)+(n-\dim(U^V)-2)+1 = n-|j|-2$. Now, if $w < n - |j| - 2$, then we choose $W^p$ and $W^o$ by means of $(i)^p$ and $(i)^o$ respectively such that $w^o < n-\dim(U^V)-2$ and $w^p+w^o+1 = w$. Then $(i)^p$ and $(i)^o$ imply that we can choose $W^p$ such that it avoids $\mathcal{F}'$ and $W^o$ in $\Res_\Delta(P)$ such that it avoids $\{P^F/P \mid F \in \mathcal{F}'\}$, implying that in $\Delta$, $\<W^p,W^o\>$ avoids $\mathcal{F}'$.
If $w=n-|j|-2$ then we are forced to choose $w^p=\cod_{\proj_U(U\cap V)}$ and $w^o=n-\dim(U^V)-2$. If $\Delta$ is not hyperbolic or $\dim(U^V) \neq n-2$, everything is as above. If $\Delta$ is hyperbolic and $\dim(U^V) = n-2$, then by $(i)^o$, there are exactly two subspaces $P^1$ and $P^2$ containing $P$ as a hyperplane which are collinear with $U$ and $V$. If we aim for a subspace $W$ of dimension $w^p+w^o+1$ in $P^c$ that avoids $U$, $V$ and $\mathcal{F'}$, then this is possible unless $\dim(F \cap P^c)=\dim(\overline{U})+1$ for some $F \in \mathcal{F}'$, as then $\dim(F \cap W)=0$.

\item[$(ii)$] 
We prove this fact by induction on $j$. The induction basis depends on $\Delta$ and $j$. Up to now, we have assumed $n \geq 3$ but since these proofs are of general nature and we want the lowest possible induction basis, we include $n=1,2$. 

\begin{itemize}
\item[(IH0)] \textit{Suppose that $|j|=0$ and that either $\Delta$ is not hyperbolic or $n \geq 2$.} Note that this assumption says that there are infinitely many points in $\Delta$.  Let there be given a finite set $\{x_1,...,x_r\}$ $(r\in \mathbb{N}$) of points. We aim for a point opposite all of them, now by using induction on~$r$. First suppose $r=1$. Let $M$ be any MSS not containing $x_1$ and not coinciding with any member of $\mathcal{F}'$. Then the set $\{M \cap F \mid F \in \mathcal{F'}\} \cup \proj_M(x_1)$ is finite and hence its union cannot cover $M$, yielding the existence of a point $p \in M$ opposite $x_1$ and avoiding $\mathcal{F}'$. Now suppose $r>1$. By induction there is a point $p$ which is opposite all points of $\{x_2,...,x_r\}$ and not contained in any member of $\mathcal{F}'$. If $p=x_1$, take any line $L$ through $x_1$. Note that no member of $\mathcal{F}'$ contains $L$ as they do not contain $p$. Since $L$ has infinitely many points, there is a point on $L$ not in $\mathcal{F}' \cup \bigcup_{i=2}^r \proj_{L} (x_i)$.  So we may assume that $p$ and $x_1$ are distinct collinear points. Now take a line $L'$ through $p$ such that, in $\Res_{\Delta}(p)$, the lines $px_1$ and $L'$ correspond to opposite points. Clearly, any point on $L'$ disjoint from $\mathcal{F}'\cup\bigcup_{i=1}^r \proj_{L'}(x_i)$ satisfies the requirements.

\item[(IH1)] \textit{Suppose that $|j|=1$ and that $\Delta$ is a hyperbolic polar space of rank 2.} 
Note that $\Delta$ is a grid and $\mathcal{F} \cup (\mathcal{F}' \cap \Omega_{j})$ is a subset of one of its reguli, whereas $\mathcal{F}' \cap \Omega_{j'}$ is a subset of the other regulus. As a regulus contains infinitely many elements, the first mentioned regulus contains an element opposite the members of $\mathcal{F}$ while avoiding the members of $\mathcal{F}' \cap (\Omega\setminus\Omega_{j'})$ and intersecting the members of $\mathcal{F}' \cap \Omega_{j'}$ in a point.
\end{itemize}

Now suppose $|j| > 0$ (in particular, $n > 1$), and if $|j|=1$, we may assume that $\Delta$ is not a hyperbolic polar space of rank 2, since that case has been dealt with already. Let $\{X_1,...,X_r\}$ ($r \in \mathbb{N}$) be a finite subset of $\Omega_j$. Take a point $x_i \in X_i$ for all $i$ with $1\leq i\leq r$. We already know that there is a point $p$ opposite all these points and avoiding $\mathcal{F}'$. In $\Res_\Delta(p)$, the $j$-spaces $X_i$ correspond to $(j-1)$-spaces $p^{X_i}$. By induction (up to case (IH1) if $\Delta$ is hyperbolic and $|j|=n-1$, otherwise up to case (IH0)), there is a $(j-1)$-space opposite all of them and avoiding the set corresponding to $\mathcal{F}'$, or, if $\Delta$ is hyperbolic and $|j|=n-1$, intersecting the members of $\mathcal{F}' \cap \Omega_{t'}$ in a point only. The corresponding $|j|$-space in $\Delta$ is opposite all members of $\{X_1,...,X_r\}$ and avoids $\mathcal{F}' \setminus \Omega_{t'}$ and up to a point, it avoids $\mathcal{F'} \cap \Omega_{t'}$ (this cannot be more than a point by going back to $\Delta$, since the dimension of intersection could only grow by one whereas the parity has to remain unchanged). 

\item[$(iii)$]
If $|t|=|j|$ , there is nothing to prove, so assume $|t|>|j|$ . Consider $\Res_{\Delta} (U)$, in which $U^V$ corresponds to a singular subspace $V'$ of dimension $|t| - |j| -1\geq 0$. By the previous fact, there is a singular subspace $W'$ opposite $V'$ that avoids the set corresponding to $\mathcal{F'}$, unless $|t|=n-1$, as then it avoids the set corresponding to $\mathcal{F}\cap(\Omega\setminus\Omega_{t'})$ and intersects each member of the set corresponding to $\mathcal{F}\cap\Omega_{t'}$ in exactly a point. Now let $Z$ be the singular subspace in $\Delta$ through $U$ corresponding to $W'$ and let $Z'$ be a subspace in $Z$ complimentary to $U$. If $W'$ avoids the set corresponding to $\mathcal{F}'$, then $Z'$ avoids $\mathcal{F}'$. If for some $F \in \mathcal{F'}$, the subspace corresponding to it intersects $W'$ in exactly a point of it, then $Z$ also contains exactly a point of $F$. Only if some $F \in \mathcal{F}'$ contains $U$, we are not able to choose  $Z'$ such that it avoids $\mathcal{F}'$. This shows the lemma.
\end{enumerate}
\end{proof}

\subsection{The construction of an element of $\mathsf{N}(J_1,J_2)$.}

We define $\mathsf{N}_{(x)}(J_1,J_2)$ to be the subset of $\mathsf{N}(J_1,J_2)$ consisting of those  $i$-spaces $I$ for which $K_1 \cap S = K_2 \cap S$ has dimension $x$. As mentioned, an element $I \in \mathsf{N}(J_1,J_2)$ consists of the building bricks $K_1,K_2,A_1,A_2,B_1,B_2$. We now want to give a construction for \emph{some} members of $\mathsf{N}_{(x)}(J_1,J_2)$ as build up from these buildings bricks. These members  will then help us to narrow down the mutual positions for a round-up quadruple. To make this a powerful tool, we need `many' elements in $\mathsf{N}_{(x)}(J_1,J_2)$, `many' in the sense that we need to be able to choose our building bricks such that they avoid certain subspaces, cf. Fact~\ref{lem4a}. Yet we can limit ourselves to `easy' members, `easy' in the sense that, for any $X$ in $\{K,A,B\}$, we choose $X_1$ and $X_2$ such that $X^-$ is as large as possible.  This part is rather technical, but it is a key element of the proof.
\par\bigskip
\textbf{Some assumptions} $-$ We list assumptions on the parameters that we use throughout the construction and in the rest of the proof.
\begin{itemize} 
\item In view of Subsection~\ref{incidence}, we may assume $k < \min\{|i|,|j|\}$ and, if $\Delta$ is hyperbolic, $k \neq n-2$.
\item As mentioned at the beginning of this section, either $|i| \leq |j|$ or $|i| = |j|+|a|+1$. \end{itemize}
Our construction depends on the mutual position of $J_1$ and $J_2$ and also on $x$. The cases of interest turn out to be those with $x=k$ in case $s \geq k$, $x=k-1$ in case $s \geq k \geq 0$ and $x=s$ if $s < k$ (note that also in the last case, $k \geq 0$). So we restrict our attention to those cases, despite the fact that a construction equal or similar to ours would also work for other values of $x$.  
We first suppose $\Gamma = \Gamma_k^{\ell}$ for a non-trivial Weyl graph  $\Gamma_k^\ell$. Afterwards we deal with the other types of graphs, which do not need much additional effort as their adjacency imposes less constraints. Moreover, we first study the case in which $\Delta$ is not hyperbolic and afterwards we summarise the differences. At the conclusion of this section, we summarise our findings.

\par\bigskip
\begin{constr}\label{con} \em Our construction consists of three steps. In the first step, we examine the possibilities for $K_1$ and $K_2$. Then, given $\<K_1,K_2\>$, we do the same for $A_1$ and $A_2$ and afterwards, taking into account $\<K_1,K_2,A_1,A_2\>$, we do this again for $B_1$ and $B_2$. In each step we verify some ``avoiding properties''.
Figure~\ref{twojspaces} depicts an element of $\mathsf{N}_{(s)}(J_1,J_2)$ as generated by its different ``building bricks'', with respect to $J_1$ and $J_2$ respectively.

\begin{figure}[h]\centering \includegraphics[scale=0.28]{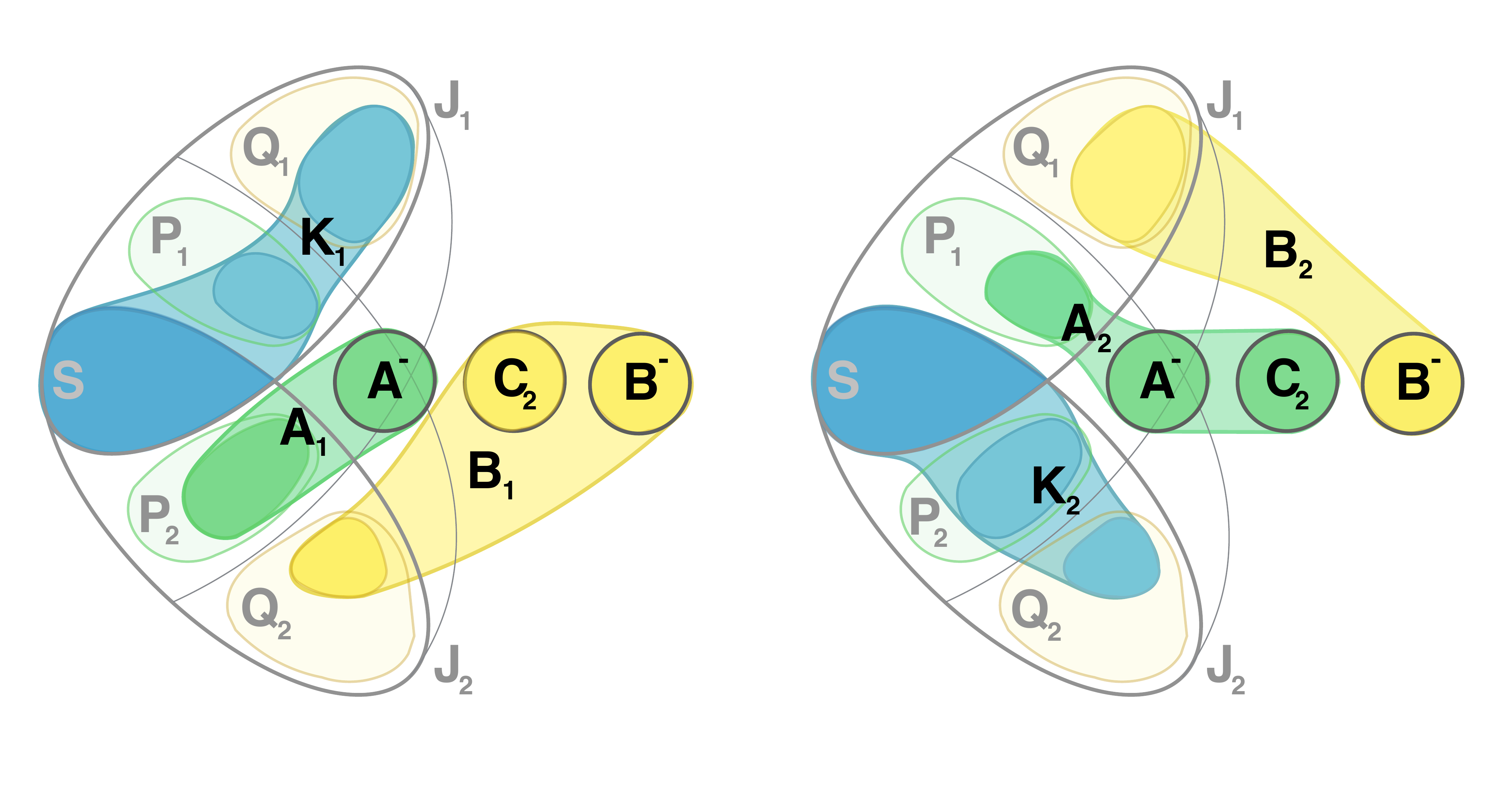} \caption{An element of $\mathsf{N}_{(s)}(J_1,J_2)$, depicted w.r.t.\ $J_1$ (left) and $J_2$ (right).}  \label{twojspaces}  \end{figure}

\par\bigskip

\begin{subcon}[Selection of $K_1$ and $K_2$]\label{con1} \em We will choose $K_c$ such that $K_c=\<S\cap K_c, K_c\cap P_c, K_c \cap Q_c\>$, i.e., we choose subspaces of $P_c$ and $Q_c$ which, together with the part of $K_c$ chosen in $S$, generate $K_c$. The parts $K_c \cap P_c$ and $K_c \cap Q_c$ need to be chosen carefully, as $\<K_1,K_2\>$ has to be singular and, moreover, $K_c \cap P_c \subseteq A_{c'}$ and $K_c \cap Q_c \subseteq B_{c'}$. Our method depends on $x$.

\boldmath{[$x=k$]} \unboldmath In this case, $K_1=K_2$ is simply any $k$-space in  $S$.

\boldmath{[$x=k-1$]} \unboldmath Now $K_c=\<K^-,z_c\>$ with $z_c \in J_c \setminus S$ and $z_1 \perp z_2$. If $-1 \notin \{a,b\}$, we can choose any pair of collinear points and such a pair always exists if $(p^*,q^*)\neq(-1,0)$. If $a=-1$ then necessarily $z_c \in Q_c$, which is possible unless $q^* \leq  0$. Likewise, if $b=-1$ then $z_c \in P_c$, which is possible unless $p^*=-1$. 

\boldmath{$[x=s]$}  \unboldmath Here we still have to choose two collinear $(k-s-1)$-spaces to complete $K_1$ and $K_2$. We call a 4-tuple $(p_1,q_1;p_2,q_2)$ of integer numbers \textbf{allowed values} if we can find an element of $\mathsf{N}_{(s)}(J_1,J_2)$ for which $K_c$ is such that $\dim(K_c \cap P_c) = p_c$, $\dim(K_c \cap Q_c)=q_c$ (still assuming $K_c=\<S,K_c\cap P_c,K_c\cap Q_c\>$). This $4$-tuple is sometimes abbreviated by  $(p_c,q_c)_c$ and, in case $p_1=p_2=p$ and $q_1=q_2=q$ for some $p$ and $q$, we will sometimes write $(p_c,q_c)_c=(p,q)$.  Note that this definition does not depend on the choices of $P_c$ and $Q_c$ in $\mathsf{P}_c$ and $\mathsf{Q}_c$, respectively, as they all play the same role. The following constraints apply to $(p_c,q_c)_c$. 
\begin{equation}\label{cond1} -1 \leq p_c \leq \min\{a,p^*\}, \qquad -1 \leq q_c \leq \min\{b,q^*\} \end{equation}

Furthermore, as $Q_1\cap K_1$ needs to be collinear with $Q_2 \cap K_2$, the latter needs to be contained inside $\proj_{Q_2} (Q_1 \cap K_1)$, resulting in the condition 
\begin{equation} \label{cond2} q_1+q_2+1 \leq q^*. \end{equation}

We also have to keep in mind that $p_c$ and $q_c$ are related by 
\begin{equation}\label{rel} p_c+q_c+1 = k- x -1. \end{equation}

Now let $(p_c,q_c)_c$ be values satisfying (\ref{cond1}), (\ref{cond2}) and (\ref{rel}) and let $K_1$ and $K_2$ be such that $\dim(K_c \cap P_c)=p_c$ and with $K_1 \cap Q_1$ and $K_2\cap Q_2$ collinear subspaces of dimensions $q_1$ and $q_2$, respectively. By choosing $A_c$ and $B_c$, i.e., by finishing the construction and obtaining an element $I=\<K_1,K_2,A_1,A_2,B_1,B_2\>$ in $\mathsf{N}_{(s)}(J_1,J_2)$, we will show that $(p_c,q_c)_c$ are indeed allowed values, so it will then follow that $(p_c,q_c)_c$ are allowed values if and only if (\ref{cond1}), (\ref{cond2}) and (\ref{rel}) hold. To see that there are values $(p_c,q_c)_c$ satisfying (\ref{cond1}), (\ref{cond2}) and (\ref{rel}), take any $I\in \mathsf{N}(J_1,J_2)$. The $k$-spaces $I\cap J_1$ and $I\cap J_2$ are generated by subspaces $\overline{S}=I \cap S$, $\overline{P}_c \subseteq \mathsf{P}_c$ and $\overline{Q}_c \subseteq \mathsf{Q}_c$ of respective dimensions $\overline{s}$, $\overline{p}_c$ and $\overline{q}_c$ with $k=\overline{s}+\overline{p}_c+\overline{q}_c+2$. Clearly, $(\overline{p}_c,\overline{q}_c)_c$ satisfies (\ref{cond1}) and (\ref{cond2}) and $\overline{p}_c + \overline{q}_c+1 = k - \overline{s} - 1 \geq k - s -1$. This means that there are $p_c$ and $q_c$ with $-1 \leq p_c \leq \overline{p}_c$ and $-1 \leq q_c \leq \overline{q}_c$ such that $p_c+q_c+1 = k-s-1$, i.e., (\ref{rel}) is satisfied for $x=s$. Furthermore, as $(\overline{p}_c,\overline{q}_c)_c$ satisfies (\ref{cond1}) and (\ref{cond2}), so does $(p_c,q_c)_c$.
\end{subcon}
We encounter our first ``avoiding property''. Suppose $p_1 \leq p_2$; if not, we switch the roles of $J_1$ and $J_2$.

\begin{subcon}[$(K_1, K_2)$-avoiding]\label{con2} \em
\textit{Suppose $x=s<k$ and let  $\mathcal F$ be a finite set of $j$-spaces intersecting $J_1$ in at most a subspace of dimension $s$. We can choose $K_1$ and $K_2$ in such a way that  $\dim(F \cap \<K_1,K_2\>)<k$ for each $F \in \mathcal F$, unless if $S \subseteq F$ for some $F \in \mathcal{F}$ or if $(p_c,q_c)_c=(p^*,b)$. In the latter case, each member $I \in \mathsf{N}(J_1,J_2)$ contains $P=\<S,P_1,P_2\>$.} 

To show this, we start from $k$-spaces $K_c$ in $J_c$ such that $(p_c,q_c)_c$ takes allowed values, i.e., such that there is an element of $\mathsf{N}_{(s)}(J_1,J_2)$ containing $\<K_1, K_2\>$, or equivalently, values satisfying conditions (1), (2) and (3) (as explained above we prove this later, independently of this). Suppose that $\dim(F \cap \<K_1,K_2\>) \geq k$ for some $F \in \mathcal{F}$. A dimension argument yields that $\dim(F \cap K_1)=\dim(F\cap J_1) = s$. If $S \neq F \cap J_c$ for all $F \in \mathcal{F}$, we show that, when $(p_c,q_c)_c \neq (p^*,b)$, we can choose $K_1$ and $K_2$ such that $\dim(K_c \cap F) \leq s-1$ for all $F\in\mathcal{F}$ and hence $\dim(\<K_1,K_2\> \cap F) <k$. If $(p_c,q_c)_c=(p^*,b)$, it readily follows that there are no other allowed values and that each $I \in \mathsf{N}(J_1,J_2)$ has to contain $\<S,P_1,P_2\>=P$.

We want to replace $K_1$ by a $k$-space $K'_1$ through $S$, possibly by also replacing $K_2$ by a $k$-space $K'_2$ with $S \subseteq K'_2 \subseteq \proj_{J_2} (Q_1 \cap K'_1)$.  Recall that $K_c=\<S, K_c \cap P_c, K_c \cap Q_c\>$.  We use the following (independent) actions. 

\begin{itemize}
\item[$\mathsf{(SP)}$] The subspace $K_1 \cap P_1$ may be replaced by any other $p_1$-space $\overline{P}_1$ in $P_1$.
\item[$\mathsf{(SQ)}$] The subspace  $Q_1$ may be replaced by any other $q^*$-space $Q'_1$ in $\mathsf{Q}_1$ and $K_1\cap Q_1$ may be replaced by any other $q_1$-space $\overline{Q}_1$ in $Q'_1$, if we replace the subspace $K_2 \cap Q_2$ by any $q_2$-space $\overline{Q}_2$ in $\proj_{Q_2} (\overline{Q}_1)$.
\end{itemize}

Then we put $K'_1=\<S,\overline{P}_1,\overline{Q}_1\>$ and $K'_2=\<S,K_2\cap P_2,\overline{Q}_2\>$ (and hence $\overline{P}_1=K'_1\cap P_1$, $\overline{Q}_c=K'_c\cap Q'_c$). Note that replacing $P_1$ by any other $p^*$-space $P'_1$ in $\mathsf{P}_1$ would not make a difference, as $\<S,K_1 \cap P_1\> = \<S,K_1 \cap P'_1\>$.

First note that it is possible that $K_1$ contains $\mathsf{P}_1$ as it can contain $\<S,P_1\>$ (at least if $q^* \geq 0$) but it is not possible that $K_1$ contains $\mathsf{Q}_1$ as then it would contain $J_1$, contradicting $k<j$. We may suppose that $\dim(F \cap J_c)=s$ for all $F \in \mathcal{F}$, because as noted before, if $\dim(F \cap J_c) < s$, then automatically $\dim(F \cap \<K_1,K_2 \>) <k$. For each $F \in \mathcal{F}$, set $s_F:=\dim(F \cap S)$, $p_F:=\cod_{F \cap \<S,P_1\>}(F \cap S)$ and $q_F:=\cod_{F \cap J_1}(F \cap \<S,P_1\>)$. Denote the subsets $\{F \in \mathcal{F} : p_F \geq 0\}$ and $\{F \in \mathcal{F} : q_F \geq 0\}$  by $\mathcal{F}_p$ and $\mathcal{F}_q$, respectively.  
\begin{itemize}
\item \textit{First suppose that $\mathcal{F}_q$ is non-empty.} Then we may assume that $q_1 \geq 0$, as otherwise $K_1$ cannot contain $F \cap J_1$ for $F \in \mathcal{F}_q$. We use $\mathsf{(SQ)}$ to replace $Q_1$ such that $K_1$ does not contain $F \cap J_1$ for each $F \in \mathcal{F}_q$. Indeed, since $K_1$ cannot contain $\mathsf{Q}_1$, we can always make sure that $K_1 \cap \mathsf{Q}_1$ avoids a point of each subspace of $\{F \cap \mathsf{Q}_1 \mid F \in \mathcal{F}_q\}$, as $\mathcal{F}_q$ is finite. 
\item \textit{Suppose next that $\mathcal{F}_p$ is non-empty.} If $p_1 < p_F$ for some $F \in \mathcal{F}_p$, then clearly $K_1$ cannot contain $F \cap J_1$, hence we may assume that $p_1 \geq p_F \geq 0$. Now, if $p_1 < p^*$, we can use $\mathsf{(SP)}$ to change $P_1 \cap K_1$ such that $\<S,P_1 \cap K_1\>$ does not contain $F \cap \<S,P_1\>$ (not containing a point from each subspace of $\{F \cap \mathsf{P}_1 \mid F \in \mathcal{F}_p\}$ suffices).  

If $p_1 = p^*$ and $q_1 <b$, we replace $K_1$ by another $k$-space $K'_1$ through $S$, one for which $(p'_1,q'_1) = (p^*-1,q_1+1)$. We claim that these are allowed values, i.e., that conditions (\ref{cond1}), (\ref{cond2}) and (\ref{rel}) are satisfied. First note that $p_1 \leq p_2$ implies $p_1=p_2=p^*$ and $q_1=q_2$. Since $p_1 \geq 0$, $p_1-1 \geq -1$ and by assumption $q_1+1 \leq b$. If $(q_1+1)+q_2+1 > q^*$, then $q_1 +q_2 +1=q^*$. However, this would imply $i = s +2p^*+ 2q_2 + (a-p^*-1) + (b-q_2-1) + 6 = j +a+1 +(b-q_2) \geq j$, whereas we know that $i \leq j$  or $i=j+a+1$. Either way, this implies $b-q_2=0$ and then, since $q_1=q_2$, this contradicts our assumption that $q_1 < b$ and the claim holds. Since $p'_1 < p^*$, we can again apply the above argument (if necessary).
\end{itemize}
Since $s_F + p_F+q_F = s$ and $s_F < s$, we have $p_F + q_F +2 \geq 0$, so $\mathcal{F}_p \cup \mathcal{F}_q = \mathcal{F}$. Hence, this shows that, if $(p_c,q_c)_c \neq (p^*,b)$ we can choose $K_1$ and $K_2$ such that $\dim(F \cap \<K_1,K_2\>)<k$ for all $F \in \mathcal{F}$, under the assumption that $S \neq F \cap J_c$ for any $F \in \mathcal{F}$. 

\begin{rem} \label{fq} \em Note that for each $F \in \mathcal{F}$, also if $(p_c,q_c)_c=(p^*,b)$, we can make sure that $\dim(\<K_1,K_2\>\cap F) <k$, as long as $S \nsubseteq F$. \end{rem}

 We continue with our construction. First note that the dimensions $(p_c,q_c)_c$  can also be used in the case $x=k-1$. In this case, $p_c$ and $q_c$ belong to $\{-1,0\}$,  satisfy conditions (\ref{cond1}) and (\ref{cond2}) and condition (\ref{rel}) with $x=k-1$ becomes $p_c+q_c+1=0$. In the sequel, we handle the cases $x=k-1$ and $x=s$ simultaneously as they behave the same with respect to choosing $A_1$, $A_2$, $B_1$ and $B_2$.
\end{subcon}

\begin{subcon}[Selection of $A_1$ and $A_2$]\label{con3} \em As $A_1$ and $A_2$ have to be collinear with $J_1$ and $J_2$, they are automatically collinear with $\<K_1,K_2\>$, the part of our $i$-space that has been constructed up to now. Denote by $\mathcal A_t$ the set of all $t$-spaces $T$ belonging to  $(J_1^\perp \cap  J_2^\perp)\setminus(J_1\cup J_2)$. By Fact~\ref{lem4a}$(i)^*$, $\mathcal A_t$ is certainly nonempty for all $t$ with $-1 \leq t \leq a$, since $a = \ell-j-1 \leq n-j-2$.

\boldmath{[$x=k$]} \unboldmath Take $A_1=A_2 \in$ $\mathcal A_a$ arbitrarily. 

\boldmath{[$x\in\{k-1,s\}$]} \unboldmath In these cases, $K_2 \cap P_2 \subseteq A_1$ and $K_1 \cap P_1 \subseteq A_2$. Assume $p_1 \leq p_2$ and let $a'=a-p_2-1$. As we prefer $A^-$ to be as large as possible, we choose it in the set $\mathcal A_{a'}$, which is nonempty as $-1 \leq a' \leq a$. Then $A_1=\<A^-,K_2 \cap P_2\>$. Put $a^p := \dim(A^- \cap P)$.

If $p_1=p_2$, we also put $A_2=\<A^-,K_1 \cap P_1\>$; if $p_1 <p_2$, we still need a $(p_2-p_1-1)$-space $C$  that together with $A^-$ and $K_1\cap P_1$ will generate the $a$-space $A_2$, as $a = (a-p_2-1)+p_1+(p_2-p_1-1)+2$. This subspace $C$ has to be collinear with $J_2$  and semi-opposite $J_1$, and also needs to be collinear with $A^-$. We define $J'_c = \<J_c, P_{c'}\cap K_{c'}, A^-\>$. A $(p_2-p_1-1)$-space $C$ collinear with $J'_2$ and semi-opposite $J'_1$ will be collinear with $A^-$, $K_1 \cap P_1$ and $J_2$ and will be semi-opposite $J_1$. By Fact~\ref{lem4a}$(iii)$, such a subspace exists if $\dim({J'_2}^{J'_1}) - \dim(J'_2)-1 \geq p_2-p_1-1$. One can verify that $\dim({J'_2}^{J'_1})-\dim(J'_2) -1 = p^* - (p_1 +a^p+1)-1$ (note that $\<A^- \cap P, P_1 \cap K_1,J_2\> \cap J_1$ has dimension $s+(p_1 + a^p +1)+1$) . Now $p_2-p_1-1 \leq p^*-p_1-a^p-2$ if and only if $a^p \leq p^*-p_2 -1$, which is true. Note that equality holds if and only if $a^p=p^*-p_2-1$. We set $A_2 = \<A^-, K_1 \cap P_1, C\>$. \par
\bigskip

We encounter some more avoiding properties. Recall that $P = \<S,P_1,P_2\>$. 
\end{subcon}

\begin{subcon}[$(A\cap P)$-avoiding] \label{con4} \em
\textit{Let $\mathcal{F}$ be a finite subset of $\Omega_j$, all of whose members intersect $J_1$ and $J_2$ in subspaces of dimension at most $s$. Suppose $a^p:=\dim(A^-\cap P) \geq 0$.
Then: \vspace{-1em}
\begin{enumerate}[$(i)$]
\item If $x=k \geq 0$,  we can choose $A \cap P$ such that $\dim(\<K, A \cap P\> \cap F) < k$ for each $F \in \mathcal{F}$, unless if either $K \subseteq F$, or, if $a^p=p^* \geq 0$ and $\dim(F \cap P)=p^*+s+1$ for some $F \in \mathcal{F}$. If $K \subseteq F$ for some $F \in \mathcal{F}$, then we can always choose $A\cap P$ such that $\<K, A \cap P\> \cap F=K$ for all $F \in \mathcal{F}$ with $K \subseteq F$.
\item If $x=s<k$, we can choose $A^-\cap P$, $K_1\setminus S$ and $K_2 \setminus S$ such that $\dim(\<K_1,K_2, A^- \cap P\> \cap F) < k$ for each $F\in\mathcal{F}$, unless if either $\dim(F \cap \<K_1,K_2\>) \geq k$ for some $F \in \mathcal{F}$, or, if $\dim(F \cap P)=p^*+s+1$ for some $F \in \mathcal{F}$, $a^p=p^*-p_2-1 \geq 0$ and $-1\in\{q^*,b\}$.  If $\dim(\<K_1,K_2\> \cap F)=k$ for some $F \in \mathcal{F}$, then we can always choose $A^-\cap P$ such that $\<K_1,K_2, A^- \cap P\> \cap F=\<K_1,K_2\> \cap F$ for all $F \in \mathcal{F}$ with $\dim(\<K_1,K_2\> \cap F)= k$.
\end{enumerate}}

We now verify that this is true. Let $F$ be an arbitrary member of $\mathcal{F}$ and put $a^p = \dim(A^-\cap P)$.
\begin{enumerate}[$(i)$]
\item Clearly, if $K \subseteq F$, then $\dim(\<K, A \cap P\> \cap F) \geq k$. We start with the first assertion, so suppose $K \nsubseteq F$ and suppose $\dim(\<K, A \cap P\> \cap F) \geq k$ for all choices of $A\cap P$. This means that $\dim(F \cap P) + \dim(\<K,A\cap P\>) - k \geq \dim(P)$. Now $\dim(P)=2p^*+s+2$, $\dim(\<K,A\cap P\>) = k + a^p + 1 \leq k + p^*+1$ and $\dim(F \cap P) \leq p^*+s+1$ (as otherwise $\dim(F\cap J_c) > s$). We obtain that $ \dim(\<K,A \cap P\>) + \dim(F \cap P) - k \leq (k+p^* +1) + (p^*+s+1) -k = \dim(P)$ and equality only holds when $\dim(F\cap P)=k+p^*+1$ and $a^p=p^*$. We conclude that only when these conditions are fulfilled, it is not possible for $F \cap P$ and $\<K, A \cap P\>$ to intersect in a subspace of dimension strictly less than $k$. On the other hand, this also reveals that if $\dim(K \cap F)=k$, it is always possible to choose $A \cap P$ such that $K \cap F = \<K,A \cap P\> \cap F$. 

\item Let $K_1$ and $K_2$ be such that $(p_c,q_c)_c$ are allowed values. We consider the singular subspace $\overline{P}:=\<P, K_1 \cap Q_1, K_2 \cap Q_2\>$, which has dimension $2p^* +s + q_1 + q_2 +4$. As before, we assume $q_1 \geq q_2$. Suppose again that $\dim(\<K_1,K_2, A^- \cap P\> \cap F) \geq k$ for all choices of $A^-\cap P$, while assuming $\dim(\<K_1,K_2\> \cap F)<k$ for each $F \in \mathcal{F}$.

Hence, as above, we then conclude $\dim(F \cap \overline{P}) + \dim(\<K_1,K_2,A^-\cap P\>) - k \geq \dim(\overline{P})$. Now $\dim(\<K_1,K_2,A^-\cap P\>)= k+a^p + p_2 +q_2+3 \leq k +p^* + q_2 +2$ and $\dim(F \cap \overline{P}) \leq p^*+s+q_2 +2$ (otherwise $\dim(F\cap J_1) > s$). This yields $\dim(F \cap \overline{P}) + \dim(\<K_1,K_2,A^-\cap P\>) - k \leq (p^*+s+q_2 +2)+(k+p^* +q_2+2)-k = \dim(\overline{P}) - q_1 +q_2 \leq \dim(\overline{P})$ and equality only holds when $\dim(F \cap \overline{P})=p^*+s+q_2+2$, $a^p=p^*-p_2-1$ and $q_1=q_2$. Note that this, like in the previous item, reveals that it is always possible to choose $A^-\cap P$ such that $\<K_1,K_2\> \cap F = \<K_1,K_2,A^- \cap P\> \cap F$  for all $F \in \mathcal{F}$ with $\dim(\<K_1,K_2\> \cap F)=k$. 

Now, we claim that there are allowed values for which $q_1 > q_2$ if and only if $-1\notin\{b,q^*\}$. So suppose that $p_1=p_2$ and $q_1=q_2$. Then $(p_1-1,p_2;q_1+1,q_2)$ is not a tuple of allowed values if and only if either $p_1=p_2=-1$ or $q_1=b$ or $q_1+q_2+1=q^*$; likewise, $(p_1+1,p_2;q_1-1,q_2)$ is not a tuple of allowed values if and only if either $p_1=p_2=p^*$ or $p_1=p_2=a$ or $q_1=q_2=-1$. First note that if $p_1=p_2=p^*$ or if $p_1=p_2=a$, then necessarily $a^p=-1$ so this contradicts our assumptions. Hence we may assume that $q_1=q_2=-1$ (otherwise the second tuple would be allowed after all). Now for the first tuple not to be allowed, we should have $p_1=p_2=q_1=q_2=-1$, or $b=q_1=q_2=-1$, or $q^*=q_1+q_2+1=-1$. The first possibility would imply that $k=s$, which contradicts $s < k$, so we conclude that $-1 \in \{b,q^*\}$, showing the claim. However, we also need to make sure that this does not conflict with the $\<K_1,K_2\>$-avoiding used when proving property~\ref{con2}. Indeed, there was one situation in which we changed the values $(p_1,q_1;p_2,q_2)$, namely from $(p^*,q_1;p^*,q_2)$ to $(p^*-1,q_1+1;p^*,q_2)$. Yet $p_2=p^*$ here, so $a^p=-1$, while we assume $a^p \geq 0$.

We conclude that only when $-1\in\{q^*,b\}$ (and hence $q_1=q_2=-1$), $a^p=p^*-p_2-1$ and $\dim(F \cap \overline{P})=p^*+s+q_2+2$, we cannot choose $K_1\setminus S$, $K_2\setminus S$ and $A^- \cap P$ such that $\dim(\<K_1,K_2, A^- \cap P\> \cap F) < k$.
\end{enumerate}

Before we continue with our construction, we give one more avoiding property, concerning the selection of $A_1$ and $A_2$. In constructing these, we used Facts~\ref{lem4a}$(i)^*$ and~\ref{lem4a}$(iii)$, and as we still assume that $\Delta$ is not hyperbolic, these facts also give the following property:
\end{subcon}
\begin{subcon}[$(A\setminus P)$-avoiding]\label{con5} \em
\textit{Let $\mathcal{F}$ be a finite subset of $\Omega_j$ such that $\dim(F \cap J_c) \leq s$ for all $F \in \mathcal{F}$. Then the parts of $A_1$ and $A_2$ outside $P$ can be chosen such that $\<K_1,K_2, A_1,A_2\> \setminus \<K_1,K_2,A\cap P\>$ avoids $\mathcal{F}$.}

\end{subcon}
\begin{subcon}[Selection of $B_1$ and $B_2$]\label{con6} \em Let $I^*$ denote $\<K_1,K_2,A_1,A_2\>$. Possibly, $\dim(I^*)=i$ and nothing more needs to be done. So suppose $\dim(I^*)<i$. As $B_1$ and $B_2$ have to be collinear with $I^*$, we look for them in $\Delta'=\Res_{\Delta}(I^*)$. Now, $\dim(J_1 \cap I^*)=k$ by definition and $\dim(\proj_{J_1}(I^*)) = j-q_1-1$ since $\<K_2 \cap Q_2,C\>$, which has dimension $q_1$ (recall $p_1 + q_1 = p_2 +q_2$), is a subspace of $I^*$ maximal with the property of being semi-opposite $J_1$. Hence, in $\Delta'$, $J_1$ corresponds to a subspace $J'_1$ of dimension $(j-q_1-1)-k-1$. Likewise, $\dim(J_2 \cap I^*)=k$ and $\dim(\proj_{J_2}(I^*)) = j-q_1-1$ as a maximal subspace of $I^*$ semi-opposite $J_2$ is $K_1 \cap Q_1$, which has dimension $q_1$. Therefore, $J_2$ corresponds to a subspace $J'_2$ in $\Delta'$ of dimension $(j-q_1-1)-k-1$ too. 

As before, we choose $B^-$ as large as possible. If $x=k$, we aim for a $b$-space $B_1=B_2$ semi-opposite $J_1$ and $J_2$. If $x \neq k$, then $I^* \cap B_1 = \<K_2 \cap Q_2,C\>$ and $I^* \cap B_2 = K_1 \cap Q_1$. Both subspaces are $q_1$-dimensional, as $q_2 + (p_2-p_1-1)+1 = q_1$. So in this case we need a $(b-q_1-1)$-space $B^-$ to define $B_1=\<K_2\cap Q_2, C, B^-\>$ and $B_2= \<K_1 \cap Q_1, B^-\>$. This can be achieved as follows. Put $q_1=-1$ in case $x=k$ and define $b'=b-q_1-1$. In $\Delta'$, we select two arbitrary $b'$-spaces $T_1$ and $T_2$ in $J'_1$ and $J'_2$, respectively. This is possible as $b'=b-q_1-1 \leq (j-k-1)-q_1-1$ (recall that $b=i+j-k-\ell-1$). By Fact~\ref{lem5a}$(ii)$, we know that there is some $b'$-space $B'$ in $\Delta'$ which is opposite $T_1$ and $T_2$, hence the subspace of $\Delta$ corresponding to $B'$ is precisely a member of $\mathsf{N}_{(x)}(J_1,J_2)$. 
\end{subcon}

\begin{rem}\label{b} $-$ \em The way we select $B^-$ has some nice features.
\begin{itemize}
\item
As we choose $B^-$ in a residue, each $b'$-space in $I\setminus\<K_1,K_2,A_1,A_2\>$ is semi-opposite $J_1$ and~$J_2$. 
\item The above implies a generalisation of Fact~\ref{lem4a}$(ii)$, stated here informally and not including the case where $\Delta$ is hyperbolic: For each finite set of subspaces of dimensions \textit{at least} $b$, there is a $b$-dimensional subspace semi-opposite them all.
\item The subspaces $T_1$ and $T_2$ can be chosen in $J'_1$ and $J'_2$ wherever we want, a feature we will exploit at some point. 
\end{itemize}
\end{rem}

We end this construction with one last avoiding property, which again follows immediately from Fact~\ref{lem4a}$(ii)$ and our assumption that $\Delta$ is not hyperbolic. 

\begin{subcon}[$(B_1,B_2)$-avoiding]\label{con7} \em
\textit{Let $\mathcal{F}$ be a finite subset of $\Omega_j$. Then $B_1$ and $B_2$ outside $\<K_1,K_2\>$ can be chosen such that $I \setminus \<K_1,K_2,A_1,A_2\>$ avoids $\mathcal{F}$.}
\end{subcon}

\textbf{Intermediate summary} $-$ If $x=k$, we have $I=\<K_c,A_c,B_c\>$ with $X_1=X_2$ for all $X\in\{K,A,B\}$ and clearly, $i=k+a+b+2$ and $I\in \mathsf{N}_{(k)}(J_1,J_2)$. If $x \in\{k-1,s\}$, we have $I=\<K_1,K_2,A_1,A_2,B_1,B_2\>$ with $K_c = \<S, K_c \cap P_c, K_c \cap Q_c\>$, $A_1=\<A^-,K_2 \cap P_2\>$, $A_2=\<A^-,K_1 \cap P_1, C\>$, $B_1=\<B^-,K_2\cap Q_2,C\>$ and $B_2=\<B^-,K_1\cap Q_1\>$ with notation as before (recall that $C$ is a subspace collinear with $J_2$ and semi-opposite $J_1$). In each stage, we checked that $\dim(K_c)=k$, $\dim(A_c)=a$ and $\dim(B_c)=b$, so the resulting singular subspace is indeed an $i$-space in $\mathsf{N}_{(x)}(J_1,J_2)$. If $x=s$, then the structure of the resulting $i$-spaces can be seen in Figure~\ref{twojspaces}.

Note that it now follows that all values of $(p_c,q_c)_c$ satisfying (\ref{cond1}), (\ref{cond2}) and  (\ref{rel}) are indeed allowed values. For any pair of collinear $k$-spaces in $J_1$ and $J_2$ satisfying those conditions, $\<K_1,K_2\>$ can, by means of the construction, be extended to an $i$-space $I \in \mathsf{N}_(J_1,J_2)$. If $s<k$, then we will sometimes call a pair of $k$-spaces  \textbf{allowed} $\mathbf{k}$\textbf{-spaces} if $x=s$, i.e., if $\dim(K_1 \cap K_2) =s$, abbreviating ``$k$-spaces as obtained in the construction in the case when $x=s<k$''. 

\par\bigskip

\end{constr}

\textbf{The hyperbolic case} $-$ 
We list the differences that occur when $\Delta$ is hyperbolic. 

$\bullet$ \textit{The selection of $K_1$ and $K_2$.} It is easy to see that choosing $K_1$ and $K_2$ can be done in the same way, once we know that we do not have to do anything special for $A_1, A_2, B_1$ and $B_2$ (if we would have to, then $K_1 \setminus S$ and $K_2 \setminus S$ could need special care). Also Avoiding Property~\ref{con2} remains unchanged.
\par\medskip
$\bullet$ \textit{The selection of $A_1$ and $A_2$.}  There could be a problem as the types $\mathsf{t}(J_c^I)=\mathsf{t}(\<J_c,A_c\>)$ should be correct. Recall that our definition is such that if $I \sim J_c$ then $\mathsf{t}(I^{J_c})=\ell$, and we put  $\mathsf{t}=\mathsf{t}(J_c^I)$. Now, our assumptions are such that $|\ell|<n-1$ unless $|j|=n-1$, in which case $j=\mathsf{t}$ and $i=\ell$ (and $\mathsf{t} = \ell$ if $b$ is odd and and $\mathsf{t}=\ell'$ if $b$ is even). When $|\ell| < n-1$, choosing $A_1$ and $A_2$ is not different than before; if $|j|=|\ell|=n-1$, then $a=-1$. 

Avoiding Property~\ref{con4} also holds in this case, but Property~\ref{con5} has one exception.

\begin{subcon}[$(A\setminus  P)$-avoiding]\label{con8} \em \emph{ If $\cod_{A^-}(A^- \cap P)=0$ and $\dim({J_1}^{J_2}) = n-2$, then possibly $\dim(F \cap \<K_1,K_2,A^-\>) = \dim(F \cap \<K_1,K_2,A^-\cap P) +1$ for some $F \in \mathcal{F}$ with $\dim(F \cap \<P,A^-\>)= p^*+s+2$. In all other cases, $\dim(F \cap \<K_1,K_2,A_1,A_2\>)=\dim(F\cap\<K_1,K_2,A^- \cap P\>)$.}

The first assertion follows immediately from Fact~\ref{lem4a}$(i)^o$ and~\ref{lem4a}$(i)^*$. According to Fact~\ref{lem4a}$(iii)$, a problem could occur in selecting $C$ if, with the notation used during the selection of $C$, $\dim({J'_2}^{J'_1})=n-1$. Note that $\<J'_2,C\>=\<J_2,A_2\>$ and hence $\dim(J'_2)=|\ell|-(p_2-p_1-1)-1$, furthermore, we have already verified that $\dim({J'_2}^{J'_1}) \leq \dim(J'_2)+(p_2-p_1-1)+1 = |\ell|$. So recalling that we assume $|\ell|<n-1$ except when $|j|=n-1$, we are fine (note that, if $|j|=|\ell|=n-1$, then there is no need a subspace $C$ collinear with $J_2$ and semi-opposite $J_1$). 
\end{subcon}
\par\bigskip
$\bullet$ \textit{The selection of  $B_1$ and $B_2$.} Also $B_1$ and $B_2$ can be chosen as before. 

Avoiding Property~\ref{con7} has one exception too. 

\begin{subcon}[$(B_1,B_2)$-avoiding]\label{con9} \em \emph{If $|i|=|j|=n-1$ then possibly $\dim(F \cap \<K_1,K_2,B^-\>) = \dim(F \cap \<K_1,K_2\>) +1$ for some $F \in \mathcal{F}$, in all other cases, the dimension of the intersection with $F$ does not increase for any $F \in \mathcal{F}$.}

Again with the notation as used during the selection of $B^-$, we run into problems when $T_c=J'_c$ and $J'_c$ is a MSS in $\Res_{\Delta}(I^*)$. The latter is a polar space of rank $n-(i-b'-1)-2$, in which $\dim(J'_c)=j-q_1-k-2$. Recalling that $b'=b-q_1-1$ and $b=i+j-k-|\ell|-1$, we obtain that $J'_c$ is a MSS if and only if $|\ell|=n-1$. Furthermore, $\dim(T_c)=\dim(J'_c)$ if and only if $b'=j-k-q_1-2$, i.e., if $b=j-k-1$. Hence $|i|=|j|$ and $a=-1$. Together with $|\ell|=n-1$, we hence obtain $|i|=|j|=n-1$.
\end{subcon}
\par\bigskip
\textbf{Final summary} $-$ Before we get to the other graphs (for which the construction follows almost immediately from this one), we give a brief overview of the selection procedure and all avoiding properties. Let $\mathcal{F}$ be a finite set of $j$-spaces intersecting $J_1$ and $J_2$ in subspaces of dimension at most $s$.
\begin{itemize}
\item If $x=s<k$ and if $x=k-1$, the values $(p_c,q_c)_c$ should satisfy conditions (\ref{cond1}), (\ref{cond2}) and (\ref{rel}). Property~\ref{con2} describes, for $x=s<k$, when we can choose $K_1$ and $K_2$ such that $\dim(\<K_1,K_2\> \cap F)<k$. 
\item Property~\ref{con4} says when we can choose $A^- \cap P$ such that $\dim(\<K_1,K_2, A^- \cap P\> \cap F)<k$.
\item In general, we can complete $A_1$ and $A_2$ (i.e., choose the part of $A^-$ outside $P$ and choose the subspace $C$) such that $\dim(\<A_1,A_2,K_1,K_2\> \cap F)$ equals $\dim(\<A^- \cap P, K_1,K_2\> \cap F)$; only when $\Delta$ is hyperbolic, this dimension increases by one. This is described in Properties~\ref{con5} and~\ref{con8}.
\item In general, we can complete  $B_1$ and $B_2$ (i.e., choose $B^-$) such that  $\dim(\<B_1,B_2, A_1,A_2,K_1,K_2\> \cap F) = \dim(\<A_1,A_2, K_1,K_2\> \cap F)$; only when $\Delta$ is hyperbolic, possibly this dimension increases by one. This is described in Properties~\ref{con7} and~\ref{con9}.
 \end{itemize}
 
\par\bigskip
\textbf{The other graphs} $-$ If $\Gamma \in \{\Gamma_k,\Gamma_{\geq k}\}$ then Construction~\ref{con} gets easier as we do not longer have to take into account the dimensions $\proj_{I}(J_1)$ and $\proj_{I}(J_2)$. We quickly go through the steps of the construction. Again, let $\mathcal{F}$ be a finite set of $j$-spaces intersecting $J_1$ and $J_2$ in subspaces of dimension at most $s$.
\par\medskip
The selection of $K_1$ and $K_2$ can be done similarly. If $x=k$, nothing changes; if $x=k-1$ we only need $(p^*,q^*)\neq(-1,0)$ in order to find a pair of collinear points in $J_1\setminus S$ and $J_2\setminus S$; if $x=s<k$ then condition (1) becomes
\begin{equation}\label{cond4} -1 \leq p_c \leq p^*, \qquad -1 \leq q_c \leq q^* \tag{1'}\end{equation}
and like before, we can show that conditions (\ref{cond4}), (\ref{cond2}) and (\ref{rel}) give allowed values $(p_c,q_c)_c$.
If $x=s<k$, we can choose $K_1$ and $K_2$ such that $\dim(\<K_1,K_2\> \cap F) < k$ unless some member $F$ of $\mathcal{F}$ contains $S$.
\par\medskip
If $k':=\dim(\<K_1,K_2\>)$, we now just need an arbitrary $(i-k'-1)$-space in $\Res_\Delta(\<K_1,K_2\>)$ avoiding the subspaces corresponding to $J_1$, $J_2$ and $\mathcal{F}$ to complete our $i$-space $I$. Only if $\Delta$ is hyperbolic and $|i|=n-1$, and hence also $|j|=n-1$, it could be the case that $\dim(I \cap F)= \dim(\<K_1,K_2\> \cap F) +1$ for some $F \in \mathcal{F}$, as the parity of the dimension of the intersection is fixed. In all other cases, there is no problem choosing $I\setminus \<K_1,K_2\>$ such that it avoids $\mathcal{F}$.

This concludes our construction.
\hfill $\blacksquare$

\section{The $k_{(\geq)}$-incidence graphs and  $(k,\ell)$-Weyl graphs for~$k~\geq~0$}\label{k>-1}
Let $\Gamma$ be one of $\Gamma_{\geq k}$, $\Gamma_k$ and $\Gamma_k^{\ell}$. We will assume $k \neq -1$ throughout this section. Yet, a couple of general lemmas also hold when $k=-1$, and this will be mentioned explicitly. On all other occasions, we assume $k \geq 0$. Recall that we assume that $|i| \leq |j|$ or either $|i|=|j|+|a|+1$. We need two more preliminary lemmas.

\begin{lemma}\label{lem0}
Let $U,V,W$ (possibly $V=W$) be singular subspaces of respective dimensions $a,a',a''$, with  $a \geq \max\{a',a'',0\}$, and such that $V \neq U \cap V = U \cap W\neq W$. Let $p \in U^\perp$ be a point not contained in $U \cup V \cup W$. If, for each $q \in U \setminus V$, the line $pq$ intersects $V$ or $W$ in a point $q'$, then $\dim (U \cap V)=a-1$. Moreover, $\<p,U\>$ contains at least one of $V,W$ and $a'=a=a''$. If, say, $W$ is not contained in $\<p,U\>$, then $pq \cap W=\mathsf{\emptyset}$  for all $q \in U \setminus V$.
\end{lemma}
\begin{proof}
Put $S=U \cap V$. Assume  for a contradiction that $\dim(S) < a-1$. Then there is some line $L$ contained in $U\setminus S$. Let $q_1,q_2,q_3$ be three points on $L$. At least two of the lines $pq_1,pq_2,pq_3$ must then intersect either $V$ or $W$, say $V$, by the condition. Hence the plane $\<p,L\>$ intersects $V$ in a line $L'$. But then the point $L\cap L'$ belongs to $U \cap V$, contradicting the fact that $L$ is disjoint from $S$. We conclude $\dim(S)=a-1$. 

The line joining $p$ and a point of $U\setminus S$ intersects $V\cup W$, clearly in a point not belonging to $S$. Hence at most one of $V,W$, say $W$, does not belong to $\<p,U\>$. One can easily see that if $pq$ intersects $W$ for some $q \in U\setminus S$, then $W$ belongs to $\<p,U\>$ as well and vice versa.
\end{proof}

\begin{lemma}\label{opp} Let $U,V,W,X$ be subspaces of the same dimension, with $V$, $W$ and $X$ opposite $U$. If each point collinear with $U$ and $V$ is also collinear with $W$ or $X$, then each point collinear with $U$ and $V$ is collinear with all four of them.  
\end{lemma} 
\begin{proof} Take any point $p$ collinear with $U$ and $V$.  As $V$, $W$ and $X$ are opposite $U$, we have $p \notin U \cup V \cup W \cup X$. Our assumptions imply that $p$ is collinear with $W$ or $X$, say $p \perp W$. The subspace $\<p,U\>$ then contains a point $q$ collinear with $X$. If $p=q$ we are done, so suppose $p \neq q$. As $q$ is collinear with $U$ and $X$, it has to be collinear with $V$ or $W$. But then the point $pq \cap U$ is collinear with $V$ or $W$, contradicting that they are opposite $U$. Hence $p=q$ after all and the lemma is proven.\end{proof}

In the next section, we deal with a special case that needs to be treated separately.

\subsection{Adjacency given by incidence}\label{incidence}
There are two types of graphs where adjacency is given by incidence: 
\begin{compactenum}[$-$]
\item When $\Delta$ is hyperbolic, the adjacency of the graph $\Gamma_{(n-1)',(n-1)'';(n-2),(n-1)'}^n(\Delta)$ coincides with the notion of being incident in the building $\Delta^b$ associated to $\Delta$.
\item The graphs $\Gamma_k^{\ell}$, $\Gamma_k$ and $\Gamma_{\geq k}$ with $|k|=\min\{|i|,|j|\}$ (hence in the first case also $\max\{|i|,|j|\}=|\ell|$), are identical, and their adjacency is given by incidence, i.e., containment made symmetric. This means that they are equal to $\mathsf{C}_{i,j}$ (recall that this is a restriction of the incidence graph of $\Delta$ to the types $i$ and $j$). In this special case  we can safely ignore our convention on $|i|$ and $|j|$ and just assume $|i| \leq |j|$.
\end{compactenum}

We readily have the following proposition.
\begin{prop}\label{propind}
Suppose $\Delta$ is a hyperbolic polar space and let $\Gamma=\Gamma_{(n-1)',(n-1)'';(n-2),(n-1)'}^n(\Delta)$. Then each automorphism of $\Gamma$ is induced by an automorphism of $\Delta^b$. Moreover, each automorphisms of $\Delta^b$ inducing an automorphism of $\Gamma$ is either type-preserving or a duality (in which case the bipartition classes of $\Gamma$ are switched). 
\end{prop}
\begin{proof}
Given $\Gamma$, we can construct the Grassmann graph $\mathsf{G}_{(n-1)'}(\Delta)$ by considering the bipartition class containing the $(n-1)'$-spaces and declaring two of them adjacent when they are at distance two in $\Gamma$. The proposition now follows from  Proposition~\ref{gras1}. 
\end{proof}

We now prove that any automorphism of $\mathsf{C}_{i,j}$ is induced by an element of $\Aut~\Delta$. The non-triviality of the graphs implies that $|i|\neq |j|$,  so as we assume $|i| \leq |j|$, we may assume $|i|<|j|$.
 
\begin{prop}\label{propk=i}
For all $i,j \in \mathsf{T}$ with  $|i| < |j|$, each automorphism of the graph $\mathsf{C}_{i,j}^n(\Delta)$ is induced by an element of  $\Aut~\Delta^b$. Moreover, each automorphisms of $\Delta^b$ inducing an automorphism of $\mathsf{C}_{i,j}^n(\Delta)$ is either type-preserving or a duality if $\Delta$ is hyperbolic and $|j|<n-1$, or, if $\Delta$ is of type $\mathsf{D}_4$, the automorphism can also be a $t$-duality if $\{i,j,t\}=\{0,3',3''\}$ (the biparts are switched) or a $t$-duality if $\{1,t\}=\{i,j\}$ (the biparts are not switched).\end{prop}

\begin{proof}
Put $\mathsf{C}=\mathsf{C}^n_{i,j}(\Delta)$. First assume that there is no type between $i$ and $j$, i.e., for no type $t$ there can be a $t$-space $T$ such that $I$, $J$ and $T$ are all incident (for clarity: this does include the case where $i=n-3$ and $j \in\{(n-1)',(n-1)''\}$). We define $\mathsf{C}'$ as a graph with vertex set $\Omega_j$ where two vertices are adjacent if they have a common neighbour in $\mathsf{C}$. Clearly,  $\mathsf{C}' \cong \mathsf{G}'_j$ and hence the assertion follows from Corollary~\ref{gras2}.

Next, assume there are types between $i$ and $j$. For any vertex $v$, let $\mathsf{B}(v)$ denote the set of vertices of $\mathsf{C}$ in the same bipart as $v$. Consider the poset $P_v=\{\mathsf{N}(v,w_1,\ldots,w_t) \,|\, w_1,\ldots,w_t \in \mathsf{B}(v), t\in\N\}$, ordered by inclusion. The length of a maximal chain in $P_v$ is precisely $j-i$, regardless of the bipart where $v$ is in. Indeed, an element in such a chain corresponds to a set of $i$-spaces or to a set of $j$-spaces incident with $v$ and some $m$-space, with $i \leq m \leq j$. We define $\mathsf{C}'$ as the graph having the elements of $P_v$ as vertices and adjacency given by containment made symmetric. Clearly, $\mathsf{C}' \cong \mathsf{C}_{[i,j]}^n(\Delta)$. Therefore it is clear that $\mathsf{C}'$ has a subgraph isomorphic to $\mathsf{C}^n_{j',j}(\Delta)$, where $j'$ is the biggest type smaller than $j$, which brings us back to the first case and hence concludes the proof. 
 \end{proof}
 
We now embark on the rest of the proof, with the extra conditions $-1 < k < \min\{|i|,|j|\}$ and, if $\Delta$ is hyperbolic, $|k| \neq n-2$. Note that the first condition implies $\min\{|i|,|j|\} \geq 1$ as $k \geq 0$.
 
\subsection{Properties of the round-up triples and quadruples} 
 
Let $\{J_1,J_2,J_3,J_4\}$ be a quadruple. We narrow down the possibilities for the mutual position of its members. We start by showing that at least one pair of them intersect in a subspace of dimension at least $k$ and then continue by proving that they all have one common intersection. These two steps are the crux of the proof. Though the intuitive idea behind them is easy, the proofs are quite long. We keep using the earlier introduced notation.

\begin{lemma}\label{dimdoorsnede} 
Up to renumbering, $\dim(J_1 \cap J_2) \geq k$.
\end{lemma} 
 
\begin{proof} Renumbering if necessary, the dimension $s$ of $S=J_1 \cap J_2$ is maximal amongst the dimensions of the intersections of all distinct pairs of the quadruple. By way of contradiction, suppose $s<k$. Let $c$ denote $1$ and $2$ again. 
According to property~\ref{con2}, there are allowed $k$-spaces $K_1$ and $K_2$ such that $\dim(\<K_1,K_2\> \cap F) <k$ for each $F \in \mathcal{F}:=\{J_3,J_4\}$ (\textbf{case 0}) \emph{unless} either  $(p_c,q_c)_c=(p^*,b)$ (\textbf{case 1}) or $S \subseteq F$ for some $F \in \mathcal{F}$ (\textbf{case 2}), as in those cases possibly  $\dim(\<K_1,K_2\> \cap F) =k$ for all choices of $K_1$ and $K_2$ (note that $\dim(\<K_1,K_2\> \cap F) >k$ is not possible since $\dim(F \cap K_c) \leq \dim(F \cap J_c) \leq s$ by assumption). Of course, case 1 only yields problems if $\Gamma=\Gamma_k^\ell$, as it involves the parameter $b$ which is only relevant in this case. The reader should keep this in mind, we will not make an explicit distinction between the three types of graphs during this proof. For clarity, the end of each case is marked by a black square. 

\par\medskip
\textbf{\boldmath Case 0: There exist $K_1$ and $K_2$ such that $\dim(\<K_1,K_2\> \cap F) <k$ for all $F \in \mathcal{F}$. \unboldmath}
Following Construction~\ref{con}, we below construct an $i$-space $I \in \mathsf{N}(J_1,J_2)$ with $\<K_1,K_2\>\subseteq I$ for which $\dim(I \cap F)<k$ for all $F \in \mathcal{F}$ (note that $k \geq 0$). Then $I$ would be adjacent to exactly two members of the quadruple, a contradiction to  the latter's definition.

Since $\dim(\<K_1,K_2\> \cap F)<k$ for all $F \in \mathcal{F}$, Property~\ref{con4}$(ii)$ says that we can choose $A^- \cap P$ such that $\dim(\<K_1,K_2, A^-\cap P\> \cap F) <k$ for each $F \in \mathcal{F}$, unless $a^p=p^*-p_2-1 \geq 0$, $-1 \in \{b,q^*\}$ and $\dim(F \cap P)=p^*+s+1$. This, however, is no problem: if $|\ell| < n-1$ we can choose $a^p < p^*-p_2-1$; if $|\ell|=n-1$ then $|i|=|j|=n-1$, in which case $a=-1$ (so $A^- \cap P$ is empty). 

Next, Properties~\ref{con5} and~\ref{con7} state that we can choose the remaining parts of $A_1,A_2,B_1,B_2$ such that for the resulting $i$-space $I:=\<K_1,K_2,A_1,A_2,B_1,B_2\>$ holds that $\dim(I \cap F) = \dim(\<K_1,K_2,A^-\cap P\> \cap F)$  for all  $F \in \mathcal{F}$, except when $\Delta$ is hyperbolic and we are in one of the below situations, in which possibly $\dim(I \cap F) = \dim(\<K_1,K_2, A^- \cap P\> \cap F) +1$ for some $F \in \mathcal{F}$.
\begin{itemize}
\item \textit{Case 0.1:  $\dim({J_1}^{J_2})=n-2$ and $\cod_{A^-}(A^-\cap P)=0$.} In this case, it is the selection of the part of $A^-$ outside $P$ that could cause a problem. Suppose $F \in \mathcal{F}$ is such that $\dim(\<K_1,K_2,A^-\>\cap F)=\dim(\<K_1,K_2,A^- \cap P\> \cap F)+1$. According to Property~\ref{con8}, then $\dim({J_1}^{J_2})=j+p^*+1=n-2$ and $\dim(\<P,A^- \setminus P\> \cap F)=p^*+s+2$, so in particular, $\dim(P \cap F)=p^*+s+1$. Then a dimension argument implies that $\dim(F \cap J_1) \geq s$, so the maximality of $s$ implies  $\dim(F \cap J_1)=s$.  Note that $\<P,A^-\setminus P\> \cap F$ is a $(p^*+s+2)$-space collinear with $J_1$, which implies that $\dim(J_1^F)=j+(p^*+1)+1=n-1$. Hence we can work with the pair $(J_1,F)$ instead, without ending up in Case 0.1 again (minor remark: later in this proof we sometimes switch the roles of the $j$-spaces again, but $\dim(J_1^F) =j+p^*+2$ will assure us that we can keep working with this pair). 

\item \textit{Case 0.2: $|i|=|j|=n-1$.} Since $a=-1$, it is only the selection of $B^-$ which could be a problem. 
First note that, as mentioned in Section 3, the graphs $\Gamma_0^\ell$ and $\Gamma_0$ are equal and, as we assume they are non-empty, they are isomorphic to $\overline{\Gamma_{\geq 2}}$, which we work with instead; moreover, if $\Gamma_{\geq 0}$ contains no adjacent pair $(I,J)$ with $\dim(I \cap J) = 0$, then we agreed to work with $\Gamma_{\geq 1}$ instead, and if it does contain such a pair then $\Gamma_{\geq 0}$ is a complete bipartite graph, which we excluded. Hence we may suppose that $k > 0$. This enables us to choose a hyperplane $H$ of $\<K_1,K_2\>$ such that $\dim(H \cap F) < k-1$ for all $F \in \mathcal{F}$.  Then we choose $B^-$ such that $\dim(\<H,B^-\> \cap F) = \dim(H \cap F)$. Let $I$ be the unique $i$-space through the $(n-2)$-space $\<H,B^-\>$. As $\<H,B^-\>$ is a hyperplane of $I$ and since the parity of the dimensions of intersection is fixed, it is easily verified that $\dim(I \cap J_c)=k$ and $\dim(I \cap F) \leq k-2$.

\end{itemize}
We obtained an $i$-space $I \in \mathsf{N}(J_1,J_2)$ with $\dim(I \cap F) < k$ for all $F \in \mathcal{F}$, and as explained in the beginning of this case, this is a contradiction.
\hfill $\blacksquare$
\par\medskip
For the sequel of this proof we may thus assume that for every pair of allowed $k$-spaces $K_1$, $K_2$ holds that $\dim(\<K_1,K_2\> \cap F) = k$ for some $F \in \mathcal{F}$, likewise for any permutation $\sigma$ of $\{1,2,3,4\}$ for which $\dim(J_{\sigma(1)} \cap J_{\sigma(2)})=s$ (the permutation also affects $\mathcal{F}$ of course).

Note that either all pairs of $j$-spaces of the quadruple intersecting  in an $s$-space are such that the only allowed values are $(p^*,b)$, or there is a pair, say $(J_1,J_2)$, for which there are allowed values $(p_c,q_c)_c \neq (p^*,b)$. In the latter situation we will suppose that we are in Case 2, since clearly Case 1 is not applicable and Case 0 is excluded by the previous paragraph. Therefore, when we are in Case 1, we may assume the first situation occurs. 

\par\bigskip
\textbf{\boldmath Case 1: Suppose, for every pair of distinct $j$-spaces $J_e$, $J_{e'}$ from the quadruple, that \emph{if} $\dim(J_e \cap J_{e'})=s$, then $(p_c,q_c)_c=(p^*,b)$ for $c\in\{e,e'\}$ (which implies $\dim({J_e}^{J_{e'}})=j+p^*+1$). \unboldmath} Consider the pair $(J_1,J_2)$, from which we know that $\dim(J_1 \cap J_2)=s$ and hence $(p_c,q_c)_c=(p^*,b)$. This implies that each $I \in \mathsf{N}(J_1,J_2)$ contains $P=\<S, P_1, P_2\>$ (cf.\ Property~\ref{con2}). Let $K_1$ and $K_2$ be any pair of allowed $k$-spaces. Analogously as in Case 0 and using the fact that $A^- \cap P$ is empty now (since $p_1=p_2=p^*$), we can take an $i$-space $I \in \mathsf{N}(J_1,J_2)$ through $\<K_1,K_2\>$ such that $\<K_1,K_2\> \cap F=I \cap F$ for all $F \in \mathcal{F}$: in Case 0.1, we obtain that the pair $(J_1,F)$ is such that $\dim(J_1 \cap F)=s$ and $\dim(J_1^F)=\dim({J_1}^{J_2}) +1=j+p^*+2$, contradicting our assumption; in Case 0.2 it is the selection of $B^-$ which is a problem whereas in the current case, $B^-$ is empty since $b=q_1$, so this does not occur.  Since $I$ is adjacent to $J_1$ and $J_2$, the definition of a quadruple implies that $I$ is adjacent to $J_3$ or $J_4$ as well. Suppose $I \sim J_3$. This has the following consequences.
\begin{itemize}
\item[(A)] By construction, $I \cap J_3$ is a $k$-space $K_3$ contained in $\<K_1,K_2\>$. Since $\dim(K_1 \cap K_2) =s$, we know that $\dim(K_1 \cap K_3) \geq s$, whereas $\dim(K_1 \cap K_3 )\leq \dim(J_1 \cap J_3) \leq s$ by assumption. We conclude that $K_1 \cap K_3=J_1 \cap J_3$ and has dimension $s$. By our assumption, $\dim(\proj_{J_1}(J_3)) = p^*+s+1$.

\item[(B)] Noting that $K_3 \cap J_1^\perp = K_3 \cap {J_1}^{J_2}$, we see that $\cod_{K_3}(K_3 \cap {J_1}^{J_2}) \leq b$ (as $I \sim J_1$, $I$ cannot contain a subspace of dimension bigger than $b$ which is semi-opposite $J_1$). Hence $\dim(K_3 \cap {J_1}^{J_2}) \geq p^*+s+1$. However, $K_3 \cap {J_1}^{J_2} \subseteq J_3 \cap {J_1}^{J_2} \subseteq  \proj_{J_3}(J_1)$, and the latter's dimension is $p^*+s+1$. We conclude that $K_3 \cap {J_1}^{J_2} =  J_3 \cap {J_1}^{J_2} = \proj_{J_3}(J_1)$ and that  $\cod_{J_3 \cap P}(J_3 \cap {J_1}^{J_2}) =b$; likewise with the indices $1$ and $2$ switched. 
\end{itemize}

Put $P_{1e}=\<\proj_{J_1}(J_e),\proj_{J_e}(J_1)\>$ for $e\in\{3,4\}$.  We show that $P_{13}=P$ (possibly by changing the roles of $J_3$ and $J_4$ -- which is only sensible to do if (A) and (B) also apply when $J_3$ is replaced by $J_4$). By (A), we already know that $\dim(P_{13})=\dim(P)$, so if we can show that $\proj_{J_1}(J_3) \cup \proj_{J_3}(J_1) \subseteq P$, then $P_{13}=P$. By definition, $\proj_{J_1}(J_3)$ is contained in $J_1$, so it belongs to $P$ if and only if it belongs to $\<S,P_1\>$. Furthermore, in the above we deduced that $\proj_{J_3}(J_1)$ is contained in ${J_1}^{J_2}(=\<J_1,P_2\>)$, so this subspace belongs to $P$ if and only if $\<\proj_{J_3}(J_1),P_2\> \cap J_1 \subseteq \<S,P_1\>$. So we try to show that $\proj_{J_1}(J_3) \cup (\<\proj_{J_3}(J_1),P_2\> \cap J_1) \subseteq \<S,P_1\>$. Observe that, since $I \in \mathsf{N}(J_1,J_2,J_3)$ and $(J_1,J_3)$ has the same mutual position as $(J_1,J_2)$, it follows as in the beginning of the proof that $P_{13}\subseteq I$ and hence we also know that $\proj_{J_1}(J_3) \subseteq K_1$ and $\proj_{J_3}(J_1) \subseteq K_3$ so also $\<\proj_{J_3}(J_1),P_2\> \cap J_1 \subseteq K_1$. 
\begin{itemize}
\item Firstly, let $b=-1$. Then $k=p^*+s+1$, and hence $\proj_{J_1}(J_3)=K_1$ and $\proj_{J_3}(J_1)=K_3$ by the above observation and equality of dimensions. Since $K_3 \subseteq \<K_1,K_2\>=P$ in this case, $P=P_{13}$ indeed.

\item Next, suppose $b \geq 0$. Put $X=\proj_{J_1}(J_3) \cup (\<\proj_{J_3}(J_1),P_2\> \cap J_1)$ and suppose for a contradiction that $X \nsubseteq \<S,P_1\>$. From the above observation we know $X \subseteq K_1$. Yet, using property $\mathsf{(SQ)}$, we can choose other $k$-spaces $K'_1$ and $K'_2$ such that $X \nsubseteq K'_1$. Similarly as before, we take an $i$-space $I' \in \mathsf{N}(J_1,J_2)$ through $\<K'_1,K'_2\>$ such that $\<K'_1,K'_2\> \cap F=I' \cap F$ for all $F \in \mathcal{F}$. Since $X \nsubseteq K'_1$, $I'$ is not adjacent to $J_3$, hence $I' \sim J_4$ and the above conclusions also hold for $J_4$. In particular, $X':=\proj_{J_1}(J_4) \cup (\<\proj_{J_4}(J_1),P_2\> \cap J_1) \subseteq K'_1$. Either $X'\subseteq \<S,P_1\>$ and then $P_{14}=P$, or $X' \nsubseteq \<S,P_1\>$ and then we can again choose $k$-spaces $K''_1$ and $K''_2$ such that $K''_1$ contains neither $X$ nor $X'$. This however leads to a contradiction, as a corresponding $i$-space $I''$ would not be adjacent to $J_3$, neither to $J_4$. The assertion follows.
\end{itemize}
Note that, even when both $J_3$ and $J_4$ satisfy (A) and (B), we cannot (yet) conclude that $P=P_{13}=P_{14}$.  

We obtain that, in $\Res_\Delta(P)$, $J_1$, $J_2$ and $J_3$ correspond to $q^*$-spaces $Q_1$, $Q_2$ and $Q_3$, with $Q_2$ and $Q_3$ opposite $Q_1$. As the notation suggests, we can identify $Q_1$ and $Q_2$ in $\Res_\Delta(P)$ with $Q_1$ and $Q_2$ in $\Delta$. Recall that $q^* \geq q_1+q_2+1 = 2b+1$. We distinguish the following three cases.
\par\medskip
$\bullet$  \textit{Suppose $q^* >2b+1$ and $b \geq 0$.} Let $K_1$ and $K_2$ be allowed $k$-spaces.  In $\Res_{\Delta}(P)$, $K_1$ and $K_2$ correspond to collinear $b$-spaces $B_1 \subseteq Q_1$ and $B_2\subseteq Q_2$.  First suppose that $\dim(\<K_1,K^*_2\> \cap J_3) =k$ for all $k$-spaces $K^*_2$ in $J_2$ such that $(K_1,K^*_2)$ are allowed $k$-spaces. Note that, in $\Res_\Delta(P)$, any $b$-space $B^*_2$ in $Q_2$ collinear with $B_1$ would yield a $k$-space $K^*_2$ such that $(K_1,K^*_2)$ are allowed $k$-spaces. So let $B'_2$ be a $b$-space in $Q_2$ collinear with $B_1$ and intersecting $B_2$ in a $(b-1)$-space (here we use $q^* > 2b+1$ and $b \geq 0$). Then $\<B_1,B_2\> \cap J_3$ and $\<B_1,B'_2\> \cap J_3$ are $b$-spaces $B_3$ and $B'_3$ in $Q_3$, respectively. The subspaces $\<B_2,B'_2\>$ and $\<B_3,B'_3\>$ have dimension (at least) $(b+1)$ and are contained in the $(2b+2)$-space $\<B_1,B_2,B'_2\>$, implying that $\dim(\<B_2,B'_2\> \cap \<B_3,B'_3\>)  \geq 0$. This however contradicts $Q_2 \cap Q_3= \emptyset$. 

Next suppose that there is a $k$-space $K^*_2$ such that $(K_1,K^*_2)$ are allowed $k$-spaces for which $\dim(\<K_1,K^*_2\> \cap J_3)<k$, and so $\dim(\<K_1,K^*_2\> \cap J_4) =k$ (as otherwise we are back to Case 0). We claim that $P=P_{14}$. If not, then $X'=\proj_{J_1}(J_4) \cup (\<\proj_{J_4}(J_1),P_2\>\cap J_1) \nsubseteq \<S,P_1\>$, according to the paragraph in which we showed $P=P_{13}$. But then we could re-choose $K_1$ such that $X' \nsubseteq K_1$, which would imply that $\dim(\<K_1,K^*_2\> \cap J_4)<k$, and hence $\dim(\<K_1,K^*_2\> \cap J_3)=k$ again, for all $k$-spaces $K^*_2$ in $J_2$ for which $(K_1,K^*_2)$ are allowed $k$-spaces, bringing us back to the previous paragraph. Hence indeed $P_{14}=P$, and thus $J_4$ plays the same role w.r.t. $J_1$ and $J_2$ as $J_3$. We extend our reasoning of the previous paragraph. Let $Q_4$ be the $q^*$-space in $\Res_\Delta(P)$ corresponding to $J_4$. Since $J_4$ plays the same role as $J_3$, $Q_4$ is also opposite $Q_1$. In $\Res_\Delta(P)$, we now take a third $b$-space $B''_2$ through $B_2 \cap B'_2$ that is collinear with $B_1$. We know that $\<B_1,B_2\>$, $\<B_1,B'_2\>$ and $\<B_1,B''_2\>$ intersect $Q_3$ or $Q_4$ in respective $b$-spaces $B$, $B'$, $B''$. Since $J_3$ and $J_4$ play the same role, we may assume that at least two of those $b$-spaces are contained in $Q_3$. We then obtain the same contradiction as in the previous paragraph.

We conclude that $q^*=2b+1$ or $b=-1$.

$\bullet$  \textit{Suppose $q^* =2b+1$ and $b \geq 0$.} We claim that, for all allowed $k$-spaces $K_1$ and $K_2$,  $\<K_1,K_2\> \cap J_3$ is a $k$-space. Suppose first that $P \neq P_{14}$ and suppose for a contradiction that there are $k$-spaces $K^*_1$ and $K^*_2$ such that $\dim(\<K^*_1,K^*_2\> \cap J_3)<k$. Then $\dim(\<K^*_1,K^*_2\> \cap J_4)=k$ (otherwise we are back in Case 0) and, as in the beginning of Case 1, we can again deduce consequences (A) and (B) with $J_3$ replaced by $J_4$. Recall that we then have $X':=\proj_{J_1}(J_4) \cup (\<\proj_{J_4}(J_1),P_2\> \cap J_1)\subseteq K^*_1$. As before, when we proved $P=P_{13}$, $P \neq P_{14}$ is equivalent with $X' \nsubseteq \<S,P_1\>$. Consequently, there are (many) $k$-spaces $K_1$ with $X' \nsubseteq K_1$, and for every pair of $k$-spaces $K_1$ and $K_2$ with $X' \nsubseteq K_1$, we know that $\<K_1,K_2\> \cap J_3$ is a $k$-space $K_3$. In $\Res_\Delta(P)$, we obtain that, for \emph{every} $b$-space $B_1$ disjoint from $B^*_1$ and for $B_2 = \proj_{Q_2}(B_1)$, the subspace $\<B_1,B_2\>$ intersects $Q_3$ in a $b$-space $B_3$. But then one can verify that also $\<B^*_1,B^*_2\> \cap Q_3$ has to be a $b$-space, contradicting $\dim(\<K^*_1,K^*_2\> \cap J_3) <k$. Next, suppose that $P=P_{14}$. Then $J_3$ and $J_4$ play the same role and in $\Res_{\Delta}(P)$, $J_4$ corresponds to a $q^*$-space $Q_4$ opposite $Q_1$. Moreover, since $J_2$, $J_3$ and $J_4$ play the same role w.r.t.\ $J_1$, we may assume that each pair of collinear $b$-spaces in $Q_1$ and $Q_e$, for $e=2,3,4$ generates a subspace intersecting at least one of the two remaining $q^*$-spaces in a $b$-space. Let $B_1$ be any $b$-space in $Q_1$ and put $B_e = \proj_{Q_e}(B_1)$. Then we may suppose that $\<B_1,B_2\> \cap Q_3 = B_3$. But then we may also assume that $\<B_1,B_4\> \cap Q_2 = B_2$, from which it follows that $\<B_1,B_2\>$ intersects \emph{both} $Q_3$ and $Q_4$ in a $b$-space. Our claim is shown.

Let $I$ be any $i$-space adjacent to $J_1$ and $J_2$. We show that $I \sim J_3$. First note that $q^*=2b+1$ implies $i=(s+2p^*+2b+4)+(a-p^*-1)+1 = j + a + 1 \geq j$. Our convention on $i$ and $j$ yields that either $a=-1$ or $i > j$. In both cases, $I \sim J_3$ if and only if $\dim(I \cap J_3)=k$ and $J_3 \setminus I$ contains no points collinear to $I$. Put $K_1=I \cap J_1$ and $K_2=I\cap J_2$. By the previous paragraph, we obtain that $\<K_1,K_2\> \cap J_3$ is a $k$-space $K_3$. Moreover, looking in $\Res_\Delta(P)$ again, it is easily seen that $J_3\setminus K_3$ cannot contain points of $I$, nor points collinear to $I$, since those points would be points of $Q_3 \setminus B_3$ collinear with $B_1$, contradicting $B_3=\proj_{Q_3}(B_1)$. But then $\mathsf{N}(J_1,J_2)\setminus\mathsf{N}(J_3)=\emptyset$, contradicting the definition of a quadruple. This case is ruled out as well.

$\bullet$ \textit{Suppose $b=-1$.} We show that $J_4$ plays the same role w.r.t.\ $J_1$ and $J_2$ as $J_3$, i.e., $\dim(J_1 \cap J_4)=s$ and $P=P_{14}$. We first claim that $\dim(J_4 \cap P)=k=p^*+s+1$. By way of contradiction, suppose $\dim(J_4 \cap P) < k$. 
The definition of a quadruple yields an $i$-space $I \in \mathsf{N}(J_1,J_2,J_4)\setminus\mathsf{N}(J_3)$, which necessarily contains $P$. In particular, $I$ contains $P \cap J_3$, which has dimension $p^*+s+1=k$. Since no point of $J_3 \setminus P$ is collinear with $J_1$ and $I \perp J_1$ because $b=-1$,  we have $I \cap J_3 = P \cap J_3$. Hence $\dim(J_3 \cap I)=k$. However, $I \nsim J_3$, so there has to be a point $p_3 \notin P$  which is collinear to $J_1$ and $J_2$ but not to $J_3$. Note that $a-p^*-1 > -1$, since if $a=p^*$ then, together with $b=-1$, this yields $I=P$, but then $P$ would be the only element in $\mathsf{N}(J_1,J_2)$, contradicting the definition of a quadruple.  So let $A$ be an $(a-p^*-2)$-space collinear with $J'_1:=\<J_1,p_3\>$ and $J'_2:=\<J_2,p_3\>$ such that $\dim(\<P,p_3\> \cap J_4) = \dim(\<P,p_3,A\> \cap J_4)$ (note that $A_c = \<P_{c'}, A, p_3\>$ then) and that such a subspace exists by Property~\ref{con5}, even if $\Delta$ is hyperbolic: if $\dim({J'_1}^{J'_2})=n-2$ and $\dim(A)=0$, then, since $p_3 \notin P$,  $\dim({J_1}^{J_2})=|j|+p^*+1=n-3$ and $a=\dim(P_{c'})+\dim(A)+\dim(p_3)+2=p^*+2$, and as $a \geq 0$, $|\ell|=|j|+a+1$, which is at its turn equal to $|j| + p^*+3 = n-1$, contradicting $|\ell|<n-1$ when $|j|<n-1$). Put $I_{p_3} := \<P,p_3,A\>$. By construction, $J_1 \sim I_{p_3} \sim J_2$. As $p_3 \in I_{p_3}$ and $p_3 \notin J_3^\perp$, necessarily $I_{p_3} \nsim J_3$, so the definition of a quadruple implies that $I_{p_3} \sim J_4$. Consequently, $J_4 \cap I_{p_3}$ is a $k$-space $K_4$ contained in $\<P,p_3\>$. Our assumption on $\dim(J_4 \cap P)$ implies that $\dim(K_4 \cap P)=k-1$, hence $K_4 \setminus P$ contains a point $q_3$ collinear with $J_1$ and $J_2$ and non-collinear with $J_3$. Likewise, there exists a point $q_2 \in J_4$ which is collinear to $J_1$ and $J_3$ but not to $J_2$. On the line $q_2q_3$, any point $q$ distinct from $q_2$ and $q_3$ is collinear with $J_1$, but not with $J_2$ nor with $J_3$. Now take an $i$-space $I_q \in \mathsf{N}(J_1,J_4)$ through $J_1 \cap J_4$ and $q$ (note that $q \in \proj_{J_4}(J_1)$). But then $I_q$ is not adjacent to $J_2$, neither to $J_3$. This contradiction shows the claim. Then, since $\dim(J_4 \cap P) = p^*+s+1$, we also have $\dim(J_1 \cap J_4)=s$, so by our assumptions we know $\dim({J_1}^{J_4})=j+p^*+1$, or equivalently, $\dim(\proj_{J_4}(J_1))=p^*+s+1=k$. Now any $i$-space $I$ adjacent with $J_1$, $J_2$ and $J_4$ contains $P$ and is collinear with $J_1$ and $J_4$. Consequently, $J_4 \cap P \perp J_1$ and $\<S,P_1\> \perp J_4$, so $J_4 \cap P \subseteq \proj_{J_4}(J_1)$ and  $\<S,P_1\> \subseteq \proj_{J_1}(J_4)$, respectively. As those subspaces are all $k$-dimensional, inclusion is in fact equality and $P_{14}=P$ follows.

Furthermore, each point $p\notin P$ collinear with $J_1$ and $J_2$ has to be collinear with $J_3$ or $J_4$, for otherwise we could find an $i$-space $\<P,p,A\>$ like in the previous paragraph which is not adjacent to $J_3$ nor to $J_4$, a contradiction.  Applying Lemma~\ref{opp} in $\Res_\Delta(P)$ on the respective subspaces corresponding to $J_1$, $J_2$, $J_3$ and $J_4$, it follows from Lemma~\ref{opp}, $p$ in fact has to be collinear with \emph{both} $J_3$ and $J_4$. Let $I$ be any member of  $\mathsf{N}(J_1,J_2)$. Then $I$ shares exactly a $k$-space with each of the four $j$-spaces, since it cannot contain points of $J_2\setminus P$, $J_3\setminus P$ or $J_4\setminus P$. As each point collinear with $J_1$ and $J_2$ is also collinear with $J_3$ and $J_4$, both are adjacent with $I$. We obtain the same contradiction as before: $\mathsf{N}(J_1,J_2)\setminus\mathsf{N}(J_3)=\emptyset$. \hfill $\blacksquare$
\par\medskip
One case remains. By Case 1 we may assume that there is a pair of $j$-spaces in the quadruple, without loss $J_1$ and $J_2$, such that $\dim(J_1 \cap J_2) = s$ and with allowed values $(p_c,q_c)_c \neq (p^*,b)$ (note that if there are allowed values distinct from $(p^*,b)$, then $(p^*,b)$ cannot be an allowed value).
\par\medskip
\textbf{\boldmath Case 2: suppose $J_3$ contains $J_1 \cap J_2$ and that there are allowed values $(p_c,q_c)_c \neq (p^*,b)$. \unboldmath} 
The maximality of $s$ implies that $J_1 \cap J_2 = J_2 \cap J_3 = J_3 \cap J_1$. 

Suppose first that $S \subsetneq J_4$. Then either $\dim(\<K_1,K_2\> \cap J_3) =k$ for all allowed $k$-spaces $K_1$ and $K_2$, or there is a pair for which $\dim(\<K_1,K_2\> \cap J_4)=k$, in which case we know that $\dim(K_c \cap J_4)=\dim(J_c \cap J_4)=s$. Since $(p_c,q_c)_c \neq (p^*,b)$ and $S \subsetneq J_4$, we can re-choose $K_1$ and $K_2$ such that they still are a pair of allowed $k$-spaces but now with $\dim(K_c \cap J_4) < s$. If we fix $K_1$ and take another $k$-space $K_2$ in $J_2$ for which $(K_1,K_2)$ is an allowed pair of $k$-spaces, we have that $\dim(\<K_1,K_2\> \cap J_3)=k$, likewise if we fix $K_2$ and vary the $k$-space $K_1$ in $J_1$. We claim that there are either multiple options for $K_2$ while fixing $K_1$, or multiple options of $K_1$ while fixing $K_2$. 
If not, then necessarily $p_1=p_2=p^*$ (as otherwise we can change $K_c\cap P_c$ in $P_c$) and $q_1+q_2+1=q^*$ (as otherwise we can change $K_c \cap Q_c$ in $\proj_{J_c}(Q_{c'} \cap K_{c'})$). But then $i=(j+a+1)+(b-q_1)$ and hence $q_1=b$, so $(p_1,q_1)=(p^*,b)$, contradicting our assumptions. This shows the claim, and without loss we may assume that there are multiple options for $K_2$ while fixing $K_1$. So let $K'_2$ be such a $k$-space with $\dim(K_2 \cap K'_2)=k-1$.  Completely similarly as in Case 1, $\<K_1,K_2\>$ and $\<K_1,K'_2\>$ contain $k$-spaces $K_3$ and $K'_3$ in $J_3$ and $\dim(\<K_2,K'_2\> \cap \<K_3,K'_3\>) > s$ since they are contained in the $(2k-s+1)$-space $\<K_1,K_2,K'_2\>$. This contradicts $J_2 \cap J_3 = S$.

If $S \subseteq J_4$, a similar argument applies: for all allowed $k$-spaces $K_1$ and $K_2$, we have $\dim(\<K_1,K_2\> \cap J_3) =k$ or $\dim(\<K_1,K_2,\> \cap J_4)=k$, and taking three $k$-spaces $K_2, K'_2, K''_2$ in $J_2$, as before, we may assume that $\dim(\<K_1,K_2\> \cap J_3)=\dim(\<K_1,K'_2\> \cap J_3)=k$, leading to the same contradiction. \hfill $\blacksquare$

\par\medskip
In all cases, we reached a contradiction, allowing us to conclude $\dim(J_1 \cap J_2) \geq k$.
\end{proof} 
Knowing this, we can show that all pairwise intersections coincide by considering well-chosen members of $\mathsf{N}_{(k)}(J_1,J_2)$. 
 \begin{lemma}\label{lem7a} 
All pairwise intersections of distinct members of the quadruple coincide.
\end{lemma} 
\begin{proof} 
Renumbering if necessary, the dimension $s$ of $S=J_1 \cap J_2$ is maximal amongst the dimensions of the intersections of all distinct pairs of the quadruple. By Lemma~\ref{dimdoorsnede}, we already know $s \geq k$. Again put $\mathcal{F}=\{J_3,J_4\}$. Suppose for a contradiction that, for each $F \in \mathcal{F}$, $S \nsubseteq F$. Then there is a $k$-space $K\subseteq S$ with $\dim(K \cap F) <k$ for all $F \in \mathcal{F}$. Completely analogously as in Case 0 of the previous lemma (though now using property~\ref{con4}$(i)$ instead of $(ii)$), Construction~\ref{con} yields an $i$-space $I \in \mathsf{N}_{(k)}(J_1,J_2)$ through $K$ such that $\dim(I \cap F) = \dim(\<K,A^- \cap P\> \cap F)<k$ for all $F \in \mathcal{F}$. This $i$-space is adjacent to exactly two members of the quadruple, a contradiction. So we may assume $S \subseteq J_3$ and by maximality of $s$ we obtain $J_1 \cap J_2 =J_1 \cap J_3 = J_2 \cap J_3$. 

If $\Gamma=\Gamma_{\geq k}$, recall that we actually work with triples, so in this case we assume $J_3=J_4$ and hence the lemma is proven here. If not, we have to show that $J_4$ contains $S$ as well. So assume for a contradiction that $ S\nsubseteq J_4$.

$\bullet$ \textbf{\boldmath Case 0:  Suppose that there is a point $p \in J_3 \setminus S$ such that $\dim(\<K,p\> \cap J_4) =k$ for every $K \subseteq S\setminus J_4$. \unboldmath} Put $K_4=\<K,p\> \cap J_4$.  Clearly,  $K_4 \subseteq J_3 \cap J_4$ and $\dim(K \cap K_4) =k-1$.
Firstly, let $k >0$. If $\dim(S \cap J_4)< s-1$, we can choose $K$ such that $\dim(K\cap J_4) \leq k-2$, contradicting $\dim(K \cap K_4) = k-1$. Hence $\dim(S\cap J_4)= s-1$, as we assume that $S \nsubseteq J_4$. Since $\<K,p\> \cap J_4$ is a $k$-space not entirely contained in $S$, it contains a point in $\<K,p\>\setminus K \subseteq  J_3 \setminus S$. Hence $\dim(J_3 \cap J_4) \geq s$ and by the maximality of $s$, $\dim(J_3\cap J_4)=s$. Repeating the argument used in the beginning of the proof, we conclude that $J_1$ or $J_2$ must contain $J_3\cap J_4$, a contradiction.  So if $k >0$, then $S \subseteq J_4$.
Next, let $k=0$. Then each line $\<p,K\>$ with $K$ a point in $S$ contains a point of $J_4$, so either $p \in J_4$ or $\dim(\<p,S\> \cap J_4) =s$. In the latter case $\dim(J_3 \cap J_4)  = s$ and, as above,  $S \subseteq J_4$. So, if $S \nsubseteq J_4$, then $p \in J_4$ for any point $p \in J_3 \setminus S$; however, this gives  $J_3=J_4$, a contradiction. \hfill $\blacksquare$

We now use Construction~\ref{con} in trying to show that there is a point $p \in J_3 \setminus S$ such that $\dim(\<K,p\> \cap J_4)=k$ for each $k$-space $K \subseteq S \setminus J_4$.  Put $U^x_y := \proj_{J_y}(J_x)\setminus S$ for $\{x,y\} \subseteq \{1,2,3\}$. We claim that we can always choose a point $p \in J_3 \setminus S$ such that either $p \in U^1_3 \cap U^2_3$ (\textbf{case 1}) or $p \notin U^1_3 \cup U^2_3$ (\textbf{case 2}), possibly by interchanging the roles of the $j$-spaces. Indeed, the only possibility where $U^1_3$ and $U^2_3$ have empty  intersection while their union is $J_3 \setminus S$, occurs when $\{U^1_3,U^2_3\} = \{\emptyset,J_3\setminus S\}$. In this case we can interchange the roles of the $j$-spaces to end up in either case 1 or case 2. The claim follows.

$\bullet$ \textbf{\boldmath Case 1: $p \in U^1_3 \cap U^2_3$.\unboldmath} On the condition that $a \geq 0$, we show that $\dim(\<K,p\> \cap J_4)=k$  for each $K \subseteq S \setminus J_4$. Noting that $\dim(\<K,p\> \cap J_4) >k$ violates $K \nsubseteq J_4$), we assume by way of contradiction that there is a $k$-space $K \subseteq S \setminus J_4$ with $\dim(\<K,p\> \cap J_4)<k$. We now construct an $i$-space $I_p$ in $\mathsf{N}_{(k)}(J_1,J_2)$ with $p \in A$ (hence the requirement $a \geq 0$) such that $\dim(I_p \cap J_4) <k$ and then, as $\dim(I_p \cap J_3) \geq \dim(\<K,p\>) =k+1$, we obtain that $I_p$ is adjacent to exactly two members of the quadruple, a contradiction to the latter's definition. 

Suppose first that $p \in P$. Property~\ref{con4}$(i)$ implies that we can choose $A \cap P$ such that $p \in A \cap P$ and $\dim(\<K,A\cap P\> \cap J_4)<k$, unless possibly if $a^p= p^* \geq 0$ and $\dim(J_4 \cap P) = p^*+s+1$. However, if $\dim(J_4 \cap P)=p^*+s+1$, then $\dim(J_4 \cap J_1)=s$ and, by the first part of the proof, $J_2$ or $J_3$ has to contain $J_4 \cap J_1$, implying $S \subseteq J_4$ after all.  
Next, suppose $p \notin P$. Then Properties~\ref{con4}$(i)$,~\ref{con5}~and~\ref{con8} imply that we can choose $A$ such that $p \in A$ and  $\dim(\<K,A\> \cap J_4)<k$, unless the conditions in Property~\ref{con8} are not met (those in Property~\ref{con4}$(i)$ we can deal with as before): if $\Delta$ is hyperbolic, $\<P,A\> = \<P,p\>$ (this expresses that $\cod_A(A\cap P)=0$) and $\dim(J_4 \cap \<P,p\>)=p^*+s+2$. Since $\dim(J_1 \cap J_4)\leq s$ by assumption, we obtain that $\dim(J_4 \cap P)=p^*+s+1$, which as before leads us to $S \subseteq J_4$. 

We can now select $B$  such that $\<K,A,B\> \cap J_4=\<K,A\> \cap J_4$, since $a \geq 0$ means that we are not in the case where $|i|=|j|=n-1$. Then $I_p:=\<K,A,B\>$ is such that $\dim(I_p \cap J_4)<k$. As explained above, this allows us to get back to Case 0 and conclude $S \subseteq J_4$.
\hfill $\blacksquare$

\par\medskip
$\bullet$ \textbf{\boldmath Case 2: $p \notin U^1_3 \cup U^2_3$.\unboldmath} On the condition that $b \geq 0$, we show that $\dim(\<K,p\> \cap J_4)=k$ for each $K \subseteq S \setminus J_4$ and therefore we again assume $\dim(\<K,p\> \cap J_4)<k$. We apply the same technique as in the previous case, but now with $p \in B$.
However, if $a\geq 0$, we first need to find an $a$-space $A$ such that $p \in A^{\perp}$ ($A$ and $B$ are selected consecutively). Note that $a \geq 0$ implies $|\ell| < n-1$. Like in the previous case, there is a subspace $A^*$ with $j+\dim(A^*)+1=n-1$ collinear with both $J_1$ and $J_2$ such that $\dim(\<K,A^*\> \cap J_4)< k$ (recall that the possible cases of exceptions imply $S \subseteq J_4$). Since $|\ell|<n-1$ we have $a < \dim(A^*)$, and as $p$ is collinear with at least a hyperplane of $A^*$, we can choose $A \subseteq A^* \cap p^\perp$.

Now we want to choose $B$ such that $p \in B$. If $B$ is moreover such that $\dim(\<K,A,B\> \cap J_4)<k$, then $I_p:=\<K,A,B\>$ is as required, and as before, Case 0 now implies $S \subseteq J_4$. 
According to Property~\ref{con9}, such a $b$-space $B$ exists, unless possibly if $\Delta$ is hyperbolic and $|i|=|j|=n-1$ (so $A = \emptyset$), as then it could be that $\dim(\<K,B\> \cap J_4) = \dim(\<K,p\> \cap J_4)+1=k$. So suppose this happens. We aim for a contradiction. Since $\<K,B\> \sim J_4$, $\dim(\<K,B\> \cap J_4)=k$ and so $\<K,p\> \cap J_4$ is a $(k-1)$-space $H_4$ and hence $\dim(H_4 \cap K) \geq k-2$. As discussed in Case 0.2 of the previous lemma, we may assume $k > 0$ because $\Delta$ is hyperbolic. The argument is completely similar as the one in Case 0: if $k>1$ we can vary $K \subseteq S$ to obtain $\dim(J_4 \cap S) \geq s-2$ and hence $\dim(J_4 \cap J_3) > s-2$, which then implies $\dim(J_3 \cap J_4)=s$ since $\Delta$ is hyperbolic; and if $k=1$, we consider the planes $\<p,K\>$ with $K$ a line in $S$ to conclude that either $\dim(J_4 \cap \<K,p\>) \geq s-1$ (and hence $\dim(J_3 \cap J_4)\geq s-1$) or $p \in J_4$ (and then we vary $p$), both leading to $S \subseteq J_4$. \hfill $\blacksquare$

\par\medskip
$\bullet$ \textbf{Case 3: Suppose the requirements in the above cases are not met.}
If this happens then, since not both $a$ and $b$ can be $-1$ (recall $k < \min\{i,j\}$), one of the following holds. 
\begin{itemize} 
\item[\textbf{(3.1)}] $a=-1$ and there is no point $p\in J_u\setminus S$ such that $p \notin U^t_u \cup U^v_u$ for $\{u,t,v\}=\{1,2,3\}$, 
\item[\textbf{(3.2)}] $b=-1$ and there is no point $p\in J_u\setminus S$ such that $p \in U^t_u \cap U^v_u$ for $\{u,t,v\}=\{1,2,3\}$. 
\end{itemize} 
In these cases we use another method to show $S\subseteq J_4$. Assume for a contradiction that $\dim(S\cap J_4)<s$.
 
\textbf{Case (3.1)} The assumptions imply $U^t_u \cup U^v_u=J_u \setminus S$, for all $\{u,t,v\}=\{1,2,3\}$. Since no two proper subspaces can cover $J_u \setminus S$, we have, without loss of generality, that $J_1\setminus S = U^2_1$ (and hence $J_2\setminus S =U^1_2$) and $J_3\setminus S = U^1_3$, i.e., $J_2 \perp J_1 \perp J_3$.

Let $p$ be any point semi-opposite both $J_1$ and $J_2$, not contained in $J_1 \cup J_2 \cup J_3$. We show that $p$ is also semi-opposite $J_3$. If $|j|=n-1$, this is trivial, so suppose $|j|<n-1$ and $p \perp J_3$.  Let $K$ be any $k$-space in $S\setminus J_4$. As before, we assume $\dim(\<K,p\> \cap J_4)<k$. Then we can select  $I=\<K,B\>\in \mathsf{N}_{(k)}(J_1,J_2)$ with $p \in B$ with $\dim(\<K,B\> \cap J_4) < k$ (by Property~\ref{con9}, this is always possible since $a=-1$ and $|j| <n-1$). As $p \in I$, we have $I \nsim J_3$, and $\dim(\<K,B\> \cap J_4) <k$ implies $I \nsim J_4$. Hence it follows that $\dim(\<K,p\> \cap J_4)$ is a $k$-space $K_4$ for each $k$-space $K \subseteq S \setminus J_4$. Furthermore, since $K_4 \nsubseteq S$, there is a point $p_4$ in $K_4 \setminus K$, which is just like $p$ contained in $J_3^\perp \setminus J_3$ and semi-opposite $J_1$ and $J_2$.

We claim that $\dim(J_4 \cap S)=s-1$. The argument is similar as the one used in Case 0, so we omit some details. If $k >0$ then varying $K \subseteq S\setminus J_4$ implies $\dim(J_4 \cap S)=s-1$. So suppose $k=0$. Then each line $\<p,K\>$ with $K$ a point in $S$ contains a point of $J_4$, so either $\dim(\<p,S\> \cap J_4) =s$, in which case $\dim(S \cap J_4) =s-1$ indeed, or $p \in J_4$. So if $\dim(S\cap J_4) < s-1$,  then varying $p \in \<p,S\> \setminus (p \cup S)$ (note that each point in $\<p,S\>\setminus (p \cup S)$ has the same collinearity relations w.r.t.\ $J_1$, $J_2$ and $J_3$ as $p$) yields $\<S,p\> \subseteq J_4$, violating $S \nsubseteq J_4$). The claim follows.

Next, we first suppose $k \leq s-1$. Then we take a $k$-space $K \subseteq J_4 \cap S$ and consider an element $I = \<K,B\> \in \mathsf{N}_{(k)}(J_1,J_2)$ with $p \in B$. Clearly, $\dim(J_4 \cap I) >k$ and hence $I \nsim J_4$; moreover, $I\nsim J_3$ because $p \in I$. This is a contradiction to the definition of a quadruple. Consequently, $k=s$. Since we already obtained that $\dim(J_3 \cap J_4) \geq k$ (recall $K_4 \subseteq J_3 \cap J_4$), this means $\dim(J_3 \cap J_4)=s$ and as before the latter implies $S \subseteq J_4$. As this violates our assumptions, we get that $p \notin J_3^\perp$.



Let $I \in \mathsf{N}(J_1,J_2)$ be arbitrary. From the above we can deduce that $I$ is also adjacent with $J_3$: firstly, $I\cap J_2$ is a $k$-space $K$ inside $S$ as $J_2 \setminus S$ contains points collinear with $J_1$, and by the same token,  $I\cap J_3 = K$; secondly, each point in $I \setminus K$ is semi-opposite both $J_1$ and $J_2$ and by the above, those points are also semi-opposite $J_3$. This contradiction to the definition of a quadruple shows that $S \subseteq J_4$.

\textbf{Case (3.2)} In a similar way as in Case 3.1, one can show that each point $p \notin J_1 \cup J_2 \cup J_3$ which is collinear with $J_1$ and $J_2$ is also collinear with $J_3$: first show that $\dim(\<K,p\> \cap J_4)=k$ for all $k$-spaces $K \subseteq S\setminus J_4$ (note that if not we can find $I=\<K,A\> \in \mathsf{N}_{(k)}(J_1,J_2)$ with $p \in A$ and $\dim(I \cap J_4)<k$, similarly as in Case 1 above), continue by showing that $\dim(J_4 \cap S)=s-1$ in exactly the same way as above, and then concluding in the same way as above (with $A$ instead of $B$) that $p \in J_3^\perp$ after all. 

Knowing this, we can show that $U^1_2$ and $U^2_1$ are empty (note that $\dim(\<S,U^1_2\>)=\dim(\<S,U^2_1\>)$): if not, each point $p$ in $\< U^1_2,U^2_1 \> \setminus (U^1_2 \cup U^2_1)$  is collinear with both $J_1$ and $J_2$ and hence also with $J_3$. As $p$ was arbitrary, the entire space  $\<U^1_2,U^2_1 \>$ has to be collinear with $J_3$, but then $U^1_2 \cap U^3_2 \neq \emptyset$, which contradicts our assumptions. This holds for all permutations of $1,2,3$, so in $\Res_\Delta(S)$, the $j$-spaces are pairwise opposite. 

As above, we now deduce that each $I \in \mathsf{N}(J_1,J_2)$ is also adjacent with $J_3$. We conclude $S \subseteq J_4$. \hfill $\blacksquare$
\par\medskip 
Finally, if $\Gamma=\Gamma_k$, the existence of $I_p$ is easily shown, as it does not matter whether $p \in A$ or $p \in B$.  \end{proof}

\textbf{Notation} $-$ We keep referring to $J_1 \cap J_2$ by $S$. We also keep using $U^x_y = \proj_{J_y}(J_x) \setminus S$.

\begin{rem} \em 
Note that Lemma~\ref{dimdoorsnede} is trivial when $k=-1$, but Lemma~\ref{lem7a} is not. Not only does the latter's proof rely on $k \geq 0$, as we encountered $i$-spaces with $\dim(I \cap J) <k$ for some $j$-space $J$,  also, when $k=-1$ we at first only have a weaker version of this Lemma (cf.\ Section~\ref{k=-1}). Hence we have to proceed in a different way than we will do now, which is the reason why we have devoted a section on $k=-1$. 

In the proofs of the previous two lemmas, we always carefully verified whether we can select $A$ and $B$ such that they do not intersect $J_3$ and $J_4$. In the sequel, we will no longer explicitly do this, since all techniques needed have been discussed above and hence it would only make the proofs longer than necessary. 
\end{rem}

The following property was used in a special case of the proof of the previous lemma. We state it with respect to $J_1$ and $J_2$ but it is valid for any pair of (distinct) $j$-spaces in a quadruple $\{J_1,J_2,J_3,J_4\}$.
\begin{itemize} 
\item[$\mathsf{(RU1)}$] Let $p$ be a point contained in at most one member of a quadruple $\{J_1,J_2,J_3,J_4\}$. If $p \in J_1^\perp \cap J_2^\perp$, then $p \in J_3^\perp \cup J_4^\perp$.
\end{itemize} 

Now that we know that all $j$-spaces have one common intersection, we can show that this property holds when $\Gamma=\Gamma_k^\ell$. Despite the above made remark, this is one of the few occasions that a lemma also holds for $k=-1$ as well.

\begin{lemma}\label{ru1} If $\Gamma$ equals $\Gamma_k^{\ell}$, possibly $k=-1$, then $\mathsf{(RU1)}$ holds. Moreover, $\mathsf{(RU1)}$ remains valid in $\Res_{\Delta}(S')$ for any subspace $S' \subseteq J_1 \cap J_2 \cap J_3 \cap J_4$. \end{lemma} 
 
\begin{proof}   
Let $p$ be an arbitrary point collinear with $J_1$ and $J_2$, not contained in $J_1 \cap J_2$. Recall that, if $k \neq -1$,  Lemma~\ref{lem7a} states that the $j$-spaces intersect each other in $S$. Suppose for a contradiction that $p \notin J_3^\perp \cup J_4^\perp$. In particular, $p \notin J_3 \cup J_4$. 

\textbf{\boldmath First suppose that $a \geq 0$,  $|i| \leq |j|$ and $p \notin J_1 \cup J_2$. \unboldmath} Note that $|j| <n-1$ as otherwise there would be no point $p \in J_1^\perp\setminus J_1$. In this case, we take an element $I=\<K,A,B\>$ of $\mathsf{N}_{(k)}(J_1,J_2)$ such that:
\begin{compactenum}[$-$]
\item The $k$-space $K$ is empty if $k=-1$ and belongs to $S$ if $k \geq 0$.
\item The $a$-space $A$ collinear to $J_1$ and $J_2$ contains the point $p$ and is such that $\<K,A\> \cap J_e = K$ for $e=1,2,3,4$ (as in the proof of Lemma~\ref{lem7a}).
\item The $b$-space $B$ is chosen in $\Res_\Delta(\<K,A\>)$ such that it is semi-opposite the subspaces corresponding to $J_1$, $J_2$, $J_3$ and $J_4$ which are all of dimension at least $(|j|-k-a-2)$, and as $|i| \leq |j|$, we have that $|j|-k-a-2 \geq b$, so in each of those subspaces we can take $b$-spaces, and by Fact~\ref{lem4a}$(ii)$, there is a $b$-space opposite them and avoiding $J_3 \cup J_4$ (note that $|j|<n-1$). 
\end{compactenum}
As $I \sim J_1,J_2$, we may assume that $I \sim J_3$. However, by our choice of $B$ in $\Res_\Delta(\<K,A\>)$ we have that $\proj_{J_3}(I)=\<K,A\>$ and in particular $p$ is collinear with $J_3$, as we wanted to show.

\textbf{\boldmath Next, suppose that $|i| = |j| +a +1$ and $p \notin J_1 \cup J_2$. \unboldmath}
Note that this case comprises $a =-1$, since if $a=-1$ then $\max\{|i|,|j|\} =|\ell|$, so our convention (cf.\ beginning of Section~\ref{constructie}) on $|i|$ and $|j|$ implies that $|i|=|\ell|=|j|+a+1=|j|$. Moreover, we may assume that $|i| < n-1$ since otherwise $|j|=n-1$ too and like above, this conflicts with $p \in J_1^\perp\setminus J_1$. Furthermore, we know $b \geq 0$ as $k < |j|$. Note that an adjacent pair $(I,J)$ is such that no point of $J\setminus(I\cap J)$ is collinear with $I$. We now take an element $I=\<K,A,B\>$ of $\mathsf{N}_{(k)}(J_3,J_4)$ such that:
\begin{compactenum}[$-$]
\item The $k$-space $K$ is empty if $k=-1$ and belongs to $S$ if $k \geq 0$.
\item The $b$-space $B$ equals $\<p,B^-\>$, where $B^-$ is a $(b-1)$-space semi-opposite $J_1$, $J_2$, $J_3$ and $J_4$ and avoiding $J_1$ and $J_2$: we choose $B^-$ in $\Res_\Delta(\<K,p\>)$, in which $J_1$ and $J_2$ correspond to $b$-spaces, and $J_3$ and $J_4$ to $(b-1)$-spaces (seeing $|j|-k-1=b$), so like above this is possible. 
\item  If $a \geq 0$, the $a$-space $A$ is chosen in $\Res_\Delta(\<K,p,B^-\>)$, in which $J_1$ and $J_2$ now correspond to points $p_1$ and $p_2$, and in which $J_3$ and $J_4$ do not correspond with anything. Let $A$ be collinear with $p_1$ and $p_2$ and avoiding $J_3$ and $J_4$, which is possible by Fact~\ref{lem4a}$(i)^*$ (note that the rank of $\Res_\Delta(\<K,p,B^-\>)$ is $n-j-1$ and $a$ is such that $|j|+a+1=|\ell|<n-1$ and even $|\ell| \leq n-3$ if $\Delta$ is hyperbolic since $i=\ell$).
\end{compactenum}
The resulting $i$-space $\<K,p,B^-,A\>$ is adjacent with $J_3$ and $J_4$ because, for $e\in\{3,4\}$, $I\cap J_e$ is the $k$-space $K$ and no point of $J\setminus K$ is collinear with $I$. However, $I$ is not adjacent with $J_1$ and $J_2$ because both contain a (unique) point collinear with $I$. This contradiction to the definition of a quadruple shows that $p \in J_3^\perp \cup J_4^\perp$.

\textbf{\boldmath Suppose  $p \in J_1 \setminus S$, i.e., $p \in U^2_1$.\unboldmath} We consider an arbitrary point $q$ of $\<U^1_2,U^2_1\> \setminus (U^1_2 \cup U^2_1)$. If $q \in J_3 \cup J_4$ then clearly $q \in J_3^\perp \cup J_4^\perp$; if $q \notin J_3 \cup J_4$, the previous cases imply $q \in J_3^\perp \cup J_4^\perp$. As $q$ was arbitrary, $\<U^1_2,U^2_1\>\subseteq J_3^\perp \cup J_4^\perp$, and as $J_3^\perp$ and $J_4^\perp$ are subspaces (even though not singular), we may assume that  $\<U^1_2,U^2_1\>\subseteq J_3^\perp$. In particular, $p\in J_3^\perp$.  
 
The fact that $\mathsf{(RU1)}$ is a residual property is easily verified.  \end{proof}  
 
If the quadruple has one common intersection (which is the case if $k>-1$, by Lemma~\ref{lem7a}), it is no restriction to require that $p$ is contained is at most one member of the quadruple, as for each point of the intersection, $\mathsf{(RU1)}$ is trivially fulfilled. In case $a \neq -1$, we can say more. The following lemma improves Lemma~\ref{ru1} in the case where $p \in J_1 \cup J_2$ in $\mathsf{(RU1)}$.
 
\begin{lemma}\label{gemprojectie} Let $\Gamma$ equal $\Gamma_k^{\ell}$ and suppose $a\geq 0$, possibly $k=-1$. Then $U^2_1=U^3_1=U^4_1$ and this for all permutations of $\{1,2,3,4\}$. \end{lemma} 
\begin{proof}  
Again, by Lemma~\ref{lem7a}, the $j$-spaces intersect each other in $S$ if $k>-1$. The condition $a \geq 0$ implies $|j|<n-1$.

\textbf{\boldmath First suppose $|i| \leq |j|$.\unboldmath}   If $p \in U^2_1$ then we may assume $p \in U^3_1$ in view of the previous lemma. We show that $p \in U^4_1$ too. Suppose for a contradiction that $p \notin J_4^\perp$.  We choose an $i$-space $I=\<K,A,B\> \in \mathsf{N}_{(k)}(J_2,J_3)$ like in the first case of the previous lemma, i.e., with $p \in A$, only now $\<K,A\> \cap J_e = K$ for $e=2,3,4$ and $\<K,A\> \cap J_1 = \<K,p\>$. The latter implies that  $I$ cannot be adjacent to $J_1$ and hence has to be adjacent to  $J_4$, forcing $p$ to be collinear with $J_4$.

\textbf{\boldmath Next suppose $|i| = |j|+a+1$.\unboldmath}  Assume for a contradiction that $U^4_1 = U^2_1 = U^3_1$ does not hold. In view of the previous lemma we may assume that $U^4_1\subsetneq U^2_1 = U^3_1$. Possibly by switching the roles of $J_2$ and $J_3$, we may also assume that  $U^2_4 \subsetneq  U^3_4 = U^1_4$ or  $U^1_4 \subsetneq U^2_4=U^3_4$ or $U^1_4=U^2_4=U^3_4$. In the first case, we find a line $p_1p_4$ such that $p_1 \in U^3_1\setminus U^4_1$ and $p_4 \in U^3_4 \setminus U^2_4$. Clearly, $p_1p_4$ is collinear with $J_1$ and $J_3$. It follows from the previous lemma that $p_1p_4$ should be collinear with $J_2$ or $J_4$ as well, a contradiction. Similarly in the second case. Hence we are in the third case: $U^1_4=U^2_4=U^3_4$. Moreover, it follows that $U^4_2\subsetneq U^3_2 = U^1_2$ and $U^4_3\subsetneq U^2_3 = U^1_3$, so $J_1$, $J_2$ and $J_3$ play the same role w.r.t.\ each other and w.r.t.\ $J_4$.

If $k \geq 0$, let $K$ be any $k$-space inside $S$; if $k=-1$, let $K =\emptyset$. Now take any $(a+1)$-space $A^*$ collinear with $J_2$ and $J_4$ such that $\<K,A^*\> \cap J_e = K$ for $e=1,3$, which is possible by Fact~\ref{lem4a}$(i)^*$ and since $|i| <n-1$ (and if $\Delta$ is hyperbolic even $|i| \leq n-3$). Then by the previous lemma, $A^* \cap J_1^\perp$ and $A^* \cap J_3^\perp$ are subspaces of $A^*$ that together cover $A^*$, which is only possible if one of them coincides with $A^*$. As $J_1$ and $J_3$ play the same role w.r.t.\ $J_2$ and $J_4$, we may assume that $A^*$ is collinear with $J_3$. Now let $A$ be an $a$-space inside $A^*$ collinear with a point $p \in J_1 \setminus S$. In $\Res_\Delta(\<K,p,A\>)$, the $j$-spaces $J_2$ and $J_3$ correspond to $b$-spaces $J'_2$ and $J'_3$ (recall $|j|-k-1=b$), hence there is a $b$-space opposite $J'_2$ and $J'_3$, that corresponds to a $b$-space $B$ in $\Delta$ semi-opposite $J_2$ and $J_3$. The corresponding $i$-space $I=\<K,A,B\> \in \mathsf{N}_{(k)}(J_2,J_3)$ is then collinear with the point $p \in J_1\setminus K$ and hence $I \nsim J_1$ (recall that no point of $J\setminus I$ is collinear with $I$ for an adjacent pair $(I,J)$, when $|i|=|j|+a+1$). Consequently, $I \sim J_4$ and hence, as $A \perp J_4$, we obtain that $B$ is semi-opposite $J_4$. 

Knowing this, we can reach a contradiction as follows. Let $A^-$ be an $(a-1)$-space of $A$ (recall $a \geq 0$) and let $p$ be a point in $U^2_1 \setminus U^4_1$. Then $I=\<K,p,A^-,B\>$ also belongs to $\mathsf{N}_{(k)}(J_2,J_3)$, since $p \in J_2^\perp \cap J_3^\perp$. Since $\dim(I \cap J_1)=k+1$, $I\nsim J_1$. But now $I \nsim J_4$ too, because $\<B,p\>$ is a $(b+1)$-space of $I$ semi-opposite $J_4$, a contradiction. We conclude that $U^4_1 = U^2_1 = U^3_1$. The lemma is proven.
\end{proof} 

 \par\bigskip
\textbf{Notation} $-$ We write $U_y$ instead of $U^x_y$, for all $x\neq y \in \{1,2,3,4\}$, if the latter does not depend on $x$.

The following lemma will be very useful in combination with Lemma~\ref{lem0}, as it states that, under certain conditions, lines intersecting two members of the round-up quadruple, have to intersect a third member of the round-up quadruple. Again, we state it w.r.t.\ $J_1$ and $J_2$ but it holds for any two (distinct) members of the quadruple.
 
\begin{lemma}\label{lem1} \begin{itemize}
\item If $\Delta$ is hyperbolic and $|i|=|j|=n-1$, then each pair of collinear lines $L_1 \subseteq J_1 \setminus S$ and $L_2 \subseteq J_2 \setminus S$ is such that $\<L_1,L_2\>$ intersects either $J_3$ or $J_4$ in a line.
\item In all other cases, each pair of collinear points $x_1 \in J_1 \setminus S$ and $x_2 \in J_2 \setminus S$ (where we, if $\Gamma=\Gamma_k^\ell$, require that $x_c \in \mathsf{Q}_c$ if $a=-1$ and $x_c \in \mathsf{P}_c$ if $b=-1$) is such that $x_1x_2$ intersects $J_3$ or $J_4$ in a point.
\end{itemize} 
\end{lemma} 
\begin{proof} First suppose that $\Delta$ hyperbolic and $|i|=|j|=n-1$. Let $L_1 \subseteq J_1 \setminus S$ and $L_2 \subseteq J_2\setminus S$ be collinear lines (hence $\dim(S) \leq n-5$). First note that, as before, when $\Delta$ is hyperbolic and $|i|=|j|=n-1$, we may in this case suppose that $k \geq 1$. Assume first that $\<L_1,L_2\>$ has nothing in common with $J_3 \cup J_4$. Let $K$ be a $(k-2)$-space in $S$ and put $K_1:=\<K,L_1\>$ and $K_2:=\<K,L_2\>$. Then $J_1$, $J_2$, $J_3$ and $J_4$ correspond to subspaces of the same type in $\Res_\Delta(\<K_1,K_2\>)$. Consequently, we can take a $(b-2)$-space $B$ in this residue semi-opposite all of them. The corresponding $i$-space $I:=\<K_1,K_2,B\>$ belongs to $\mathsf{N}_{(k-2)}(J_1,J_2)$ but is not adjacent to $J_3$ nor to $J_4$, a contradiction. Hence, for each pair of collinear lines $L_1$ and $L_2$, $\<L_1,L_2\>$ has to intersects at least one of $J_3$, $J_4$ in a point (which is in particular collinear with $L_1$). Consider the $j$-space $L_1^{J_2}$ and put $X_3:=J_3 \cap L_1^{J_2}$ and $X_4:=J_4 \cap L_1^{J_2}$. We claim that $\max\{\dim(X_3),\dim(X_4)\}=n-3$. Suppose for a contradiction that $\dim(X_3) \leq n-5$ and $\dim(X_4) \leq n-5$. Then there is a line $L'_2 \subseteq \proj_{J_2}(L_1)$ disjoint from $(\<L_1,X_3\> \cap J_2) \cup (\<L_1,X_4\> \cap J_2)$ (as this is the union of two subspaces of $J_2$ of dimension smaller or equal to $n-5$). But then it is impossible that $\<L_1,L'_2\>$ contains a point from $J_3 \cup J_4$, a contradiction. This proves the claim. Without loss, $\dim(X_3)=n-3$. Since $\dim(\<L_1,L_2\> \cap X_3) \geq 1$ we obtain that $\<L_1,L_2\> \cap J_3$ is a line after all, proving the first assertion.

Now suppose that we are not in the previous case. Let $x_1$ and $x_2$ be as in the statement of the lemma.
We want $I\in \mathsf{N}_{(k-1)}(J_1,J_2)$ such that $x_c \in K_c$ and with $I\setminus\<K_1,K_2\>$ avoiding $J_3 \cup J_4$. If such an $i$-space exists, then, without loss of generality, we have $I \sim J_3$, implying  $\dim(J_3 \cap I)= k$. As $J_3 \cap I \subseteq \<K_1,K_2\>$, the line $x_1x_2$ intersects $J_3$ in a point $x$.  Like before, the existence of such an $i$-space could only be a problem when $\Gamma=\Gamma_k^\ell$ and $\Delta$ is hyperbolic, and either  $|i|=|j|=n-1$ (which we excluded) or $|\ell|=n-2$, $\dim({J_1}^{J_2})=n-2$, $\dim(J_e \cap P) = p^*+s+1$ and $J_e\setminus P$ contains a point collinear with $J_1$ and $J_2$, for some $e\in\{3,4\}$. However, the latter situation does not occur. Indeed, the fact that $|j|<|\ell|$ implies $a \geq 0$, and then Lemma~\ref{gemprojectie} tells us that, for $e\in\{3,4\}$, $\dim(\proj_{J_e}(J_1)) = \dim(\proj_{J_2}(J_1))=p^*+s+1$,  which means that $\dim(J_e \cap P) = p^*+s+1$ implies that $\proj_{J_e}(J_1) \subseteq P$ and hence $J_e \setminus P$ contains no points collinear with $J_1$. 
\end{proof}

\subsection{Classification of the round-up triples and quadruples} 
We narrow down the possibilities for the quadruples to one of the five below types. 
 \begin{defi} \em   Let $\{J_1,J_2,J_3,J_4\}$ be a $4$-tuple (with $J_3=J_4$ when $\Gamma=\Gamma_{\geq k}$, as then we only need $3$-tuples) such that all pairwise intersections equal a fixed subspace $S$ and denote by $\Delta'$ the residue $\Res_{\Delta}(S)$ and by $J'_a$ the subspace of $\Delta'$ corresponding to $J_d$, $d=1,2,3,4$. Consider the following configurations. The definitions of a hyperbolic line and a hyperbolic $3$-space can be found in Section 2.1. Let  $t$ be an integer with $1 \leq t \leq j-k-1$. 

\begin{itemize} 
\item[$\mathsf{I}\;$]   $\dim(S)=j-1$ and $J'_1,J'_2,J'_3,J'_4$ are on a line in $\Delta'$;
 
\item[$\mathsf{II}$]  $\dim(S)=j-1$ and $J'_1,J'_2,J'_3,J'_4$ are pairwise opposite points in $\Delta'$. If, moreover, these points are on a hyperbolic line, we say that the quadruple is of type $\mathsf{II}^{*}$;

\item[$\mathsf{III}\;$]  $\dim(S)=j-2$ and $J'_1,J'_2,J'_3,J'_4$ are pairwise opposite lines in $\Delta'$  with the property that any line in $\Delta'$ meeting two of them, meets them all. If, moreover, these lines span a hyperbolic $3$-space, we say the quadruple is of type $\mathsf{III}^*$;

\item[$\mathsf{IV}\;$]  $\dim(S)=j-1$-space and $J'_1,J'_2,J'_3,J'_4$ are points of $\Delta'$,  three of which are on a line and opposite the remaining point.

\item[$\mathsf{V(t)}$\;] $\dim(S) \geq k$ and $J'_1,J'_2,J'_3,J'_4$ are $t$-spaces  in $\Delta'$. The subspaces $S \cup U_1,S \cup U_2,S \cup U_3,S \cup U_4$ correspond in $\Delta'$ to points on a line $L$ and in $\Res_{\Delta'}(L)$, the $j$-spaces correspond to pairwise opposite $(t-1)$-spaces defining a hyperbolic $(2t-1)$-space. 
\end{itemize} 
 \end{defi}
 
\begin{rem} \em If $J_3=J_4$, then a type $\mathsf{IV}$ coincides with type $\mathsf{I}$; if not, they are different. \end{rem}
 
Whether or not these $4$-tuples occur in $\Delta$ depends on $\Delta$, or more precisely, on the existence of hyperbolic lines in $\Delta$ and of the presence of hyperbolic quadrangles as hyperbolic subspaces. Hyperbolic lines do not occur precisely if $\Delta$ is a strictly orthogonal polar space (``orth.'' for short in the table below). On the other hand, the strictly orthogonal polar spaces are the only ones containing grids as hyperbolic subspaces (the lines of quadruples of types $\mathsf{III}$ and $\mathsf{III}^*$ are contained in one regulus of a grid). Furthermore, it will also depend on $j$ whether or not the $4$-tuples occur in $\Delta$. We summarise this in the table below, where ``$\mathsf{X}$'' means that the $4$-tuple occurs and ``$\mathsf{x}$'' means that it occurs \emph{if} $|j|<n-1$. Recall that $|j| > 0$. We do not include $\Delta$ hyperbolic and $|i|=|j|=n-1$ here, as this will be a special case anyway (see below).

\begin{table}[h]
\begin{center}
\begin{tabular}{|l||c|c|c|c|c|c|c|c|c|}
\hline
 							& $\mathsf{I}$ 		& $\mathsf{II}$ 		&$\mathsf{II}^*$	&$\mathsf{III}$		&$\mathsf{III}^*$	&$\mathsf{IV}$ 		& $\mathsf{V(1)}$	& $\mathsf{V(>1)}$	 \\ \hline\hline
$\Delta$ orth., not hyperbolic		& $\mathsf{x}$		&  $\mathsf{X}$		&  				& $\mathsf{X}$		&$\mathsf{x}$ 		&$\mathsf{x}$		& 				& $\mathsf{x}$ 			 \\\hline 
$\Delta$ hyperbolic, |i|,|j|<n-1 		& $\mathsf{x}$		& $\mathsf{x}$  	&				&$\mathsf{x}$		&$\mathsf{x}$		& $\mathsf{x}$		& 			 	&$\mathsf{x}$			\\\hline
other $\Delta$ 					& $\mathsf{x}$		& $\mathsf{X}$		& $\mathsf{x}$ 		& 				&				& $\mathsf{x}$ 		& $\mathsf{X}$ 		& $\mathsf{x}$			\\\hline

\end{tabular}
\end{center}
\vspace{-1em}
\caption{Occurrence of $4$-tuples in $\Delta$ in function of $\Delta$.}
\end{table}

It also depends on $\Gamma$ which of the in $\Delta$ occurring $4$-tuples actually occur as a quadruple. To keep track of the different cases, we already give a summary of our results now in the table below, this time not taking into account that some of those types possibly do not occur in $\Delta$. So for a given graph and a given polar space, one has to combine the two tables to know which are the occurring quadruples.

\begin{table}[h]
\begin{center}
\begin{tabular}{|l||c|c|c|c|c|c|c|}
\hline
 									& $\mathsf{I}$ 		& $\mathsf{II}$ 		& $\mathsf{II}^*$ 	& $\mathsf{III}$ 	& $\mathsf{III}^*$	&$\mathsf{IV}$ 		& $\mathsf{V(t)}$ 	 \\ \hline\hline
$\Gamma^\ell_k$: $|j|=n-1$ 				& 				& $\mathsf{X}$ 		&   				& $\mathsf{X}$ 		&				& 				& 					\\\hline  
$\Gamma^\ell_k$: $a,b \geq 0$,  $|j|<n-1$ 	& $\mathsf{X}$		&				&  $\mathsf{X}$		& 				&$\mathsf{X}$ 		&  				&					 \\\hline 
$\Gamma^\ell_k$: $a=-1$, $|i|=|j|<n-1$ 		& $\mathsf{X}$ 		&    				& $\mathsf{X}$ 		&				&				&  				&					 \\\hline
$\Gamma^\ell_k$: $b=-1$ (so $|j|<n-1$)		& $\mathsf{X}$ 		&    				& 				&  				&				&  				& $\mathsf{X}$ 			 \\\hline
$\Gamma_k$ 							& $\mathsf{X}$ 		& $\mathsf{X}$  	&  				& 				& $\mathsf{X}$		&$\mathsf{X}$  		&				 	 \\\hline
$\Gamma_{\geq k}$						& $\mathsf{X}$ 		& $\mathsf{X}$   	& 				& 				&$\mathsf{X}$ 		& $\mathsf{X}$ 		& 					  \\\hline
\end{tabular}
\end{center}
\vspace{1em}
\caption{Occurrence of $4$-tuples as round-up quadruples in function of $\Gamma$.}
\end{table}

We now prove that each of our quadruples is of one of these types. If $\Gamma=\Gamma_k^{\ell}$, it appears that some cases in which $-1 \in \{a,b\}$ behave differently.  So we start with the ``generic case'' in which we do not take special cases into account.

\subsubsection{Most general case}  
 
\begin{lemma}\label{type1234} 
Let $\Gamma$ be one of $\Gamma_{\geq k}$, $\Gamma_k$, $\Gamma_k^{\ell}$. If $\Delta$ is hyperbolic, we assume that $|i|,|j|<n-1$. If $\Gamma=\Gamma_k^\ell$, we  assume $b \geq 0$ and, if moreover $|j|<n-1$, we also assume $a \geq 0$. Then every $\Gamma$-round-up quadruple is of type I, II, III or IV. Type IV does not occur if $\Gamma = \Gamma_k^\ell$ (and coincides with type I if $\Gamma=\Gamma_{\geq k}$). Type II occurs for all graphs, and if $\Gamma = \Gamma_k^{\ell}$ and $|j|<n-1$ then each quadruple of type II is of type II$^{*}$. If $|j|<n-1$, then a quadruple of type III is of type III$^*$.
\end{lemma} 
 
\begin{proof} Let $\{J_1,J_2,J_3,J_4\}$ be a quadruple. By  Lemmas~\ref{dimdoorsnede} and~\ref{lem7a}, we already know that the $j$-spaces have one common intersection $S$ of dimension $s$ with $s\geq k$. Suppose $x_t$ and $x_u$ are collinear points in $J_t \setminus S$ and $J_u\setminus S$, respectively, for $u,t \in \{1,2,3,4\}$ with $J_u \neq J_t$. 
We may apply Lemma~\ref{lem1} without extra conditions on $x_t$ and $x_u$ because, if $a=-1$, we assume $|j|=n-1$, which implies that $\mathsf{P}_t$ and $\mathsf{P}_u$ are empty, so automatically $x_t \in \mathsf{Q}_t$ and $x_u \in \mathsf{Q}_u$; furthermore we assume $b \geq 0$. This lemma then implies that $x_ux_t$ intersects a third member of the quadruple, unless $\Gamma=\Gamma_k^\ell$, $\Delta$ is hyperbolic and $|i|=|j|=n-1$ Let $Q':=\{J'_1,J'_2,J'_3,J'_4\}$ be the set of $(j-s-1)$-spaces in $\Res_\Delta(S)$ corresponding to the quadruple. There are two cases.

\textbf{\boldmath Case 1: There is a pair in $Q'$ which is not opposite.\unboldmath} Suppose $J_1'$ and $J'_2$ are not opposite. This implies that $|j|<n-1$ and then our assumptions are that $a\geq 0$. Moreover, there is a point $x_1 \in J'_1$ collinear with $J'_2$. These previous facts together with the above, imply that $x_1x_2$ intersects $J_3$ or $J_4$ in a point, for any $x_2 \in J_2 \setminus S$. This allow us to apply Lemma~\ref{lem0} on $(x_1,J'_2,J'_3,J'_4)$. We obtain that $s=j-1$, so $J'_1$, $J'_2$, $J'_3$ and $J'_4$ are just points in $\Res_{\Delta}(S)$ (with $J'_1=x_1$); moreover, without loss, $J'_3$ is a point on the line $J'_1J'_2$.

If $\Gamma=\Gamma_{\geq k}$ then $J_3=J_4$ and we are done. If not, there are two possibilities. Firstly, $J'_4$ can be non-collinear with any of $J'_1,J'_2,J'_3$, and then the quadruple is of type IV. Suppose now that $J'_4$ is collinear with $J'_1$. In particular, $J'_1$ and $J'_4$ are not opposite, so we can apply the reasoning of the beginning of the first paragraph on them, and obtain that $J'_2$ or $J'_3$ is on the line $J'_1J'_4$. Anyhow, the lines $J'_1J'_2$ and $J'_1J'_4$ have at least two points in common, so they coincide and the quadruple is of type I. By Lemma~\ref{gemprojectie} (recall $a \geq 0$ in this case), there are no quadruples of type IV if $\Gamma=\Gamma_k^{\ell}$. 

\textbf{\boldmath Case 2: All pairs in $Q'$ are opposite. \unboldmath} We reason in $\Res_{\Delta}(S)$. Let $x_1 \in J'_1$ be arbitrary and let $d=2,3,4$. Consider $P'_d=\proj_{J'_d}(x_1)$. If $P'_2$ is empty, then $s=|j|-1$ and the quadruple is of type II. So suppose $P'_2$ is nonempty.  Suppose first that we are not in the special case of Lemma~\ref{lem1}. As above, it follows that we can apply Lemma~\ref{lem0} on $(x_1,P'_2,P'_3,P'_4)$, which implies that $\dim(P'_d)=0$ and, without loss, $P'_3$ is on the line $x_1P'_2$. So, as $P_d$ is a hyperplane of $J'_d$, $s=|j|-2$ and $J'_1, J'_2, J'_3,J'_4$ are pairwise opposite lines.

Now suppose that we are in the special case, i.e., $\Gamma=\Gamma_k^\ell$, $\Delta$ is hyperbolic and $|i|=|j|=n-1$. Then suppose or a contradiction that $\dim(S) < |j|-2=n-3$, i.e., $\dim(S) \leq n-5$ since $\Delta$ is hyperbolic. Let  $x_1$ be a point in $J_1\setminus S$. Note that $\cod_{S}(x_1^\perp \cap J_2) \geq 2$. We take a hyperplane $H_2$ of $J_2$ through $S$ distinct from $x_1^\perp \cap J_2$. Then there is a point $x_2$ in $(x_1^\perp \cap J_2)\setminus H_2$ such that the line $x_1x_2$ contains a point of $J_3 \cup J_4$. Now taking a hyperplane $H'_2$ in $J_2$ through $\<S,x_2\>$ and distinct from $x_1^\perp \cap J_2$, we likewise obtain a point $x'_2 \in (x_1^\perp \cap J_2)\setminus H'_2$ such that $x_1x'_2$ contains a point of $J_3 \cup J_4$. By our choice of $H'_2$, $x_2 \neq x'_2$ and, moreover, $\<x_2,x'_2\>$ does not meet $S$ (otherwise $x'_2 \in H'_2$ after all). Lastly, we take a hyperplane $H''_2$  in $J_2$ through $\<S,x_2,x'_2\>$ and distinct from $x_1^\perp \cap J_2$, to obtain a point $x''_2 \in (x_1^\perp \cap J_2)\setminus H''_2$ such that $x_1x''_2$ contains a point of $J_3 \cup J_4$. The choice of $H''_2$ implies that  $\<x_2,x'_2,x''_2\>$ is a plane in $x_1^\perp \cap J_2$ which is disjoint from $S$. 
Now, without loss, the lines $x_1x_2$ and $x_1x'_2$ both contain a point $x_3$ and $x'_3$ from $J_3$. But then $x_3x'_3$ and $x_2x'_2$ have to intersect as they are contained in the plane $\<x_1,x_2,x'_2\>$, contradicting that $x_2x'_2$ does not meet $S$. Hence also in this case, $\dim(S)=|j|-2$. As $|j|-2=n-3$ and $\Delta$ is hyperbolic, it follows immediately that the quadruple is of type III.

If $\Gamma=\Gamma_{\geq k}$ then $J_3=J_4$ and we are done. If not, we still need to show that the line $x_1P'_2$ intersects both $J'_3$ and $J'_4$. By the above, we may already assume that $P'_3$ is on this line.  If $P''_4$ is the unique point on $J'_4$ collinear with $P'_3$, then the same arguments as used just above imply that $P_3P''_4$ contains $x_1$ or $P'_2$, so $x_1P'_2 = P'_3P''_4$ and hence $P''_4$ is collinear with $x_1$, implying $P'_4 = P''_4$. This shows that the quadruple is of type III.  

If $\Gamma = \Gamma_k^{\ell}$ and $|j|<n-1$, each quadruple of type II is of type II$^*$, by $\mathsf{(RU1)}$ and Lemma~\ref{opp}. Also, if $|j|<n-1$, then each quadruple of type III is of type III$^*$ because if a point is collinear to two of those lines in $\Delta'$, then it is also collinear with the two other lines as it is collinear to all transversals. 

One can verify that each of the $4$-tuples obtained above indeed satisfies the definition of a round-up quadruple.
\end{proof} 

Lemmas~\ref{type1234} does not yet cover all cases if $\Gamma=\Gamma_k^\ell$.  We deal with the remaining cases separately.

\subsubsection{$\Delta$ hyperbolic and $|i|=|j|=n-1$}\label{hyperbolicn-1} 

\begin{lemma}\label{hypn-1}
If $\Delta$ is hyperbolic and $|i|=|j|=n-1$ then each quadruple $\{J_1,J_2,J_3,J_4\}$ consists of four $j$-spaces intersecting each other in a common subspace $S$ of dimension $n-3$ (in $\Res_\Delta(S)$ they hence correspond to four pairwise opposite lines). 
\end{lemma}
\begin{proof}
Let $\{J_1,J_2,J_3,J_4\}$ be a quadruple. By  Lemmas~\ref{dimdoorsnede} and~\ref{lem7a}, we already know that the $j$-spaces have one common intersection $S$ of dimension $s$ with $s\geq k$.  If $s=n-3$ then it is clear that the quadruple has hyperbolic type II (note that $\Res_\Delta(S)$ has rank 2 so disjoint lines are opposite).

So suppose $s \leq n-5$. Let $J'_1,J'_2,J'_3,J'_4$ be the set of $(j-s-1)$-spaces in $\Res_\Delta(S)$ corresponding to the quadruple. Let $L_1$ be a line in $J'_1$ and let $L_2$ be a line in $J'_2$ collinear with $L_1$ (which exists since $s \leq n-5)$.  By Lemma~\ref{lem1}, $\<L_1,L_2\>$ intersects a third member of the quadruple in a line. If we apply the same reasoning as in Lemma~\ref{lem0} on $(L_1,\proj_{J'_2}(L_1),\proj_{J'_3}(L_1),\proj_{J'_4}(L_1))$, we obtain that $\dim(\proj_{J'_2}(L_1))=1$, i.e.,  $s=n-5$, so $J'_1$, $J'_2$, $J'_3$ and $J'_4$ are pairwise disjoint $3$-spaces in $\Res_{\Delta}(S)$, moreover, without loss,  $\<L_1,L_2\>$ intersects $J'_3$ in a line $L_3$. Moreover, we can show that also $J'_4$ intersects $\<L_1,L_2\>$ in a line. Indeed, let $L_4$ be the unique line in $J'_4$ collinear with $L_1$. Then $\<L_1,L_4\>$ has to intersects at least one of $J_2$, $J_3$ in a line, say $J_2$. As $\proj_{J'_2}(L_1)=L_2$, we obtain that $\<L_1,L_4\>=\<L_1,L_2\>$ intersects each of $J'_1$, $J'_2$, $J'_3$ and $J'_4$ in a line. 

It remains to show that this last possibility does not occur. In $\Res_\Delta(S)$, which is of type $\mathsf{D}_4$, we obtain four pairwise disjoint $3$-dimensional subspaces, say of type $3'$,  such that each $3'$-space intersecting two of them in a line intersects all of them in a line. Applying the triality principle, this amounts to four pairwise opposite points such that each point collinear to two of them is collinear to all of them. As there are no hyperbolic lines in a hyperbolic polar space, this is impossible.
\end{proof}

\subsubsection{The projection of adjacent vertices of $\Gamma$ on each other is their intersection ($a=-1$) }\label{Aleeg} 
 
Let $\{J_1,J_2,J_3,J_4\}$ be a quadruple of $\Gamma=\Gamma_k^\ell$, where $a=|\ell|-|j|-1=-1$, i.e., $|\ell|=|j|$. However, we assumed that, if $\max\{|i|,|j|\} = |\ell|$ then $|j| \leq |i|$, hence also $|i|=|j|$. We will furthermore assume that $|i|=|j| < n-1$, as the case where $|j|=n-1$ is already covered by Lemma~\ref{type1234}. So for this subsection: $i=j < n-1$.

We can prove the following property.
\begin{itemize} 
\item[$\mathsf{(RU2)}$]  Let $L$ be a line containing distinct points $p$ and $p'$ such that $p \in J_u^\perp$ and $p' \in J_t^\perp$ for $u \neq t$. Then $L$ contains a point $q$ with $q \in J_v^\perp \cup J_w^\perp$, where $\{u,t,v,w\}=\{1,2,3,4\}$.
\end{itemize} 

\begin{lemma}\label{ru2} If $\Gamma$ equals $\Gamma_k^{\ell}$, $a=-1$ and $i=j<n-1$, possibly $k=-1$, then $\mathsf{(RU2)}$ is valid for any quadruple. Moreover, $\mathsf{(RU2)}$ remains valid in $\Res_{\Delta}(S')$ for $S' \subseteq J_1 \cap J_2 \cap J_3 \cap J_4$. \end{lemma} 
\begin{proof} 
Recall that $k < \min\{i,j\}$. If $k+1=i$ and $k \geq 0$, then by Lemma~\ref{lem7a} the $j$-spaces all intersect each other in $S$ and $\dim(S)=k=j-1$. If two distinct points of a line $L$ are collinear with $J_1$ and $J_2$ respectively, then $L$ is collinear with $S$. As any point $p_e$ of $J_e\setminus S$ with $e\in\{3,4\}$ is collinear to at least one point of $L$, $J_e$ is collinear with this point. If $k=-1$ and $i=j=0$, then this property is also trivial. 

Now suppose $k+1 <i$. Let $L=pp'$ be a line with $p \in J_1^{\perp}$ and $p' \in J_2^{\perp}$. Suppose for a contradiction that none of its points is collinear with $J_3$ or $J_4$. In particular, $L$ does not meet any of $J_3,J_4$. Hence, we can choose an element $I = \<K,B\> \in \mathsf{N}_{(k)}(J_3,J_4)$ such that $L \subseteq B$ (recall $k+1 <i$). Then $I\cap J_3 = I \cap J_4 = K \subseteq J_3 \cap J_4 = J_1 \cap J_2 \cap J_3 \cap J_4$. If $J_c$, $c\in\{1,2\}$, would be adjacent to $I$, then $I\cap J_c =K$. But then $I\setminus K$ contains $p$ and $p'$, which are collinear with $J_1$ and $J_2$, making $I \sim J_c$ impossible, a contradiction. 
 
Since $L$ is collinear with $S$, it follows that $\mathsf{(RU2)}$ is a residual property.  \end{proof} 
 
\begin{lemma}\label{sdopp} 
Let $\{J_1,J_2,J_3,J_4\}$ be a set of four $j$-spaces having one common intersection $S$ and satisfying $\mathsf{(RU1)}$ and $\mathsf{(RU2)}$. Then at least one pair of them is contained in a singular subspace or is such that their projections on each other equal their intersection. 
\end{lemma}  
\begin{proof} 
If $|j|=n-1$ this is trivial. So suppose $|j|<n-1$ and assume for a contradiction that no such pair exists. If $\dim(S) = j-1$, the lemma is trivial, so we may assume $\dim(S) < j-1$. In $\Delta':=\Res_\Delta(S)$, which has rank at least 3, the $j$-spaces correspond to subspaces $V_1, V_2, V_3$ and $V_4$ of dimension $v$ with $v \geq 1$. We denote the subspaces corresponding to $\<S,U^u_v\>$ by $U^u_v$ as well; note that $U^u_v=\proj_{V_v}(V_u)$, for $u,v \in\{1,2,3,4\}$. By assumption, these are all nonempty.
By looking in $\Res_{\Delta'} (V_1)$, we find a singular $(v+1)$-space $\overline{V}_1$ through $V_1$ in $\Delta'$ such that all points in $\overline{V}_1\setminus V_1$ are collinear with none of  $U^1_2, U^1_3$ and $U^1_4$, and hence collinear with none of $V_2,V_3$ and $V_4$. Likewise, we can find such a singular $(v+1)$-space $\overline{V}_2$ w.r.t.\ $V_2$.  Let $p_1 \in \overline{V}_1 \setminus V_1$ be arbitrary. As $V_2 \nsubseteq p_1^{\perp}$, there is a unique hyperplane $H$ of $\overline{V}_2$ collinear with $p_1$. Clearly, $H \neq V_2$. So let $p_2\in H\setminus V_2$. Denote by $Z$ the subspace $\<p_1,H\>$. 
 
By $\mathsf{(RU1)}$ and possibly by switching the roles of the $j$-spaces (as the above holds for any permutation of $\{1,2,3,4\}$), we may assume that $U^1_2 \subseteq U^3_2=U^4_2$. As $U^1_2$ is not collinear with $p_1$, it is not contained in $H \cap V_2$, and so $\<H \cap V_2,U^v_2\>=V_2$, for all $v\in\{1,3,4\}$. It follows that none of $V_3,V_4$ is collinear with $H \cap V_2$, for this would mean that they are contained in a singular subspace with $V_2$. So both $V_3$ and $V_4$ are collinear with at most a hyperplane of $V_2\cap H$. As they are not collinear with any point of $\overline{V}_2\setminus V_2$,  they are collinear with at most a codimension 1 subspace of $Z$. This shows that there is a line $L$ in $Z$ which is disjoint from $V_3^{\perp} \cap V_4^{\perp}$. We may assume that $L=p_1p_2$ and so $\mathsf{(RU2)}$ is violated. 
\end{proof} 

The conditions (RU1) and (RU2) also appear in \cite{Kas-Mal:13}, though used for round-up triples, and the idea of the previous lemma is taken from the proof of Lemma 4.7 of the same article, and extended to quadruples.

\begin{lemma}\label{aempty} 
Assuming $|i|=|j| < n-1$, the quadruple is of type I or II$^*$. 
\end{lemma}  
 
\begin{proof}  
We already know that the four $j$-spaces intersect each other in a common subspace $S$ of dimension at least $k$ and they satisfy $\mathsf{(RU1)}$ and $\mathsf{(RU2)}$. By Lemma~\ref{sdopp}, there are only two cases to consider.
 
\textbf{Case 1: There is a pair of $j$-spaces contained in a singular subspace.} Suppose $J_1 \perp J_2$. By $\mathsf{(RU1)}$,  $\<J_1,J_2\> \subseteq J_3^\perp \cup J_4^\perp$, so we may assume that $\<J_1,J_2\> \subseteq J_3^\perp$. This implies that $J_1,J_2$ and $J_3$ are contained in a singular subspace $Z$. We will prove that they are all contained in a singular subspace spanned by any pair of the $j$-spaces, afterwards we show that $\dim(S)=j-1$. This is accomplished in the following steps.
\begin{itemize}
\item \textit{Claim 1: $J_4$ has to be collinear with $J_1,J_2$ and $J_3$.} In view of $\mathsf{(RU1)}$ and by switching the roles of $J_1$, $J_2$ and $J_3$ if necessary (they play the same role), we may assume $U^1_4\subseteq U^2_4 = U^3_4$. Assume for a contradiction that $U^1_4 \subsetneq J_4$. Then, in $\Res_{\Delta} (\<S,U^1_4\>)$, $J_1,J_2$ and $J_3$ correspond to collinear $(j-s-1)$-spaces $V_1,V_2$ and $V_3$, respectively, and $J_4$ corresponds to a subspace $V_4$ of dimension at most $(j-s-1)$, which is semi-opposite $V_1$. Take any $(j-s)$-space through $V_1$ which is not collinear with $V_2$ nor with $V_3$.  Clearly, this subspace contains a point which is collinear with $V_4$ but not collinear with $V_2$ or $V_3$. As this violates $\mathsf{(RU1)}$, the claim is proved. Now, put $Z=\<J_1,J_2,J_3,J_4\>$.

\item \textit{Claim 2: $Z$ is generated by any two members of the quadruple.} We first show that at least one of $J_3, J_4$ belongs to $\<J_1,J_2\>$.  Assume for a contradiction that $J_3$ and $J_4$ are both not contained in $\<J_1,J_2\>$. Then there is a hyperplane $H$ of $Z$ containing $\<J_1,J_2\>$ and not containing $J_3$ nor $J_4$. Let $p$ be a point collinear with $H$ but not with $Z$. As $p$ is collinear with $J_1$ and $J_2$ but not with $J_3$ nor with $J_4$, this contradicts $\mathsf{(RU1)}$, showing that one of $J_3, J_4$ is contained in $\<J_1,J_2\>$.

Now suppose that $J_4$ would not be contained in $\<J_1,J_2\>$. Then we apply the same arguments as above to $J_1$ and $J_4$ and obtain that $\<J_1,J_4\>$ contains one of $J_2, J_3$, say $J_2$. Then $\<J_1,J_4\>$ contains $\<J_1,J_2\>$, and as their dimension are equal, $\<J_1,J_2\>=\<J_1,J_4\>$. Since $J_3 \subseteq \<J_1,J_2\>$, we conclude that $Z=\<J_1,J_2\>$ and this proves the claim.

\item \textit{Claim 3: $\dim(S) =j-1$.} Suppose for a contradiction that $\dim(S)<j-1$.  We will exploit property $\mathsf{(RU2)}$. Since this is a residual property, we may assume that $S$ is empty. Our assumption implies $j\geq 1$ and from the previous claim, we know $\dim(Z)=2j+1$. Hence we can find a line $L$ in $Z$ intersecting $J_1$ and $J_2$, but disjoint from $J_3$ and $J_4$. Let $Z'$ be a singular $(2j+1)$-space that intersects $Z$ in $L$ and with $\proj_Z(Z')=L$. Then $\proj_{Z'}(J_1)$ and $\proj_{Z'}(J_2)$ both have dimension $j+1$ whereas $\proj_{Z'}(J_3)$ and $\proj_{Z'}(J_4)$ have dimension $j$. The pairwise intersection of these four subspaces is $L$, as no point of $Z'\setminus L$ is collinear with $Z$. Hence we can find a line $M$ inside $Z'$ intersecting both $\proj_{Z'}(J_1)$ and $\proj_{Z'}(J_2)$, but disjoint from $\proj_{Z'}(J_3)\cup \proj_{Z'}(J_4)$. This contradiction to $\mathsf{(RU2)}$ yields $\dim(S)=j-1$. \end{itemize}
We conclude that the quadruple is of type I. 
\par\medskip

\textbf{Case 2: There is a pair of $j$-spaces whose projections on each other coincide with their intersection.} Suppose $U^4_3$ (and hence also $U^3_4$) is empty. By Case 1, we know that no pair amongst the $j$-spaces is contained in a singular subspace, for otherwise they are all contained in a singular subspace. Consider the following two cases.

\begin{itemize}

\item  \textit{Case 2(a): $\dim(S)=j-1$.} It readily follows that the quadruple is of type II. Combining $\mathsf{(RU1)}$ and Lemma~\ref{opp}, we obtain that the quadruple is of type II$^*$.  

\item \textit{Case 2(b): $\dim(S)<j-1$.}  Let $x_3 \in J_3 \setminus S$ be arbitrary and note that $x_3 \in \mathsf{Q}_3$ as $U^3_4$ is empty. As $\dim(S) < j-1$, there is a point $x_4\in J_4 \setminus S$ (again, automatically, $x_4 \in \mathsf{Q}_4$) collinear with $x_3$, so it follows from  Lemma~\ref{lem1} (recall $i=j <n-1$) that $x_3x_4$ intersects $J_1$ or $J_2$.  Applying Lemma~\ref{lem0} on $(x_3,\proj_{J_4}(x_3),\proj_{J_1}(x_3),\proj_{J_2} (x_3))$, we obtain that $\dim(S) = |j|-2$. 

Let $L_1,L_2,L_3,L_4$ denote the lines in $\Res_{\Delta}(S)$ corresponding to $J_1,J_2,J_3,J_4$, respectively.   By assumption, $L_3$ and $L_4$ are opposite. We claim that they are all pairwise opposite.  Suppose for a contradiction that $y_3$ is a point of  $L_3$ collinear with $L_1$.  By $\mathsf{(RU1)}$, $y_3$ has to be collinear with $L_2$. Let $y_4$ be the unique point on $L_4$ collinear with $y_3$. Then there is an $i$-space $I \in \mathsf{N}_{(k-1)}(J_3,J_4)$ such that the subspace $I'$ corresponding to it in $\Res_{\Delta}(S)$ contains $y_3y_4$. As $I'$ contains $y_3$, a point collinear with $J_1$ and $J_2$, those two $j$-spaces are not adjacent to $I$. This contradiction to the definition of a quadruple proves the claim, as we can now switch the roles of the $j$-spaces.   The same arguments as in the proof of Lemma~\ref{type1234} imply that  the quadruple is of type III$^*$. 

We now show that these kind of quadruples do not occur when $a=-1$. Let $L$ and $L'$ be two distinct transversals of $L_1,L_2,L_3,L_4$, with $L\cap L_d=x_d$ and $L' \cap L_d = x'_d$ for $d\in\{1,2,3,4\}$. Consider $x_3^\perp \cap {x'}_4^\perp$, which is isomorphic to a polar space of rank $n-(j-1)-1\geq 2$ and which contains the points $x'_3$ and $x_4$. In there, take a line $M$ through $x'_3$ and a line $N$ through $x_4$ with $M$ and $N$ opposite, and let $R$ be a line joining a point $m \in M$ and a point $n \in N$, with $m \neq x'_3$ and $n \neq x_4$. Now note that $m$ is collinear with $L_3$ (since $m \in x_3^\perp$ is collinear with $x'_3$) and not with $L_4$ (since $m$ is not collinear with $x_4$ because $n \neq x_4$), likewise, $n$ is collinear with $L_4$ but not with $L_3$. Consequently, $\mathsf{(RU2)}$ implies that there is a point $r \in R$ collinear with $L_1$ or $L_2$, say $L_1$. Then $r$ is collinear with $L$ (as it is collinear with $x_3$ and $x_1 \in L_1$) and, likewise, with $L'$; and therefore $r$ is collinear with both $L_3$ and $L_4$. In particular it follows that $r \notin \{m,n\}$, but then $r,m \perp L_3$ implies $n \perp L_3$, a contradiction. This implies that there are no quadruples of type III$^*$.
\end{itemize}
Again, it is easily verified that each of those $4$-tuples obtained above indeed satisfies the definition of a round-up quadruple.
\end{proof} 

 Next, we deal with the case where $b=-1$.
 
\subsubsection{Adjacent vertices are contained in a singular subspace ($b=-1)$} 
 
Clearly, if $b=-1$ then $|i|,|j|<n-1$. The fifth type of $4$-tuple emerges here. 
\par\bigskip
\begin{lemma}\label{b-1} The quadruple is of type I or V$(t)$. \end{lemma} 
\begin{proof} 
Recall that the $j$-spaces intersect in a fixed subspace $S$ with $\dim(S) \geq k$.  According to Lemma~\ref{gemprojectie}, all pairs of them have the same mutual position, so all of them play the same role. Depending on $U_1$, there are three cases. We will apply Lemma~\ref{lem1} again (note that $|j|<n-1$). 
\begin{itemize}
\item \textit{Case 1: $U_1 = J_1 \setminus S$.} In this case, $J_1 \perp J_2$ and hence we can apply Lemma~\ref{lem1} on every pair of points $x_c \in J_c\setminus S$, $c=1,2$. Proceeding like in the proof of Lemma~\ref{type1234}, we obtain that the quadruple is of type I.

\item \textit{Case 2: $U_1$ is empty.} Suppose $I \in \mathsf{N}(J_1,J_2)$. Then $I \cap J_d \subseteq S$ for all $d \in \{1,2,3,4\}$, as $b=-1$. Combining $\mathsf{(RU1)}$ with Lemma~\ref{opp}, we conclude that $\mathsf{N}(J_1,J_2)=\mathsf{N}(J_1,J_2,J_3,J_4)$, a contradiction to the definition of a quadruple.

\item \textit{Case 3: $\emptyset \neq U_1 \subsetneq J_1 \setminus S$.} Let $x_1\in U_1$ and $x_2 \in U_2$ be points. By Lemma~\ref{lem1}, $x_1x_2$ meets $J_3 \cup J_4$, more precisely, $x_1x_2$ meets $U_3\cup U_4$. Applying Lemma~\ref{lem0} on $(x_1,S \cup U_2,S \cup U_3,S \cup U_4)$, we obtain $\dim(S)  = \dim(S \cup U_d) -1$  for all $d\in\{1,2,3,4\}$ and at least one of $S \cup U_3,S \cup U_4$ lies in $\<S,U_1,U_2\>$. Interchanging the roles of the $j$-spaces like before, we obtain that both $S \cup U_3$ and $S \cup U_4$ belong to $\<S,U_1,U_2\>$. In $\Res_{\Delta} (\<S,U_1,U_2\>)$ (which has rank $n-(s+2)-1$), this translates to four $(j-s-2)$-spaces which are pairwise opposite. Note that these are not maximal singular subspaces, for otherwise $j=n-2$, and together with $a \geq 0$ (as $b=-1$) this would imply $|\ell|=n-1$, which we only allow when $|i|=|j|=n-1$. Again by $\mathsf{(RU1)}$ and Lemma~\ref{opp}, every point collinear with two of them is collinear with all of them, i.e., they define a hyperbolic $(2t+1)$-space with $t= j-s-2$. Hence, the quadruple is of type V$(t)$.
\end{itemize}
Also here, it is easily verified that each of those $4$-tuples obtained above indeed satisfies the definition of a round-up quadruple.
 \end{proof}

\subsection{Constructing $\mathsf{G}_j$ or $\mathsf{G}'_j$} 
 
With the just obtained classification, we want to construct $\mathsf{G}_j$ or $\mathsf{G}'_j$. By Corollary~\ref{gras2} or  Proposition~\ref{gras1}, respectively, this would finish the proofs of Main Theorems~\ref{main1} and~\ref{main2} in the case where $k\geq0$.  We start from the graph $\Gamma'$, which we define as the  graph having $\Omega_j$ as vertices, adjacent whenever they are contained in a quadruple. We aim to identify the types of the quadruples. The following notion will be useful. 
 
\begin{definition} \em Let  $J_1,J_2$ be adjacent vertices of $\Gamma'$. A \emph{near-line (based at $\{J_1,J_2\}$)} is defined as the union of all quadruples containing $\{J_1, J_2\}$. We denote this set of $j$-spaces by $[J_1,J_2]$.
The \emph{type set} of $[J_1,J_2]$ is the set of types of the quadruples containing $\{J_1,J_2\}$. If this set contains only one element, we call this element the \emph{type} of $[J_1,J_2]$.
\end{definition}

\begin{lemma}\label{geom} Suppose $J_1$ and $J_2$ are contained in a quadruple and $\dim(J_1 \cap J_2) \geq j-2$. Let $J,J'$ be two members of $[J_1,J_2]$. Then $J \cap J' = J_1 \cap J_2$. Moreover, if $\dim(J_1 \cap J_2)=j-2$, then in $\Res_\Delta(J_1 \cap J_2)$, the lines $L$ and $L'$ are contained in the regulus determined by the lines $L_1$ and $L_2$ corresponding to $J_1$ and $J_2$. 
\end{lemma} 
\begin{proof}
 Observe that $J \cap J'$ always contains $J_1\cap J_2$ for any two distinct members $J,J'$ of $[J_1,J_2]$, since Lemma~\ref{lem7a} implies that $J$ and $J'$ both contain $J_1 \cap J_2$. If  $\dim(J_1 \cap J_2)=j-1$, then of course $J \cap J' =J_1 \cap J_2$. So suppose  $\dim(J_1 \cap J_2)=j-2$ from now on.

In $\Delta':=\Res_\Delta(S)$, the $j$-spaces correspond to respective lines $L_1$, $L_2$, $L$ and $L'$. If $|j|<n-1$, then $L_1$ and $L_2$ determine a hyperbolic $3$-space which has the structure of a hyperbolic quadrangle. As $L$ and $L'$ both belong to the regulus determined by $L_1$ and $L_2$, it follows that there is a quadruple of type $\mathsf{III}^*$ containing $J$ and $J'$. If $|j|=n-1$, we need to do some more work. Note that in that case, $\Delta'$ is a generalised quadrangle. If $L$ and $L'$ would intersect in a point $p$, then a line $M$ intersecting $L_1$ and $L_2$ and not containing $p$, will intersect $L$ and $L'$ in distinct points (since each of $L$, $L'$ is contained in some quadruple together with $L_1$ and $L_2$), clearly a contradiction. Hence $L$ and $L'$ are disjoint and hence opposite. We only need to show that each line intersecting $L$ and $L'$ also intersects $L_1$ and $L_2$. So let $M$ be a line intersecting $L$ and $L'$ in points $q$ and $q'$, respectively. Then the unique line $N$ through $q$ intersecting $J_1$ will also intersect $J_2$ since $J,J_1$ and $J_2$ are contained in a quadruple. But since also $J_1,J_2$ and $J'$ are contained in a quadruple, $N$ also intersects $J'$. As $L$ and $L'$ are opposite, $M=N$ and hence $M$ indeed intersects all four lines.  The lemma is proven. 
\end{proof}

\begin{rem}\label{not} \em Like before, we write $S= J_1 \cap J_2$. Let $S'\subseteq S$. If $\dim(S')=j-1$, the subspaces  in $\Res_{\Delta}(S')$ corresponding to the members $J_a$ of $[J_1,J_2]$  are points, denoted by $p_a$; likewise, if $\dim(S')=j-2$, the corresponding subspaces are lines, denoted by $L_a$.
\end{rem}

Note that no near-line will have type IV. Indeed, type IV quadruples occur when $\Gamma=\Gamma_k$ and $|j|<n-1$, and in this case there are quadruples of type I, II and IV (possibly also of type III is $\Delta$ is orthogonal). If $J_1$ and $J_2$ are $j$-spaces with $\dim(J_1 \cap J_2)=j-1$, then if $J_1 \perp J_2$, the type set of $[J_1,J_2]$ is $\{I,IV\}$, and if they are not collinear, it is $\{II,IV\}$. In this particular case, no near-line $[J_1,J_2]$ will have type I or type II either, since we will in both cases find a quadruple of type IV that contains $\{J_1,J_2\}$. 

We now focus on near-lines having a singleton as their type set. By the above, we do not need to consider near-lines of type IV (the above also implies that near-lines of type I or II would not occur if type IV quadruples occur). We neither consider near-lines of type V(t), existing or not, as we will not need them. 

\begin{lemma}\label{singletype}
Let $[J_1,J_2]$ be a near-line with type I, II$^*$ or III$^*$ if $|j|<n-1$, or, if $|j|=n-1$, type II or III. Then in $\Res_\Delta(J_1 \cap J_2)$, we get the following respective sets for $[J_1,J_2]$:
\begin{compactenum} 
\item[$(\mathsf{I})$] The set of points on the line spanned by $p_1$ and $p_2$, 
\item[$(\mathsf{II})$] the set of points which are opposite both $p_1$ and $p_2$, 
\item[$(\mathsf{II^*})$] the set of  points of the hyperbolic line spanned by $p_1$ and $p_2$, 
\item[$(\mathsf{III}, \mathsf{III}^{*})$] the set of lines of the regulus of the grid determined by $L_1$ and $L_2$, 
\end{compactenum} 
In particular, each four elements occurring in $[J_1,J_2]$ form a quadruple, which is of the same type as $[J_1,J_2]$, and each two distinct members $J'_1,J'_2$ of $[J_1,J_2]$ satisfy $[J_1,J_2]=[J'_1,J'_2]$. 
\end{lemma}

\begin{proof} Let $J$ and $J'$ be two elements of $[J_1,J_2]$. By Lemma~\ref{geom}, $J_1 \cap J_2 = J'_1 \cap J'_2$.

If $[J_1,J_2]$ is of type I, then in $\Res_\Delta(J_1,J_2)$, those four $j$-spaces correspond to points $p_1$, $p_2$, $p'_1$, $p'_2$, such that both $p'_1$ and $p'_2$ are on the line $p_1p_2$. As such, $p'_1$ and $p'_2$ determine the same line. It follows that each point on $p_1p_2$ corresponds to a $j$-space in $[J_1,J_2]$ and vice versa, making it clear that $[J_1,J_2]=[J'_1,J'_2]$ and that each four points on this line correspond to four $j$-spaces in a quadruple of type I. If $[J_1,J_2]$ is of type II$^*$, the same argument applies, the only difference being that $p_1$ and $p_2$ determine a hyperbolic line instead of an ordinary line.  

If $[J_1,J_2]$ is of type II (so $|j|=n-1$) then $p_1$, $p_2$, $p'_1$ and $p'_2$ are all just points in $\Res_\Delta(J_1 \cap J_2)$, which is a polar space of rank 1, each point of which corresponds to a $j$-space in $[J_1,J_2]$ and vice versa. It is again clear that any four such points then determine a quadruple of type II.  

If $[J_1,J_2]$  has type III (so $|j|=n-1$) or $\text{III}^{*}$ ($|j|<n-1)$, it follows from Lemma~\ref{geom} that, in $\Res_\Delta(J_1 \cap J_2)$, the near-line $[J_1,J_2]$ corresponds to the regulus determined by the respective lines $L_1$ and $L_2$ corresponding to $J_1$ and $J_2$. As this regulus is determined by any two of its members, the assertion follows.  
\end{proof}

We start with the special case where $|i|=|j|=n-1$ and we encounter Main Theorem~\ref{main1}$(i)$. 

\subsubsection{Maximal singular subspaces ($|j|=n-1$)}
If $|j|=n-1$, then $|\ell|=|j|$ and since we assume that if $\max\{|i|,|j|\}=|\ell|$ then $|j|=\min\{|i|,|j|\}$, we have $|i|=|j|=n-1$. Note that $\Gamma_k^\ell = \Gamma_k$ in this case. If $\Delta$ is not hyperbolic, then for all $\Gamma\in\{\Gamma_k^\ell,\Gamma_k,\Gamma_{\geq k}\}$, the occurring triples/quadruples can only be of types II or III.  Consequently,  $\Gamma'$ is independent of $\Gamma$, so we can treat $\Gamma_k^\ell$, $\Gamma_k$ and $\Gamma_{\geq k}$ at the same time. We aim to separate type II from type III. Of course we only need to do this when type III really occurs, so we may assume that $\Delta$ is orthogonal.

\textbf{Case 1: $\Delta$ is neither parabolic nor hyperbolic.}  The following lemma distinguishes between quadruples of types II and III, making use of the corresponding near-lines. Note that, since we only have types II and III, each near-line has a type (cf.\ Lemma~\ref{geom}).

\begin{definition}\label{prec} \rm Let $\mathcal{L}$ and $\mathcal{L'}$ be two near-lines having one element $J_0$ in common. We say that $\mathcal{L} \preccurlyeq \mathcal{L'}$ if there is at least one $j$-space $J \notin \mathcal{L} \cup \mathcal{L'}$ which is $\Gamma'$-adjacent to all members of $\mathcal{L}$ and such that  each near-line through $J$ meeting $\mathcal{L}$ also meets $\mathcal{L'}$; if moreover, for any of those $j$-spaces $J$, there is a near-line through it that meets $\mathcal{L'}$ without meeting $\mathcal{L}$, then we write $\mathcal{L} \prec \mathcal{L'}$. \end{definition}

Despite the suggestive notation, we do not claim or intend to prove that $\preccurlyeq$ is an order relation. So in principle $\mathcal{L} \prec \mathcal{L}'$ and $\mathcal{L}' \prec \mathcal{L}$ is possible at the same time (this is because of the dependence on $J$). This will not matter for the next lemma. 

\begin{lemma}\label{onderscheidn-1} 
Suppose $\Delta$ is neither parabolic nor hyperbolic and $i=j=n-1$. A near-line $\mathcal{L}$ is of type III if and only if there is a near-line $\mathcal{L'}$ such that $\mathcal{L} \prec \mathcal{L'}$. 
\end{lemma} 
\begin{proof} \textbf{Suppose first that $\mathcal{L}:=[J_0,J_1]$ is of type III. We show that there is a near-line $\mathcal{L}'$ such that $\mathcal{L} \prec \mathcal{L}'$.} By definition, $\dim(J_0 \cap J_1)=n-3$. Take any $j$-space $J_2$ such that $J_0 \cap J_2$ is an $(n-2)$-space containing $J_0 \cap J_1$ and put $\mathcal{L'}:=[J_0,J_2]$. In $\Res_{\Delta}(J_0 \cap J_1)$, which is a generalized quadrangle $Q$, the lines $\mathcal{L}$ and $\mathcal{L'}$ correspond to a regulus and a pencil, respectively, sharing a line. In the dual generalized quadrangle $Q^*$, they correspond to a hyperbolic line $L$ and an ordinary line $L'$, respectively, meeting in a point $p_0$.

Let $p$ be a point of $L \setminus \{p_0\}$  and $x$ the unique point on $L'$ collinear with $p$. As $x$ is then collinear with two points of $L$, $x$ is collinear to each point of $L$. Recall that (since $\Delta$ is orthogonal) we have that each hyperbolic line is the common perp of two non-collinear points, in particular, $L= \{x,x'\}^\perp$ for some point $x'$ not collinear to $x$. 
Now consider the structure $\mathcal{P}_x$ induced by the ordinary and hyperbolic lines in the pencil $x^\perp$. 
Each of its ordinary lines contains $x$ as otherwise we have a triangle. Since each hyperbolic line $h$ in $\mathcal{P}_x$ intersects two lines of $x^\perp$ (in points distinct from $x$), $h=\{x,x''\}^\perp$ for some point $x''$ not collinear to $x$. In particular, $h$ intersects $L'$.
Suppose for a contradiction that each hyperbolic line meets $L$. Seeing that $\Delta$ is Moufang (as its rank is at least three), it follows by transitivity that any two hyperbolic lines in $\mathcal{P}_x$ meet, implying that $\mathcal{P}_x$ is a projective plane. By a result of Schroth (\cite{schroth}), this means that $Q^*$ is a symplectic quadrangle. But then $Q$, and therefore also $\Delta$, is parabolic. This contradiction implies that there is a hyperbolic line $h$ in $\mathcal{P}_x$ not intersecting $L$. Let $q \in h$ be a point not on $L'$, then each hyperbolic line through $q$ that meets $L$ is contained in $\mathcal{P}_x$ and therefore it also meets $L'$, so $\mathcal{L} \preccurlyeq \mathcal{L'}$. As $h$ meets $L'$ but does not meet $L$, $\mathcal{L} \prec \mathcal{L'}$, as required. 

\textbf{For the converse, suppose $\mathcal{L}:=[J_0,J_1]$ is of type II. We show that there is no near-line $\mathcal{L'}$ with $\mathcal{L} \prec \mathcal{L}'$.} Assume for a contradiction that there is a near-line $\mathcal{L}'$ with $\mathcal{L}\prec\mathcal{L}'$. Let $J$ be a $j$-space as in Definition~\ref{prec}. Put $S=J_0 \cap J_1$. We claim that $\dim(J\cap S)=n-3$ and that each $j$-space in $\mathcal{L}'$ contains $J \cap S$.

Firstly, $J \notin \mathcal{L}$ implies $S \nsubseteq J$ (i.e., $\dim(S \cap J) < n-2$); secondly, $J$ is $\Gamma'$-adjacent to each member of $[J_0,J_1]$ so in particular it has to intersect both $J_0$ and $J_1$ in at least an $(n-3)$-space, and since $J$ cannot contain points from both $J_0 \setminus S$ and $J_1 \setminus S$ (as those are not collinear), so $\dim(S \cap J)=n-3$.
Now $J\setminus S$ contains a line which has a unique point collinear with $S$ (if all its points were collinear with $S$ then $S \subseteq J$), so $\mathcal{L}$ contains a unique element, say $J^*$, such that $\dim(J \cap J^*)=n-2$.
By definition, $\mathcal{L'}$ does not contain $J$ and intersects $\mathcal{L}$ and all near-lines $[J,J_L]$ with $J_L \in \mathcal{L}$ in a unique $j$-space.
Take $J_L \in \mathcal{L}\setminus\{J^*\}$. Then $J \cap J_L = J \cap S$, and hence $[J,J_L]$ is of type III, but more importantly, each $j$-space in $[J,J_L]$ contains $J \cap S$. Consequently, at least two members of $\mathcal{L}'$ contain $J \cap S$, hence, so does each member of $\mathcal{L}'$. This shows the claim. This allows us to restrict ourselves again to the generalised quadrangle $Q^*$ which is the dual of $Q=\Res_\Delta(J \cap S)$. 

Suppose first that $\mathcal{L'}$ has type III. In $Q^*$, $\mathcal{L}$ again corresponds to an ordinary line $L$, the $j$-space $J$ corresponds to a point $q \notin L$ and the line $L'$ corresponding to $\mathcal{L}'$ does not contain $q$  and meets each line $qp_L$ with $p_L \in L$. Let $x$ be the unique point on $L$ on an ordinary line with $q$ (this points corresponds to $J^*$). Then $\mathcal{P}_x$ contains $L$ and all lines $qp_L$ with $p_L \in L$, and hence $L' \subseteq \mathcal{P}_x$ too. But then each line through $q$ that meets $L'$ also belongs to $\mathcal{P}_x$ and, as such, it intersects $L$ too (as we deduced before). This contradicts $\mathcal{L}\prec\mathcal{L}'$. 

Next, suppose $\mathcal{L'}$ has type II. Then, in $Q^*$ and with the same notation as above, $L'$ is an ordinary line, which hence contains $x$. Again, $q\cup L \cup L' \subseteq \mathcal{P}_x$. Then it is clear that each line through $q$ that intersects $L'$ in a point distinct from $x$ is a hyperbolic line, and each such hyperbolic line has to intersect $L$ as well, again contradicting $\mathcal{L} \prec \mathcal{L'}$. 
\end{proof} 
 
As this allows us to recognise quadruples of type III, we can remove the edges in $\Gamma'$ between $j$-spaces that are contained in such a quadruple, hereby obtaining that all remaining edges join $j$-spaces that intersect each other in a $(j-1)$-space, as they are contained in a quadruple of type II. The resulting graph is $\mathsf{G}_{n-1}$; the result follows.

 \textbf{Case 2: $\Delta$ is hyperbolic.} By Lemma~\ref{hypn-1}, there is only one type of quadruple here, and these are such that $\dim(S)=n-3$. It follows that $\Gamma' = \mathsf{G_{n-1}}$. Therefore each element of $\Aut(\Gamma)$ is induced by an automorphism of $\Delta$, possibly up to a duality. The duality occurs precisely if either $\{i,j\}=\{(n-1)',(n-1)''\}$ (interchanging the biparts) or, if $n=4$ and $i=j$, then an $i$-duality also induces an automorphism of $\Gamma$ (not interchanging the biparts).

\textbf{Case 3: $\Delta$ is parabolic.}   We exploit the natural embedding of $\Delta$ in a hyperbolic polar space $\Delta'$ of rank $(n+1)$ (defined over the same field as $\Delta$). Let $\Gamma=\Gamma_k = \Gamma_k^\ell$ or $\Gamma = \Gamma_{\geq k}$. In Example~\ref{special1} at the very beginning of this paper, we explained that, if $k=-1$, $\Gamma$ is isomorphic to some graph $\Gamma'$ defined over $\Delta'$ and hence $\Aut(\Gamma) \cong \Aut(\Gamma')$, and by the previous case, it then follows that each automorphism of $\Delta'^b$ induces an automorphism of $\Gamma'$ and vice versa. 

\textit{Claim: for $k \geq 0$, the automorphisms of $\Gamma$ are also induced by automorphisms of $\Delta'^b$, but only those automorphisms of $\Delta'^b$ preserving $\Delta$  (i.e., the automorphisms of $\Delta$) will induce automorphisms of $\Gamma$.}

Inspired by Special case~\ref{special1}, we first define a graph $\Gamma'$, associated to $\Delta'$, such that there is a bijection between the vertices of $\Gamma$ and those of $\Gamma'$. Let $\mathsf{M}_1$ be one family of MSS of $\Delta'$ and let $\mathsf{M}_2$ be the family of MSS such that $\mathsf{M}_1=\mathsf{M}_2$ if $n-k$ is even and $\mathsf{M}_1 \neq \mathsf{M}_2$ if $n-k$ is odd. Let $m_c$ denote the type of the elements in $\mathsf{M}_c$. If $\Gamma=\Gamma_k$, then we define for each $k \geq 0$ the graph $\Gamma'$ as $\Gamma_{m_1,m_2;k}^{n+1}(\Delta')$; if $\Gamma=\Gamma_{\geq k}$, then, also for each $k \geq 0$, we define $\Gamma'$ as $\Gamma_{m_1,m_2,\geq k}^{n+1}(\Delta')$. 

For each member $X$ of a bipartition class $C_c$ of $\Gamma$, we denote by $\beta_c(X)$ the unique element of $\mathsf{M}_c$ going through it, $c=1,2$.  Then $\beta_1 \times \beta_2$ gives a bijection between the vertices of $\Gamma$ and $\Gamma'$, and $\Gamma'$ is chosen such that if $(I,J)$ is an adjacent pair in $\Gamma$, then $(\beta_1(I),\beta_2(J))$ is an adjacent pair in $\Gamma'$ and moreover, such that there are adjacent pairs $(I',J')$ in $\Gamma'$ that intersect in a $k$-space (the definition of $\Gamma'$ is not entirely canonical, but other sensible choices for a graph isomorphic to $\Gamma$ would behave similarly so we take this one as an example).

Yet, we next show that the fact that $k \geq 0$ will imply that there are adjacent pairs $(I',J')$ in $\Gamma'$, for which $(\beta_1^{-1}(I'),\beta_2^{-1}(J'))$ is not an adjacent pair of $\Gamma$. To see this, suppose $I'$ and $J'$ are such that $I' \cap J'$ is a $k$-space in $\Delta'$ which is not contained in $\Delta$ and note that $I:=\beta_1^{-1}(I')=I' \cap \Delta$, likewise $J:=\beta_2^{-1}(J)=J' \cap \Delta$. But then $I \cap J=I' \cap J' \cap \Delta$ is only a $(k-1)$-space, and hence $(I,J)$ is not an adjacent pair in $\Gamma$ (for all $\Gamma$ under consideration). The mapping $\beta_1 \times \beta_2$ is an embedding of $\Gamma$ in $\Gamma'$ (since adjacency is preserved in one direction). 

Now suppose that $\sigma$ is an automorphism of $\Delta'^b$ with $\sigma(\Delta) \neq \Delta$, then we still need to show that $\sigma$ cannot induce an automorphism of $\Gamma$. As $\sigma$ does not preserve $\Delta$, there is a $k$-space  $K$ in  $\Delta$ such that $\sigma(K) \nsubseteq \Delta$. Thus, if $(I',J')$ is an adjacent pair of $\Gamma'$ with $I' \cap J'  =K$, then $\sigma(I') \cap \sigma(J') = \sigma(K) \nsubseteq \Delta$ and so $\beta_1^{-1}(\sigma(I')) \cap\beta_2^{-1}(\sigma(J')) = \sigma(I') \cap \sigma(J') \cap \Delta$ is only a $(k-1)$-space, implying that $(\beta_1^{-1}(\sigma(I')),\beta_2^{-1}(\sigma(J')))$ is not adjacent in $\Gamma$.  This shows the claim.
 
\textbf{Conclusion.} We know that each automorphism $\alpha$ of $\Gamma'$ is induced by an automorphism $\tilde{\alpha}$ of $\Delta'^b$ (see previous case). The above implies that $\alpha$ preserves $\Gamma$-adjacency (we view $\Gamma$ as embedded in $\Gamma'$) if and only if $\tilde{\alpha}$ preserves $\Delta$. Hence $\Aut(\Gamma) \cong \{ \alpha \in \Aut(\Gamma') \mid \tilde{\alpha}(\Delta)=\Delta\}$. With other words, each automorphism of $\Gamma$ is induced by an automorphism of $\Delta$, as required.

\subsubsection{The $(k,\ell)$-Weyl graphs: $|j|<n-1$ and $b\neq -1$} 
 
In this case, the type sets of the quadruples are the singletons I, II$^*$ or III$^*$. Let $J_1$ and $J_2$ be adjacent vertices of $\Gamma'$.  

 The following lemma will allow us to separate type III$^*$ from the others, allowing us to remove the edges in $\Gamma'$ between $j$-spaces that intersect each other in a $(j-2)$-space,  hereby obtaining $\mathsf{G}_j$ or $\mathsf{G}'_j$. We may hence assume that quadruples of type III$^*$ occur, otherwise we immediately obtain $\mathsf{G}_{j}$ or $\mathsf{G}'_j$. Note that type I quadruples always occur (since $|j|<n-1$). Furthermore, only in mixed polar spaces quadruples of type II$^*$ and of type III$^*$ can both occur; other polar spaces admit at most one of these types.
 
\begin{lemma}\label{III} Let $J$ be a $j$-space, not contained in $[J_1,J_2]$ and $\Gamma'$-adjacent to all members of $[J_1,J_2]$, except one or two members $J'$ and $J''$ (possibly $J'=J''$). Then $\dim(J \cap J^*) = j-2$ for all $J^* \in [J_1,J_2]\setminus\{J',J''\}$. Moreover, for each pair of $j$-spaces $J$ and $J^*$ intersecting each other in a $(j-2)$-space, we can always find a near-line containing $J^*$ and not containing $J$, such that this situation occurs. 
\end{lemma}  
 
\begin{proof} 
Let $J$ be as stated. We may assume that $J_1 \sim J \sim J_2$, as $[J_1,J_2] = [J_3,J_4]$ for all distinct $J_3$, $J_4$ in $[J_1,J_2]$, as noted before.

\textbf{\boldmath Claim 1: The dimension of $J\cap S$ is at least $j-2$.\unboldmath} Suppose for a contradiction that $\dim(J\cap S) \leq j-3$. 
Firstly, let $[J_1,J_2]$ be of type I or II$^*$.  As $\dim(J \cap S) \leq j-3$, it follows from $\dim(S)=j-1$ that  $\dim(J_3 \cap J_c) \leq j-2$ for $c=1,2$. However, $J \sim J_c$ so $\dim(J \cap J_c)=j-2$ (and hence $J$ and $J_c$ are contained in a quadruple of type III$^*$) and $\dim(J\cap S)=j-3$. This also means that $J_3$ contains a point $p_c$ from $J_c \setminus S$. But then the point $p_2$ is collinear with $S$ and with $p_1$, and hence with $\<S,p_1\>=J_1$. As $p_2 \in J\setminus J_1$, this contradicts that $J$ and $J_1$ are contained in a quadruple of type III$^*$, as the lines corresponding to $J$ and $J_1$ in $\Res_\Delta(J \cap J_1)$ are clearly not opposite. This contradiction implies that $\dim(J \cap  S) \geq j-2$ in this case. 

Secondly, suppose $[J_1,J_2]$ is of type III$^*$. First note that $\dim(J \cap J_c) \geq j-2$ for $c=1,2$. Then $\dim(J \cap S) \leq j-3$ implies that $J$ contains at least a point in $J_1 \setminus S$ and in $J_2 \setminus S$, say $p_1$ and $p_2$. Moreover, $J$ cannot contain a line in $J_1 \setminus S$, since no line of $J_1 \setminus S$ is collinear with $p_2$. But that means that $\dim(J \cap S)=j-3$ and that $\dim(J\cap J_c)=j-2$. In $\Res_{\Delta}(J \cap S)$, the members $J_r$ of $[J_1,J_2]$ correspond to planes $\pi_r$ through some point $x$ and $J$ corresponds to a plane $\pi$ (not containing $x$) intersecting the planes $\pi_r$ in respective points $p_r$. Let $J_3$ be the member of  $[J_1,J_2]$ with $J \nsim J_3$. Then the lines corresponding to  $\pi$ and $\pi_3$ in $\Res_{\Delta}(\<J\cap S,p_3\>)$ are not opposite, so there is some line $M$ in $\pi$ through $p_3$ collinear with $\pi_3$. As the point $p_1$ is not collinear with $\pi_3$, it is not on $M$. Consequently, $x$ is collinear with $\<p_1,M\>=\pi$. However, this implies that $J_1$ and $J$ do not correspond to opposite lines in $\Res_{\Delta}(\<J\cap S,p_1\>)$, contradicting the fact that $J \sim J_1$. Hence, in this case, $J$ even contains $S$.

\textbf{\boldmath Claim 2: For all $J_r \in [J_1,J_2]\setminus\{J',J''\}$, we have $\dim(J \cap J_r)=j-2$.\unboldmath} 
Now we know that $\dim(S \cap J) \geq j-2$, we take a $(j-2)$-space $S' \subseteq J \cap S$ and consider $\Res_{\Delta}(S')$. Firstly, if $[J_1,J_2]$ is of type I, then in $\Res_{\Delta}(S')$ it corresponds to a set of lines through a point $x$ and contained in a plane $\pi$. Secondly,  if $[J_1,J_2]$ is of type II$^*$,  then in $\Res_{\Delta}(S')$ it corresponds to  the set of lines through a point $x$ such that, in $\Res_{\Delta}(\<S',x\>)$, this set corresponds to a hyperbolic line. Lastly, if $[J_1,J_2]$ is of type III$^*$, then in $\Res_{\Delta}(S')$ it corresponds to one regulus of a hyperbolic quadric. Denote by $L$ the line corresponding to $J$ and denote by $J'$ and $J''$ the member(s) of $[J_1,J_2]$ not adjacent to $J$. 

Suppose that $\dim(J\cap J^*) =j-1$ for at least two members $J^* \in [J_1,J_2]$, then for all members. We reason in $\Res(S')$ (using the notation settled in Remark~\ref{not}).
\begin{compactenum}[$-$]
\item If $[J_1,J_2]$ is of type I, then $L$ either contains $x$ or is contained in $\pi$. Either way, we conclude that $\dim(J \cap J^*)=j-1$ for all $J^* \in [J_1,J_2]$. But then $J \sim J^*$ for one or for all $J^* \in [J_1,J_2]$ (`one' occurs if $L$ contains $x$, is collinear with a unique line of $\pi$ and when there are no quadruples of type II$^*$).
\item  If $[J_1,J_2]$ is of type II$^*$, then $L$ goes through $x$. Since quadruples of type I and II$^*$ both occur now,  $J \sim J^*$ for all $J^* \in [J_1,J_2]$.
\item  If $[J_1,J_2]$ is of type III$^*$, then $L$, intersecting two lines of the regulus, intersects them all. Depending on whether II$^*$ quadruples occur,  $J$ is either adjacent to none or to all members of $[J_1,J_2]$.
\end{compactenum}
As in none of the previous cases, $J$ is adjacent to all members of $[J_1,J_2]$ except one or two, we conclude that $J$ cannot intersect more than one members of $[J_1,J_2]$ in a $(j-1)$-space. 

Next, suppose $\dim(J \cap J^*)=j-1$ for a unique member $J^*$ of $[J_1,J_2]$ (if no such member exists, nothing needs to be shown). Note that this situation does not occur if $[J_1,J_2]$ is of type I, since no line meeting $L^*$ can be opposite the lines corresponding to the other members of $[J_1,J_2]$.  We show that $J^*\in\{J',J''\}$. We will reason in $\Res(S)$; recall that we settled our notation already in Remark~\ref{not}.

\begin{itemize}
\item \textit{Suppose there are quadruples of type II$^*$.} In this case, $J \sim J^*$, hence $J^*\notin\{J',J''\}$. Then there is a point $z \in L'$ collinear to $L$. If $[J_1,J_2]$ is of type II$^*$, then $z=x$ as $L'$ and $L^*$ are not collinear, but then no member of $[J_1,J_2]\setminus\{L^*\}$ is adjacent to $J$, a contradiction. If $[J_1,J_2]$ is of type III$^*$, then all points of $\<L \cap L^*, z\>$ are collinear to $L$ so each line corresponding to a member of $[J_1,J_2]$ contains a point collinear to $L$, implying that no member of $[J_1,J_2] \setminus \{J^*\}$ is adjacent to $J$, a contradiction.

\item \textit{Suppose there are no quadruples of type II$^*$.}
Then $[J_1,J_2]$ is of type III$^*$. If $J \perp J^*$ then $J\sim J^*$ and hence $J^*\notin\{J',J''\}$. As above, $L'$ contains a point collinear to $L$ and we obtain that $J$ is not adjacent to any member of $[J_1,J_2]\setminus\{J^*\}$, a contradiction.
If $J$ and $J^*$ are not collinear, then $J \nsim J^*$, so indeed, $J^*\in\{J',J''\}$.\end{itemize}

 \textbf{\boldmath Claim 3: For each pair $J,J_1$ of $j$-spaces intersecting each other in a $(j-2)$-space, there is a near-line containing $J_1$ and not containing $J$ such that $J$ is adjacent to all its members except one or two.\unboldmath} 

We first look for a near-line, and then show that $J$ is adjacent to all but one or two of its members.
Consider $\Res_\Delta(J \cap J_1)$. If $L$ and $L_1$ are not opposite, we take a line $L_2$ opposite both of them. If $L$ and $L_1$ are opposite, take a plane $\pi$ through $L$ and note that $\pi$ is semi-opposite $L_1$. Then $\pi$ contains a point $p \notin L \cup L_1^\perp$. Through $p$, we can then find a line $L_2\nsubseteq \pi$ opposite $L_1$ by taking a point in $\Res_\Delta(\<J \cap J_1, p\>)$ opposite the point corresponding to $L_1$ and avoiding the line corresponding to $\pi$. All of this is possible by Fact~\ref{lem4a}$(ii)$ and~$(iii)$. 

Let $J_2$ be the $j$-space through $J \cap J_1$ corresponding to $L_2$. Then $[J_1,J_2]$ is a near-line of type III$^*$. In the first case, $J$ is adjacent to $J_2$ and not adjacent to $J_1$, in the second case, $J$ is adjacent to $J_1$ and not adjacent to $J_2$. In both cases, we show that there is at most one member of $[J_1,J_2]\setminus\{J_1,J_2\}$ not adjacent to $J$. Suppose $J$ would not be adjacent to a third member $J_3  \in [J_1,J_2]$. Then $L_3$ would be collinear with a point $z \in L$. Now $L$ also contains a point $z'$ collinear with $L_c$ ($c$ equals $1$ or $2$, depending on the case we are in). If $z=z'$, then $z$ is collinear to all lines of the regulus determined by $L_1$ and $L_2$, contradicting the fact that $L$ is opposite $L_{c'}$ (with $\{c,c'\}=\{1,2\}$). If $z \neq z'$, then there are exactly two points (namely, those on $L_c$ and $L_3$) of the hyperbolic $3$-space spanned by $L_1$ and $L_2$ that are collinear with $L$ (this is easily verified when $\Delta$ is embeddable since it has to be orthogonal, and also holds true if $\Delta$ is not embeddable). This implies that $J$ is collinear to all members of $[J_1,J_2]\setminus\{J_c,J_3\}$.
 
The lemma is proven.
\end{proof}

\subsubsection{The $(k,\ell)$-Weyl graphs: $|j|<n-1$ and $b= -1$} 
 
 Now, the quadruples are of type I or V$(t)$ (cf.\ Lemma~\ref{b-1}) and  the following lemma enables us recognise the type I quadruples, by which means we obtain~$\mathsf{G}_j$. 

\begin{lemma} A quadruple  $\{J_1,J_2,J_3,J_4\}$ is of type V$(t)$ (with $1 \leq t \leq j-k-1$) if and only if there is a $j$-space $J^* \neq J_4$ such that $\{J_1,J_2,J_3,J^*\}$ is a quadruple, whereas $\{J_1,J_2,J_4,J^*\}$ is not. \end{lemma} 
 
\begin{proof} 
Suppose the quadruple is of type V$(t)$ and let $S$ and $L$ be as described in the definition. Let $J^*$ be a $j$-space through $S$ which is collinear with $L$ such that $J^* \cap L \notin J_1 \cup J_2 \cup J_3 \cup J_4$ and $J^*$ and $J_4$ correspond to the same $(t-1)$-space in $\Res_{\Delta}(\<S,L\>)$. Clearly, $\{J_1,J_2,J_3,J^*\}$ is still a quadruple, as opposed to $\{J_1,J_2,J_4,J^*\}$. Next, suppose the quadruple is of type I. If $\{J_1,J_2,J_3,J^*\}$ is a quadruple, then $J^*$ contains $J_1\cap J_2$ and is contained in $\<J_1,J_2\>$. As $\{J_1,J_2,J_3,J^*\}$ is a quadruple, $J^* \notin \{J_1,J_2,J_3\}$, it follows $\{J_1,J_2,J_4,J^*\}$ is a quadruple too. 
\end{proof}

\subsubsection{The $k_\geq$-intersection graph: $|j|<n-1$} 
 
In this case, there are triples of types I, II and III$^*$ (since the triples of type IV are the same as those of type I).  
\begin{lemma}\label{lem1a} 
Suppose $\{J_1,J_2,J_3\}$ and $\{J_1,J_2,J_4\}$ are triples, while $\{J_1,J_3,J_4\}$ is not. Then $\dim (J_3 \cap J_4) = j-1$ and $J_3$ and $J_4$ are contained in a singular subspace. Moreover, for $j$-spaces $J_3$ and $J_4$ with $\dim(J_3 \cap J_4) =j-1$ and $J_3 \perp J_4$ we can find $j$-spaces $J_1$ and $J_2$ such that $\{J_1,J_2,J_3\}$ and $\{J_1,J_2,J_4\}$ are triples whereas $\{J_1,J_3,J_4\}$ is not.
\end{lemma} 
\begin{proof} 
Note that all $3$-tuples in a near-line of type I or III$^*$ need to be triples themselves (of the same type, in contrast to $3$-tuples occurring in a near-line of type II). This observation shows that the near-line $[J_1,J_2]$ is of type II. As $\{J_1,J_3,J_4\}$ is not a triple, $J_3$ and $J_4$ are collinear. This shows the first part of the lemma. For the second part, consider $\Res_\Delta(J_3 \cap J_4)$, in which $J_3$ and $J_4$ correspond to points $p_3$ and $p_4$ on a line $L$. Let $p$ be a third point on this line, and take two non-collinear lines $L_1$ and $L_2$ through $p$ which both are non-collinear with $L$. Points $p_c \in L_c \setminus \{p\}$ ($c=1,2$) then correspond to $j$-spaces $J_1$ and $J_2$ satisfying our needs.
\end{proof} 

Consequently, we can deduce $\mathsf{G}_j$ from $\Gamma'$. 
 
\subsubsection{The $k$-intersection graph: $|j|<n-1$ }\label{pm3} 
This time, the quadruples are of types I, II, III$^*$ and IV. The presence of quadruples of type IV implies that $j$-spaces intersecting each other in a $(j-1)$-space fit into two types of quadruples. On the other hand, two $j$-spaces intersecting each other in a $(j-2)$-space only fit in a type III$^*$ quadruple. This is the idea behind the following lemma.

\begin{lemma}\label{lem3a}  
Let $J_1$ and $J_2$ be adjacent vertices in $\Gamma'$. They are contained in a quadruple of type III$^*$ if and only if each $4$-tuple in $[J_1,J_2]$ is a quadruple too. 
\end{lemma} 
\begin{proof} Suppose $J_1$ and $J_2$ are contained in a quadruple of type III$^*$. Then it is clear that all $4$-tuples in $[J_1,J_2]$ are quadruples (of type III$^*$) themselves. 

Conversely, suppose that $J_1$ and $J_2$ are contained in a quadruple of type I,II or IV.  Then $\dim(J_1 \cap J_2)=j-1$ and we can always find $J_3$ and $J_4$ such that $\{J_1,J_2,J_3,J_4\}$ is a quadruple of type IV.  We continue in $\Res_\Delta(J_1 \cap J_2)$. There are two cases, depending on whether or not $p_1$ and $p_2$ are collinear. 
\begin{compactenum}[$-$]
\item \textit{Suppose first that $p_1 \perp p_2$ and $p_3 \in p_1p_2$.} Then $p_4$ is collinear to a point $p'_4$ on $p_1p_2$, distinct from $p_1, p_2$ and $p_3$. Let $J'_4$ be the $j$-space through $J_1 \cap J_2$ that corresponds to $p'_4$. As $\{J_1,J_2,J_3,J'_4\}$ is also a quadruple, $J'_4$ belongs to $[J_1,J_2]$. But then $\{J_1,J_2,J_4,J'_4\}$ is a $4$-tuple of $[J_1,J_2]$ which is not a quadruple. 
\item \textit{Next, suppose $p_1$ is opposite $p_2$}. We may moreover assume that $p_2$, $p_3$ and $p_4$ are on one line. Then $p_1$ is collinear with a unique point $p'_1$ on this line, distinct from $p_2$, $p_3$ and $p_4$.  Let $p_5$ and $p_6$ be two distinct points in $p_1p'_1\setminus\{p_1,p'_1\}$. Then $p_2$ is not collinear with $p_5$ nor with $p_3$. If $J_e$ are the $j$-spaces through $J_1 \cap J_2$ corresponding to $p_e$, for $e =5,6$, then $\{J_1,J_2,J_5,J_6\}$ are also a quadruple, and hence $J_5$ and $J_6$ belong to $[J_1,J_2]$. Yet, $\{J_3,J_4,J_5,J_6\}$ is not a quadruple.
\end{compactenum}
In both cases we found a $4$-tuple in $[J_1,J_2]$ which is not a quadruple, which proves the lemma.
\end{proof} 

Hence we can remove the edges in $\Gamma'$ between the $j$-spaces that are contained in a quadruple of type III$^*$. We obtain $\mathsf{G}'_j$. 
\par\bigskip
In all cases we were able to deduce $\mathsf{G}_j$ or $\mathsf{G}'_j$ from $\Gamma'$. This finishes the proofs of Main Theorems~\ref{main1} and~\ref{main2} in case $k \neq -1$. In the next section, we handle the case where $k=-1$.
 
\section{The $(-1,\ell)$-Weyl graph}\label{k=-1}
We may assume that $\Gamma=\Gamma_{-1}^\ell$, since $\Gamma_{\geq -1}$ is a complete bipartite graph and $\Gamma_{-1}$ is the bipartite complement of $\Gamma_{\geq 0}$. We try to apply the same strategy as before though some cases require an alternative approach or lead to the counter examples described in cases $(i)$ and $(ii)$ of Main Theorem~\ref{main1}. 

Again, let $\{J_1,J_2,J_3,J_4\}$ be a quadruple. Note that $J_3 \neq J_4$ now as we are working with Weyl-graphs only. Note also, that now it is possible that $|j|=0$, which is not always useful, because $\mathsf{G}'_0$ is complete graph and hence if we obtain this one, this does not help. When $|j|=0$, this implies that we are in one of the following cases:
\begin{enumerate}[$(i)$]
\item $|i|=|j|=0$, $k=-1$, $a=0$ and $b=-1$:  two vertices (i.e., points) are adjacent whenever they are distinct but collinear. 
\item 
$|i|=|j|=0$, $k=-1$, $a=-1$ and $b=0$: two vertices (i.e., points) are adjacent whenever they are opposite. This has been dealt with in \cite{Kas-Mal:13}.
\item $|i|=|\ell|>0$, $k=-1$, $a=i-1$, $b=0$: An adjacent pair $(I,J)$ consists of a point $J$ and an $i$-space $I$ with $J \notin I^\perp$. We deal with this in subsection~\ref{semiopposite}.
\end{enumerate}

\subsection{Adjacent vertices are in a general position ($a,b\geq0$)}
Assume first that $a,b \geq 0$. The following lemma is a weaker version of Lemma~\ref{lem7a}.

\begin{lemma}\label{-1, gemdoorsnede}  Each point contained in two members of a quadruple is contained in a third member. By renumbering if necessary,  $J_1 \cap J_2 = J_2 \cap J_3 = J_3 \cap J_1$. \end{lemma}
\begin{proof} We show that, for any permutation of $\{1,2,3,4\}$, $J_3 \cap J_4 \subseteq J_1 \cup J_2$. So suppose for a contradiction that a point $p\in J_3 \cap J_4$ is not contained in $J_1 \cup J_2$. We claim that there is an  $I_p:=\<A_1,A_2,B_1,B_2\> \in \mathsf{N}_{(-1)}(J_1,J_2)$  containing $p$. To this end we apply Construction~\ref{con}. Note that there is no need to avoid any subspace now but $J_1$ and $J_2$ themselves, which makes things easier. There are three cases. Firstly, if $p \in J_1^\perp \cap J_2^\perp$, we can apply Construction~\ref{con} such that $p \in A_1=A_2$. Secondly, if $p \in J_1^\perp\setminus J_2^\perp$, we take any $a$-space $A$ in $(J_1^\perp \cap J_2^\perp)\setminus(J_1 \cup J_2)$, which certainly contains a hyperplane $A^-$ collinear with $p$, so we can put $A_1=\<A^-,p\>$. We still need a point $q \in (A_1^\perp \cup J_2^\perp)\setminus J_1^\perp$. We look for this point in $\Res_\Delta(\<A^-,\proj_{J_2}(p)\>)$. In here, $J_1$ corresponds to a subspace $p_1$ of dimension at least $0$ (since $\proj_{J_2}(p)$ is collinear with at least a point of $J_1 \setminus J_2$) and $J_2$ to a point $p_2$, the point $p$ corresponds to a point that we keep denoting by $p$. Then there is a point $q$ in $p^\perp \cap p_2^\perp$ which is opposite $p_1$ (note that by now we are looking in a polar space of rank $n-|\ell|-1\geq 1$ since $|\ell|<n-1$ if $|j|<n-1$). In $\Delta$, a point in the subspace corresponding to $q$ and disjoint from $J_2$ satisfies our needs. We can now select $B$ in the standard way. Finally, suppose $p\notin J_1^\perp \cup J_2^\perp$. As above we can find an $a$-space $A$ collinear with $p$ and then we only need to select $B$ such that it contains $p$. The claim is proven.

As before, this leads to a contradiction, since $I_p$ not adjacent to $J_3$ nor to $J_4$. Therefore, the intersection of two members of the quadruple is contained in a third member. If we start with a pair having maximal dimension of intersection, the lemma follows. \end{proof}


There is also a weaker version of $\mathsf{(RU2)}$, that holds whenever $0=\min\{a,b\}$.

\begin{itemize} 
\item[$\mathsf{(RU2')}$]  Let $L$ be a line containing distinct points $p$ and $p'$ such that $p \in J_u^\perp\setminus J_u$ and $p' \in J_r^\perp\setminus J_r$ for $u \neq r$. Then $L$ contains a point $q$ with $q \in J_v^\perp \cup J_w^\perp$, where $\{u,r,v,w\}=\{1,2,3,4\}$.
\end{itemize} 

\begin{lemma}\label{ru2'} If $\Gamma$ equals $\Gamma_{-1}^{\ell}$ and $0 \in \{a,b\}$, then $\mathsf{(RU2')}$ is valid for any quadruple. Moreover, $\mathsf{(RU2')}$ remains valid in $\Res_{\Delta}(S')$ for $S' \subseteq J_1 \cap J_2 \cap J_3 \cap J_4$. \end{lemma} 

\begin{proof} 
Let $L=pp'$ be a line with $p\in J_1^\perp\setminus J_1$ and $p' \in J_2^\perp \setminus J_2$ and suppose for a contradiction that none of its points is contained in $J_3^\perp \cup J_4^\perp$. By $\mathsf{(RU1)}$, $p \notin J_2^\perp$ as then $p \in J_3^\perp \cup J_4^\perp$; likewise, $p' \notin J_1^\perp$. 

\textbf{Suppose first that $b=0$.} Let $A^-$ be an $(a-1)$-space collinear with $\<J_1,p\>$ and $\<J_2,p'\>$ (cf.\ Fact~\ref{lem4a}$(i)^*$), again, there is no need to avoid $J_3$ and $J_4$. Then we take $I=\<L,A^-\> \in \mathsf{N}_{(-1)}(J_1,J_2)$ such that $A_1=\<p,A^-\>$ and $A_2=\<p',A^-\>$. As we may assume $I \sim J_3$, there has to be an $a$-space $A_3$ in $I$ collinear with $J_3$, and hence $A_3$ intersects $L$ in at least a point, i.e.\ $L$ contains a point of $J_3^\perp$ after all, a contradiction.

\textbf{Next, suppose $a=0$.} By the previous case we may assume that $b > 0$. In $\Res_\Delta(L)$, $J_1$ and $J_2$ correspond to  $(j-1)$-spaces $J'_1$ and $J'_2$ and $J_3$ and $J_4$ to $(j-2)$-spaces $J'_3$ and $J'_4$. If $b-1 \leq j-2$ (recall $b-1\geq 0$), we take a $(b-1)$-space $B'_x$ in $J'_x$ for each $x\in \{1,2,3,4\}$. Fact~\ref{lem4a}$(ii)$ then says that there is a $(b-1)$-space $B^-$ which is opposite all of them, i.e., such that no point of $B^-$ belongs to  ${J'_1}^\perp \cup {J'_2}^\perp \cup {J'_3}^\perp \cup {J'_4}^\perp$. The corresponding $i$-space $\<L,B^-\>$ in $\Delta$ is adjacent to $J_1$ and $J_2$ (since $J_1^\perp \cap I=\{p\}$ and $J_2^
\perp \cap I =\{p'\}$) and hence $I$ needs to be adjacent with one of $J_3$, $J_4$ too. If $I \sim J_3$, then by our choice of $B^-$ in the residue, $J_3$ is collinear with a point of $L$; likewise if $I\sim J_4$. This violates our assumptions.

If $b=j \geq 1$, then $i=j+a+1=b+1$.  If $b \geq 2$, we take $(b-2)$-spaces $B'_x \subseteq J'_x$ for each $x\in\{1,2,3,4\}$ and let $B^-$ be a $(b-2)$-space opposite each of these. If $b =1$, we put $B^-=\emptyset$. In both cases, $I^-:=\<L,B^-\>$ is a $b$-space opposite $J_3$ and $J_4$ containing unique points ($p$ and $p'$, respectively) collinear with $J_1$ and $J_2$. Then $J_1$ and $J_2$ also contain unique respective points $p_1$ and $p_2$ collinear with $I^-$. We claim that there is a $(b+1)$-space $I$ (recall $b+1=i$) through $I^-$  collinear with $p_1$ and $p_2$. If so, then $I \in \mathsf{N}_{(-1)}(J_3,J_4)$ because $I \cap J_3 = I \cap J_4 = \emptyset$, since no point of $J_3 \cup J_4$ is collinear with $L$; and $I \notin \mathsf{N}(J_1) \cup \mathsf{N}(J_2)$ since $J_1$ and $J_2$ both contain a point collinear with $I$. This is a contradiction. 

Note that we may assume that $n \geq b+3$, for otherwise $\Gamma_{b+1,b;-1,b+1}^n$ is isomorphic to the bipartite complement of $\Gamma_{b+1,b;\geq 0}^n$. 
Firstly, let $p_1$ and $p_2$ be non-collinear points. Then $I^-$ is a non-maximal singular subspace in  $p_1^\perp \cap p_2^\perp$, from which it follows that $I$ exists.
Secondly, let $p_1$ and $p_2$ be collinear. Then $\<p_1,p_2,I^-\>$ is contained in a singular $(b+2)$-space $D$, which intersects $J_1$ in precisely $p_1$ since each $(b-1)$-space in $J_1$ complimentary to $p_1$ is opposite $\<B^-,p'\>$, likewise, $D \cap J_2 = \{p_2\}$. Now any $(b+1)$-space $I$ in $D$ through $\<p,p',B^-\>$ satisfies our needs. The claim is proven and as mentioned above, this leads to a contradiction.

This works for all permutations of $\{1,2,3,4\}$. The fact that $\mathsf{(RU2')}$ is a residual property is again easily verified.
\end{proof}

\begin{lemma}\label{AenBlijn} If $a,b \geq 0$, each quadruple is of type I, II$^*$ or III$^*$. Moreover, if $0\in\{a,b\}$, then there are no quadruples of type III$^*$.\end{lemma}
\begin{proof}
By Lemma~\ref{-1, gemdoorsnede}, we may assume that $J_1,J_2,J_3$ intersect each other in one common subspace $S$ and that $J_4$ intersects them in a subspace $S'$ of $S$. By Lemma~\ref{gemprojectie} and $a \geq 0$, we may write $U_y$ instead of $U^x_y$ for $x,y \in \{1,2,3,4\}$ with $x\neq y$. Put $u=\cod_{\<S,U_1\>}(S) = \cod_{\<S,U_2\>}(S) = \cod_{\<S,U_3\>}(S)$ and $u' = \cod_{\<S',U_4\>}(S')$. Since $S' \subseteq S$, we have $u \leq u'$. Observe that $S=S'$ precisely when $u=u'$.

\textbf{Case 1: \boldmath $u'=u=-1$.\unboldmath} In $\Res_{\Delta}(S)$, we obtain four opposite subspaces $T_1,T_2,T_3$ and $T_4$. By $\mathsf{(RU1)}$ and Lemma~\ref{opp}, they form a hyperbolic $(2t+1)$-space for $t=j-s-1$. There are two cases.
\begin{itemize}
\item \textit{Claim 1: If $a,b \geq 1$, the quadruple is of type II$^*$ or III$^*$.} Assume for a contradiction that $t>1$ and take any point $p$ in  $T_1$ and any line $L$ in $\proj_{T_2}(p)$. If every line joining $p$ and a point of $L$ intersects $T_3 \cup T_4$, then the plane $\<p,L\>$ meets one of $T_3,T_4$ in a line $L'$. But then $L$ and $L'$ have to intersect, a contradiction. We conclude that there is a line $M$ through $p$ meeting $L$ which is disjoint from $T_3 \cup T_4$. Since the $t$-spaces form a hyperbolic $2t+1$-space, they correspond to maximal singular subspaces of a polar space of rank $t+1$ (that also contains $M$). It follows that no point of $M$ can be collinear with $J_3$ or $J_4$. Take an $a$-space $A \subseteq (J_3^\perp \cap J_4^\perp)\setminus(J_3 \cup J_4)$. As $T_1, T_2, T_3$ and $T_4$ are in a hyperbolic subspace, it follows that $A$ is also collinear with $T_1$ and $T_2$ (so in particular, $A \perp M$) and is not contained in $\<T_3,T_4\>$ (so $\<A,M\> \cap \<T_3,T_4\> = M$ and hence $\<A,M\> \cap (J_3 \cup J_4)$ is empty). In $\Res_\Delta(\<A,M\>)$,  we take a $(b-2)$-space $B^-$ opposite the subspaces corresponding to $J_3$ and $J_4$. Let $I$ be the subspace in $\Delta$ corresponding to $B^-$, i.e.,  $I=\<A,M,B^-\>$. Then $I \in \mathsf{N}_{(-1)}(J_3,J_4)$, but as $I$ intersects $J_1$ and $J_2$ non-trivially, $I \notin \mathsf{N}(J_1) \cup \mathsf{N}(J_2)$. This contradiction implies that $t \leq 1$. If $t=1$, one shows, similarly as in the proof of Lemma~\ref{type1234}, that any line joining two collinear points of $T_1$ and $T_2$ will also intersect $T_3$ and $T_4$, so the quadruple is of type III$^*$. If $t=0$, the quadruple is of type II$^*$ since the $t$-spaces are on a hyperbolic line, as mentioned above.

\item \textit{Claim 2: If $0\in\{a,b\}$, the quadruple is of type II$^*$.} 
In this case, $\mathsf{(RU2)'}$ holds and similarly as in the proof of Lemma~\ref{aempty}, we can show that there are no quadruples of type III$^*$.

\end{itemize}

\textbf{\boldmath Case 2: $u' \geq 0$.\unboldmath} Let $p$ be a point of $U_4$ and $q$ an arbitrary point of $J_1\setminus S$.  If $q \notin U_1$, we claim that the line $pq$ has to intersect a third member of the quadruple. Suppose the contrary. Let $A^*$ be an $a$-space which is collinear with $\<J_2,p\>$ and with $\<J_3,p\>$ (by Fact~\ref{lem4a}$(i)^*$, this is possible as $|\ell|<n-1$ as $a \geq 0$ implies $\max\{|i|,|j|\}<n-1$). Let $A^-$ be an $(a-1)$-space of $A^*$ collinear with $q$. Next, let $B^-$ be a $(b-1)$-space chosen in $\Res_{\Delta}(\<A^-,p,q\>)$ semi-opposite the subspaces corresponding to $J_2$ and $J_3$. Then $I=\<A^-,p,B^-,q\> \in \mathsf{N}_{(-1)}(J_2,J_3)$. However, $I$ cannot be adjacent to $J_1$ or $J_4$ as it contains a point of both of them (the points $q$ and $p$, respectively).  This contradiction to the definition of a quadruple implies that $pq$ intersects a third member of the quadruple after all. If we vary $q$ over $J_1 \setminus(S\cup U_1)$ and as all lines $pq$ have to intersect $J_2$ or $J_3$, it follows just as in the proof of Lemma~\ref{lem0} that each line in $J_1$ which is not contained in $S \cup U_1$ has to intersect $S$. But then  either $S \cup U_1 =J_1$ or $U_1$ is empty and $\dim(S)=j-1$. But the latter case does not occur, for another consequence of Lemma~\ref{lem0} implies that $J_1$ has to be collinear with at least one of $J_2,J_3$, which would violate the fact that $U_1$ is empty. Hence $S \cup U_1 =J_1$ and we conclude that $J_1,J_2,J_3$ and $J_4$ are contained in a singular subspace.  

If $a\geq 1$, then the previous arguments also apply if $q \in U_1$, as we can choose $A$ such that it contains $pq$. Like before, it follows from Lemma~\ref{lem0} that the quadruple is of type I. If $a=0$, we can also show that the quadruple is of type I, by proceeding as in Case 1 of Lemma~\ref{aempty} (with $\mathsf{(RU2)'}$ instead of $\mathsf{(RU2)}$). 

It is easily verified that each of those $4$-tuples obtained above indeed satisfies the definition of a round-up quadruple.
\end{proof}

\par\bigskip
As the quadruples are the same as before, we can continue in the same way as in the previous section to conclude the proof in this case.
\subsection{Adjacent vertices are semi-opposite ($a=-1$ and case $(iii)$)} \label{k=a=-1}\label{semiopposite} 
Our convention on $i$ and $j$ implies that if $a=-1$, then $|\ell|=|i|=|j|$ (see also Subsection~\ref{Aleeg}). Then two adjacent vertices of $\Gamma$ correspond to opposite subspaces. The opposition case, or at least its non-bipartite version, has been dealt with in \cite{Kas-Mal:13}. The same techniques apply and hence we limit ourselves to summarising their approach: In this particular case, one can work with ``reverse'' round-up triples, i.e., a round-up triple of the complement of $\Gamma$: a (round-up) triple consists of three vertices ($j$-spaces) such that each vertex is either adjacent with one or all of them. After classification we obtain (with our own notation) when $|j|<n-1$ either a triple of type I or of type II$^*$ and when $|j|=n-1$, either a triple of type II or one of type III. As before, we can look at the near-lines. These can be distinguished from each other, leading us to the Grassmannian just like before, except in the two following cases.

\begin{itemize}
\item Let $\Delta$ be a symplectic polar space and $i=j=0$: then a near-line of type I corresponds in $\Res_\Delta(S)$ ($S$ still denotes the common intersection of the triple) to an ordinary line and a near-line of type II$^*$ to a hyperbolic line. Both are lines of the projective space in which $\Delta$ naturally embeds. They behave in the same way and hence cannot be separated.
\item Let $\Delta$ be a parabolic polar space and $i=j=n-1$. Also here, a near-line of type II now behaves in the same way as a near-line of type III.
\end{itemize}

So only in those two cases there are extra automorphisms, and those are explained in detail in Examples~\ref{special1} and~\ref{special2}.

\par\bigskip
Next, suppose that we are in case $(iii)$, i.e., $0=|j|<|i|$. If $|i|=n-1$, then the bipartite complement $\overline{\Gamma}$ of $\Gamma$ is precisely $\mathsf{C}_{0,n-1}(\Delta)$, and we can refer to Proposition~\ref{propk=i}. Hence, we may assume that $|i|<n-1$. In $\overline{\Gamma}$, a point $p$ and a $i$-space $I$ are adjacent when $p \in I^{\perp}$. We continue to work in $\overline{\Gamma}$. With the following lemma we can construct $\mathsf{G}'_i$, which as what we needed to prove.

\begin{lemma} The intersection of two  $i$-spaces $I,I'$ has dimension $i-1$ if and only if there is no $i$-space $I^* \neq I$ such that $\mathsf{N}_{\overline{\Gamma}}(I,I') \subsetneq \mathsf{N}_{\overline{\Gamma}}(I,I^*)$. \end{lemma}
\begin{proof}
Let that $I$ and $I'$ be $i$-spaces such that $\dim(I \cap I')< i-1$. In $\Res_{\Delta}(I \cap I')$, $I$ and $I'$ correspond to $v$-spaces $V,V'$ with $v\geq 1$. Take any $v$-space $V^*$, corresponding to a $i$-space $I^*$ in $\Delta$ through $I\cap I'$, such that $\dim(V \cap V^*)=v-1$ and $\dim(V'\cap V^*)=0$. Clearly, $\mathsf{N}_{\overline{\Gamma}}(I,I') \subseteq \mathsf{N}_{\overline{\Gamma}}(I,I^*)$. Suppose for a contradiction that $\mathsf{N}_{\overline{\Gamma}}(I,I') =\mathsf{N}_{\overline{\Gamma}}(I,I^*)$. Consequently, $V \cap V^*$ is collinear with $V'$. Hence, in $\Res_{\Delta}(I \cap I^*)$, where $I$ and $I^*$ correspond to points $q,q^*$,  corresponding to $V'$ is a subspace $D$ strictly containing $q^*$. It is easy to see that there is a point in $q^\perp \cap {q^*}^\perp$ not collinear with $V'$. We conclude that $\mathsf{N}_{\overline{\Gamma}}(I,I') \subsetneq \mathsf{N}_{\overline{\Gamma}}(I,I^*)$. Note that $\dim(I\cap I^*)=i-1$.
 
For the converse, let $I$ and $I'$ be $i$-spaces with $\dim(I \cap I')= i-1$. Suppose for a contradiction that there is an $i$-space $I^*$ such that $\mathsf{N}_{\overline{\Gamma}}(I,I')  \subsetneq \mathsf{N}_{\overline{\Gamma}}(I,I^*)$. If $\dim(I \cap I^*) <i-1$, then the preceding paragraph yields an $i$-space $I^{**}$ with $\dim(I \cap I^{**})=i-1$  such that $\mathsf{N}_\Gamma(I,I^*) \subsetneq \mathsf{N}_\Gamma(I,I^{**})$. But then $\mathsf{N}_{\overline{\Gamma}}(I,I')  \subsetneq \mathsf{N}_{\overline{\Gamma}}(I,I^*) \subsetneq \mathsf{N}_\Gamma(I,I^{**})$, so by replacing $I^*$ by $I^{**}$ if necessary, we may assume that $\dim(I\cap I^*)= i-1$. Then $\dim(I \cap I' \cap I^*) \in \{i-2,i-1\}$. Taking into account that $\mathsf{N}_{\overline{\Gamma}}(I,I') \subsetneq \mathsf{N}_{\overline{\Gamma}}(I,I^*)$,  it is easily deduced that the lines/points corresponding to the $i$-spaces in $\Res_{\Delta} (I \cap I' \cap I^*)$  have to be in a plane/on a (hyperbolic) line. But then one deduces that $\mathsf{N}_{\overline{\Gamma}}(I,I') = \mathsf{N}_{\overline{\Gamma}}(I,I^*)$, a contradiction.
\end{proof}

\subsection{Adjacent vertices are contained in a singular subspace ($b=-1$)}
In this case, adjacent vertices $I$ and $J$ are disjoint subspaces spanning a singular subspace, implying $|i|\leq |j| <n-1$. We will be using triples instead of quadruples and a new type of triple will turn up.

Suppose $J_1,J_2,J_3$ are $j$-spaces intersecting each other in a common subspace $S$. With the same notation as before, we say that they form a triple of type VI$(t)$, with $t$ an integer such that $0 \leq t \leq j$, if the following condition is satisfied. 
\begin{itemize}
\item[$\mathsf{VI(t)}\;$] $J'_1,J'_2$ and $J'_3$ are $t$-spaces in $\Delta'$ generating a hyperbolic $(2t+1)$-space.
\end{itemize}

Note that a triple of type VI$(0)$ is the same as a triple of type II$^*$, but a triple of type VI$(1)$ is in general not the same as a triple of type III$^*$ for the lines in the hyperbolic $3$-space do not necessarily lie on a regulus of a hyperbolic quadric. Hence a triple of type VI$(0)$ occurs when $\Delta$ is not a strictly orthogonal polar space, and a triple of type VI$(t)$ with $t>0$ occurs in every kind of polar space.

\begin{lemma} A triple $\{J_1,J_2,J_3\}$ is of type VI$(t)$. \end{lemma}
\begin{proof}  We claim that each point contained in precisely one member of the triple is not collinear with any of the two other members of the triple. So assume for a contradiction  that $p \in J_1 \setminus (J_2 \cup J_3)$ is collinear with $J_2$. By Lemma~\ref{ru1}, we know that $\mathsf{(RU1)}$ holds, so $p$ is also collinear with $J_3$. But then, as $p \notin J_2 \cup J_3$ there is an $i$-space $I=A \in \mathsf{N}_{(-1)}(J_2,J_3)$ with $p \in A$, which cannot be adjacent to $J_1$. This contradiction shows the claim.

Denote by $S_{xy}$ the intersection $J_x \cap J_y$ for $x,y \in \{1,2,3\}$ and $x \neq y$ and suppose that these do not coincide. We may assume that $S_{12} \cap S_{23} \cap S_{31}$ is empty, as otherwise we look at its residue.

By $\mathsf{(RU1)}$, each point of $S_{xy}$ is collinear with $J_z$, with $\{x,y,z\}=\{1,2,3\}$. Now suppose that $S_{12}$ and $S_{13}$ are nonempty, and consider a line $L$ intersecting both of them, necessarily in distinct points. These two points are necessarily collinear with $J_2$ and $J_3$. Hence, each point of $L$ has to be collinear with $J_2$ and $J_3$. However, $L$ contains a point which is contained in $J_1$ only, contradicting the first paragraph of this proof. This implies that at most one of the intersections, say $S_{12}$ can be nonempty. But as $S_{12}$ is collinear with $J_3$, the latter contains a point which is collinear with $J_1$, again a contradiction. 

We conclude that the $j$-spaces have one common intersection $S$ and, in $\Res_\Delta(S)$, they correspond to $t$-spaces which are on a hyperbolic $(2t+1)$-space, as required. It is easily verified that each of those $3$-tuples obtained above indeed satisfies the definition of a round-up triple.
\end{proof}
We now distinguish either $t=0$ or $t=1$ from the others, depending on the type of $\Delta$. 

\begin{lemma}\label{dist} 
Suppose $\{J_1,J_2,J_3\}$ is triple of type VI$(t)$. If no $j$-space $J_4$ (with $J_4 \neq J_3$) is such that $\{J_1,J_2,J_4\}$ is a triple whereas $\{J_1,J_3,J_4\}$ is not, then precisely one of the following occurs.

\begin{itemize}
\item[$(i)$] $t=0$ and $\Delta$ contains hyperbolic lines.
\item[$(ii)$] $t=1$ and $\Delta$ contains hyperbolic quadrangles as hyperbolic $3$-spaces.
\end{itemize}

\end{lemma}
\begin{proof}
In $\Delta'$, the $j$-spaces of the triple correspond to $t$-spaces $T_1,T_2$ and $T_3$. Denote the hyperbolic $(2t+1)$-space generated by them by $H$. Suppose there is a $t$-space $T_4$ in $H$ opposite $T_1, T_2$ such that $T_3\cap T_4\neq \emptyset$. Then clearly, $\{J_1,J_3,J_4\}$ cannot be a triple. If $t > 1$, such a $t$-space can always be found.
\begin{itemize}
\item
If $t=1$, such a $t$-space can always be found, except when $\Delta$ is a strictly orthogonal polar space, as in that case $H$ is a hyperbolic quadrangle (the strictly orthogonal polar spaces are the only ones in which a a hyperbolic $3$-space consists of precisely a regulus).
\item
If $t=0$, then $H$ is a polar space of rank $1$ and hence such a $t$-space can never be found. Note that, in this case, $H$ only contains more than two points (i.e., $H$ is a hyperbolic line) if $\Delta$ is not a strictly orthogonal polar space.
\end{itemize}
So we see that either $t=0$ or $t=1$ and only one of these possibilities occurs, depending on $\Delta$. It is easily seen that, in cases $(i)$ and $(ii)$, every $j$-space $J_4$ such that $\{J_1,J_2,J_4\}$ is a triple is also such that $\{J_1,J_3,J_4\}$ is a triple. \end{proof}

Since each polar space either contains hyperbolic lines or is strictly orthogonal, and no strictly orthogonal polar space contains hyperbolic lines, we can either recognise the triples of type VI$(0)$ or those of VI$(1)$. This gives us the following two cases to consider.

\textbf{\boldmath Case $(i)$: $\Delta$ contains hyperbolic lines.\unboldmath} The previous lemma enables us to reduce $\Gamma'$ by restricting its adjacency relation to being contained in a triple of type VI$(0)$. We obtain the (non-bipartite) graph $\Gamma_{j;j-1,j}(\Delta)$ and the result follows from \cite{Kas-Mal:13}.

\textbf{\boldmath Case $(ii)$: $\Delta$ contains hyperbolic quadrangles as hyperbolic $3$-spaces.\unboldmath} In this case $\Delta$ is a strictly orthogonal polar space. Again with the help of the previous lemma, we reduce $\Gamma'$ by restricting the adjacency relation to being contained in a triple of type VI$(1)$ (which is in fact the same as a triple of type III$^*$ in this kind of polar space). We obtain the non-bipartite Weyl-graph $\Gamma_{j;j-2,j}^n(\Delta)$, and the result follows from Section~\ref{k>-1} if $j >1$ and from \cite{Kas-Mal:13} if $j=1$ (see also Subsection~\ref{k=a=-1}).

So also in this case we obtain that each automorphism of $\Gamma$ is induced by one of $\Delta^b$. 
\par\bigskip

As we went trough all cases,  we have reached the end of the proofs of Main Theorems~\ref{main1} and~\ref{main2}.
\hfill $\blacksquare$
\par\bigskip
\emph{We thank the referee for the opportunity to improve the paper. Many arguments have been corrected and/or clarified.}

\end{document}